\documentclass[11pt,letterpaper]{article}
\usepackage[margin=0.9in]{geometry}
\usepackage[utf8]{inputenc} % allow utf-8 input
\usepackage[T1]{fontenc}    % use 8-bit T1 fonts
\usepackage{dsfont}
\usepackage[maxfloats=100]{morefloats}[2015/07/22]% v1.0h
\usepackage{multirow}
\usepackage{array}
\usepackage{framed}
\usepackage{mathrsfs}
\usepackage{wrapfig}
\usepackage[T1]{fontenc}
\usepackage{upgreek}
\usepackage{lmodern}
\usepackage{float}
\usepackage{bm}

\usepackage{booktabs}       % professional-quality tables
\usepackage{amsfonts}       % blackboard math symbols
\usepackage{nicefrac}       % compact symbols for 1/2, etc.
\usepackage{microtype}      % microtypography
\usepackage{enumerate}
\usepackage{lipsum}
\usepackage{mathtools}
\usepackage{cuted}
\usepackage{float}
\usepackage{mdframed}
\usepackage{amssymb}
\usepackage[dvipsnames,table,xcdraw]{xcolor}
\usepackage{bbm}
\usepackage{varioref}
\usepackage{hyperref}
                                
\usepackage{rotating}
\usepackage{graphicx}
\usepackage{subcaption}
\usepackage[normalem]{ulem}

\usepackage{amsthm}
\DeclareUnicodeCharacter{00A0}{~}
\usepackage{algorithm}
\usepackage{algpseudocode}
\usepackage{amsmath}
\usepackage{tikz}
\usepackage{pgfplots}
\pgfplotsset{compat=1.18}
\usetikzlibrary{arrows.meta, positioning}
\usetikzlibrary{decorations.pathreplacing, calc}

\newcommand{\tabincell}[2]{\begin{tabular}{@{}#1@{}}#2\end{tabular}}

\newtheorem{theorem}{Theorem}[section]
\newtheorem{lemma}{Lemma}[section]
\newtheorem{proposition}{Proposition}[section]
\newtheorem{definition}{Definition}[section]
\newtheorem{example}{Example}[section]
\newtheorem{corollary}{Corollary}[section]
\newtheorem{remark}{Remark}[section]
\theoremstyle{plain}
\RequirePackage[capitalize,nameinlink]{cleveref}[0.19]

\def\be{\begin{enumerate}}
\def\ee{\end{enumerate}}

\def\argmin{\mathop{\rm argmin}}

 \newcommand{\remove}[1]{}

\newcommand{\pmat}[1]{\begin{pmatrix} #1 \end{pmatrix}}

\newcommand{\uu}{{\mathbf{u}}}
\newcommand{\vv}{{\mathbf{v}}}
\newcommand{\y}{{\mathbf{y}}}

\def\argmin{\mathop{\rm argmin}}

\def\argmax{\mathop{\rm argmax}}

\definecolor{britishracinggreen}{rgb}{0.0, 0.36, 0.15}

\usepackage[textsize=small]{todonotes}
\usepackage{tcolorbox}

\newcommand{\bxi}{\boldsymbol{\xi}}

 \usepackage{pifont}
 
\crefname{section}{section}{sections}
\crefname{subsection}{subsection}{subsections}
\Crefname{section}{Section}{Sections}
\Crefname{subsection}{Subsection}{Subsections}

\Crefname{figure}{Figure}{Figures}

\crefformat{equation}{\textup{#2(#1)#3}}
\crefrangeformat{equation}{\textup{#3(#1)#4--#5(#2)#6}}
\crefmultiformat{equation}{\textup{#2(#1)#3}}{ and \textup{#2(#1)#3}}
{, \textup{#2(#1)#3}}{, and \textup{#2(#1)#3}}
\crefrangemultiformat{equation}{\textup{#3(#1)#4--#5(#2)#6}}%
{ and \textup{#3(#1)#4--#5(#2)#6}}{, \textup{#3(#1)#4--#5(#2)#6}}{, and \textup{#3(#1)#4--#5(#2)#6}}

\Crefformat{equation}{#2Equation~\textup{(#1)}#3}
\Crefrangeformat{equation}{Equations~\textup{#3(#1)#4--#5(#2)#6}}
\Crefmultiformat{equation}{Equations~\textup{#2(#1)#3}}{ and \textup{#2(#1)#3}}
{, \textup{#2(#1)#3}}{, and \textup{#2(#1)#3}}
\Crefrangemultiformat{equation}{Equations~\textup{#3(#1)#4--#5(#2)#6}}%
{ and \textup{#3(#1)#4--#5(#2)#6}}{, \textup{#3(#1)#4--#5(#2)#6}}{, and \textup{#3(#1)#4--#5(#2)#6}}

\crefdefaultlabelformat{#2\textup{#1}#3}

\title{\Large Stochastic Linesearch Methods for Stochastic Variational Inequality Problems: Parameter-Free Schemes and Non-Monotone Variants
}
\author{Zhuoyu Xiao\footnote{Department of Industrial and Operations Engineering, University of Michigan, Ann Arbor, MI, USA, 48109. Email: zyxiao@umich.edu.} \and Uday V. Shanbhag\footnote{Corresponding author. Department of Industrial and Operations Engineering, University of Michigan, Ann Arbor, MI, USA, 48109. Email: udaybag@umich.edu.} \and Yue Xie\footnote{Department of Mathematics, University of Hong Kong, Hong Kong SAR, China. Email: yxie21@hku.hk.}}

\newmdenv[
  backgroundcolor=gray!15,
  linecolor=white,
  linewidth=0pt,
  innertopmargin=10pt,
  innerbottommargin=10pt,
  innerleftmargin=10pt,
  innerrightmargin=10pt,
]{custombox}

\begin{document}

\maketitle
\thispagestyle{empty}

%%%%%%%%%%%%%%%%%%%%%%%%%%%%%%%%%%%%%%%%%%%%%%%%%%%%%%%%%%%%%%%%%%%%%%%%%%%%%%%%
\begin{abstract}
Much of the existing convergence guarantees and rate statements for stochastic variational inequality (SVI) problems are limited to stochastic approximation schemes that employ constant or diminishing steplength sequences bounded by some reciprocal of the Lipschitz constant $L$ of the map. This requires both knowing such constants and may lead to debilitatingly small steplengths, afflicting the scheme with potentially poor performance. We propose a parameter-free stochastic extragradient method with periodically restarted linesearch (\textbf{SEG-LS}) for solving pseudomonotone and monotone SVI problems. The resulting stepsize sequence is neither monotone nor requires prior knowledge of the Lipschitz constant $L$. We show that the iteration and sample complexities for computing a point whose residual has norm at most $\epsilon$ in expectation are $\mathcal{O}(\epsilon^{-2})$ and $\mathcal{O}(\epsilon^{-4})$, respectively. These results improve to $\mathcal{O}(\epsilon^{-1})$ and $\mathcal{O}(\epsilon^{-2})$ in the monotone setting. In addition, we propose a periodically restarted biased stochastic linesearch method (\textbf{B-SLAM}) for solving non-monotone SVIs under strict copositivity. The resulting iteration and sample complexities are $\mathcal{O}(\epsilon^{-2})$ and $\mathcal{O}(\epsilon^{-4})$, respectively. Preliminary numerics support these findings and display promise.
\end{abstract}


\section{Introduction}\label{Sec-1}

The variational inequality (VI) problem \cite{facchinei-pang-2003} is a broad framework that encompasses a wide range of optimization and equilibrium problems, with applications including convex–concave saddle-point minimax problems in machine learning \cite{kovalev-gasnikov-2022, lin-jin-jordan-2020}, the modeling of imperfectly competitive regimes in electricity markets \cite{hobbs-pang-2007, jing-yuan-smeers-1999}, and routing problems in networks \cite{dafermos-1980, friesz-bernstein-smith-tobin-wie-1993}, among others. Given a closed convex set $X \subseteq \mathbb{R}^n$ and a mapping $F: X \to \mathbb{R}^n$, the variational inequality problem $\mathrm{VI}\:(X,F)$ requires an $x^{\ast} \in X$ such that 
\begin{equation*}
    F(x^{\ast})^{\top}\left(\, x - x^{\ast} \, \right) \, \geq \, 0, \quad \forall \, x \, \in \, X.
\end{equation*}
Our focus lies in the stochastic generalization $\mathrm{SVI}\:(X,F)$ where $F$ is defined as $F(x) \triangleq \mathbb{E} \left[ \Tilde{F}(x, \bxi) \right]$, with $\xi: \Omega \to $ $\mathbb{R}^{m}$ defined on the probability space $(\Omega, \mathcal{F}, \mathbb{P})$ and $\Xi \triangleq \{ \xi(w) \:\vert\: w\in \Omega \}$. Such a problem capture a range of problems, three of which are presented below.

\subsection{Motivating examples}

In what follows, we highlight three representative settings in which SVIs provide a unified framework for machine learning, stochastic optimization, and equilibrium computation.

\begin{example}[Stochastic optimization problems with expectation-valued constraints \cite{akhtar-bedi-rajawat-2021, lan-zhou-2020}]\em
    The class of expectation-constrained stochastic programs capture learning and decision-making problems in which performance is optimized subject to requirements expressed through expected losses or moments. Representative applications include logistic regression with covariance-based fairness constraints \cite{zafar-valera-rogriguez-gummadi-2017} and approximate linear programming for Markov decision processes \cite{lin-nadarajah-soheili-yang-2020}. These applications admit convex formulations with expectation-valued objectives and constraints. We consider the following formulation with smooth convex objective and constraint functions.
    \begin{equation}\label{SOEC}
        \begin{aligned}
            \min_{x \in X} & \;\; f(x) \, \triangleq \, \mathbb{E} \left[ \, \tilde{f}(x,\bxi) \, \right] \\
            \mathrm{s.t.} & \;\; c(x) \, \triangleq \, \mathbb{E} \left[ \, \tilde{c}(x,\bxi) \, \right] \le 0, \quad (\lambda)
        \end{aligned}\tag{SOEC}
    \end{equation}
    where $f : X \to \mathbb{R}$ and $c_i : X \to \mathbb{R}$ are smooth convex functions for $i=1, \hdots, m$. Under some regularity condition, we may show that $z \triangleq (x, \lambda)$ is a KKT point if and only if $z \triangleq (x, \lambda)$ is a solution of $\mathrm{SVI}\:(Z,F)$, where $Z \triangleq X \times \mathbb{R}^{m}_{+}$ and
    \begin{equation*}
        F(x,\lambda)\, \triangleq \, \pmat{\nabla f(x) + \displaystyle \sum_{i=1}^m \lambda_i \nabla c_i(x) \\ -c(x)}.
    \end{equation*}
    Therefore, computing a KKT point of \eqref{SOEC} reduces to solving a monotone SVI problem.
\end{example}

\begin{example}[Stochastic fractional programming \cite{kannan-shanbhag-2019, li-cai-song-ma-2023}]\em
    Fractional programs arise in signal processing and communication networks when performance is measured relative to resource consumption or interference. Their stochastic counterparts account for uncertainty in the quantities defining these ratios \cite{bot-mertikopoulos-staudigl-vuong-2021}. A representative class consists of quadratic fractional models with random numerator coefficients and a positive deterministic affine denominator. When the expected numerator is positive and convex, the resulting objective has a pseudomonotone gradient mapping \cite{bot-mertikopoulos-staudigl-vuong-2021}. To this end, we consider the following expectation-of-ratio formulation.
    \begin{equation}\label{SFP}
        \min_{x \, \in \, X} \;\; f(x) \, \triangleq \, \mathbb{E} \left[ \, \frac{\Tilde{g}(x, \bxi)}{\Tilde{h}(x, \bxi)} \, \right], \tag{SFP}
    \end{equation}
    where $g(x) \triangleq \mathbb{E}[\Tilde{g}(x, \bxi)]$ and $h(x) \triangleq \mathbb{E}[\Tilde{h}(x, \bxi)]$ are strictly positive convex quadratic and linear functions, respectively. By \cite[Lemma 1]{kannan-shanbhag-2019}, it can be known that the optimality condition of \eqref{SFP} is characterized by a pseudomonotone $\mathrm{SVI}\:(X, F)$ problem where $F(x) = \nabla f(x)$.
\end{example}

\begin{example}[Stochastic quasi-Nash equilibrium problems \cite{pang-scutari-2011, xiao-shanbhag-2025-GR}]\em
Stochastic noncooperative games arise naturally in adversarial learning, where competing models optimize coupled expected losses. In generative adversarial networks \cite{gidel-berard-vignoud-vincent-lacoste-julien-2019}, the generator and discriminator form a two-player game whose first-order stationarity conditions admit a variational inequality formulation. Neural-network parameterizations can make the players' objectives nonconvex, so the existence of a pure-strategy Nash equilibrium (NE) is no longer guaranteed. A quasi-Nash equilibrium (QNE) instead captures the players' first-order stationarity conditions and can be characterized through an SVI. We consider the following stochastic $N$-player game.
\begin{equation}\label{SQNE}
    \min_{x_i \, \in \, X_i} \;\; f_i(x_i,x_{-i}) \, \triangleq  \mathbb{E}\left[\, \tilde{f}_i(x_{i}, x_{-i}, \bxi)\, \right], ~ \forall \, i \in [N], \tag{SQNE}
\end{equation}
where each $X_i\, \subseteq \, \mathbb{R}^{n_i}$ is closed and convex, $x_{-i} \, \triangleq \, (x_j)_{j \ne i}$, $X_{-i} \, \triangleq \, \prod_{j \ne i} X_j$, and each $f_i(\bullet, x_{-i}): X_i \to \mathbb{R}$ is smooth but possibly nonconvex. Then a quasi-Nash equilibrium (QNE) of \eqref{SQNE} is a tuple $x^{\ast} = (x^{\ast}_{i})_{i=1}^{N}$ such that for any $i\in [N]$, we have
\begin{equation*}
    \nabla_{x_{i}} f_{i}(x^{\ast}_{i}, x^{\ast}_{-i})^{\top}\left(\, x_{i} - x^{\ast}_{i} \, \right) \, \geq \, 0, ~ \forall \, x_{i} \in X_{i}.
\end{equation*}
Therefore, we may see that $x^{\ast}$ is a QNE if and only if $x^\ast$ is a solution of $\mathrm{SVI}\:(X,F)$ where
\begin{equation*}
    X \, \triangleq \, \prod_{i=1}^N X_i \;\;\text{and}\;\; F(x) \triangleq \pmat{\nabla_{x_1} f_1(x) \\
    \vdots \\
    \nabla_{x_N} f_N(x)}.
\end{equation*}
If each $f_i(\bullet, x_{-i}): X_i \to \mathbb{R}$ is additionally convex, then a QNE reduces to an NE. If $F$ is additionally monotone, a NE can be captured by a solution of the resulting monotone SVI. Furthermore, if each $f_i(\bullet, x_{-i}): X_i \to \mathbb{R}$ is nonsmooth, then an equilibrium of the game can be characterized by a generalized stochastic variational inequality $\mathrm{GSVI}\:(X,F)$, where the map $F$ is a concatenation of the (Clarke) subdifferentials \cite{xiao-2026, xiao-shanbhag-2026-BR} and is therefore a set-valued mapping. In this work, we do not consider such a case and focus exclusively on the single-valued setting.
\end{example}

\subsection{Related work}

This paper is closely related to several influential works in variational inequalities and game-theoretic problems. We review them below. \\

\noindent (I) \textbf{Monotone SVIs.} While deterministic VI problems have seen a comprehensive study over the last 50 years \cite{cottle-pang-stone-2009, facchinei-pang-2003}, there has been relatively less study of the stochastic counterpart in which the map $F$ is expectation-valued. Jiang and Xu \cite{jiang-xu-2008} presented a stochastic approximation framework with asymptotic almost sure convergence guarantees for strongly monotone problems and their variants. Building on the deterministic mirror-prox method \cite{nemirovski-2004}, Juditsky et al. \cite{juditsky-nemirovski-tauvel-2011} developed a stochastic counterpart for monotone SVIs over a compact convex set. Under Lipschitz continuity and an unbiased stochastic oracle, their scheme achieves a sublinear convergence of expected dual gap for the averaged iterate. Further efforts have established asymptotic convergence guarantees for merely monotone settings via iterative regularization \cite{kannan-shanbhag-2012, koshal-nedic-shanbhag-2013}. \\

\noindent (II) \textbf{Pseudomonotone and other non-monotone SVIs.} Pseudomonotone VIs are characterized by maps satisfying pseudomonotonicity, a weaker condition than monotonicity. Stochastic generalizations find relevance in aircraft design \cite{fokin-1998}, the competitive exchange economy \cite{brighi-john-2002}, and product pricing \cite{choi-deSarbo-harker-1990}. Stochastic extragradient schemes have been studied extensively \cite{dang-lan-2015}, with rate statements under strong pseudomonotonicity \cite{kannan-shanbhag-2019} and variance reduction through mini-batches \cite{iusem-jofre-oliveira-thompson-2017}. Yet much of the prior work relies on diminishing or constant steplength schemes requiring knowledge of the Lipschitz constant $L$ of the mapping $F$. Beyond pseudomonotonicity, recent work has investigated non-monotone VIs under the Minty condition \cite{arefizadeh-nedic-2024, hsieh-iutzeler-malick-mertikopoulos-2020, huang-zhang-2024, vankov-nedic-sankar-2023}, weak Minty condition \cite{alacaoglu-kim-wright-2024, bohm-sedlmayer-csetnek-bot-2022, choudhury-gorbunov-loizou-2023, diakonikolas-daskalakis-jordan-2021}, and cohypomonotonicity \cite{cai-alacaoglu-diakonikolas-2023, cai-oikonomou-zheng-2022, cai-zheng-2022, lee-kim-2021}. Non-monotone VIs are also closely related to nonconvex games, whose equilibrium structure often leads naturally to non-monotone operators; recent related developments can be found in \cite{xiao-shanbhag-2025-GR, xiao-2026, xiao-shanbhag-2026-BR}. \\

\noindent (III) \textbf{Stochastic linesearch methods for stochastic optimization and SVIs.} The stochastic linesearch method in \cite{vaswani-mishkin-laradji-schmidt-gidel-lacoste-julien-2019} achieves nonconvex optimization convergence under a strong growth condition and $\mathcal{O}(1/L)$ stepsize bounds. Alternatively, probabilistic oracles enable adaptive stepsizes with expected or high-probability complexity guarantees \cite{berahas-cao-scheinberg-2021, jin-scheinberg-xie-2024, paquette-scheinberg-2020}. A comprehensive comparison between stochastic linesearch methods in optimization can be found in \cite{wang-shanbhag-xie-2025}. Amongst the earliest linesearch schemes for deterministic pseudomonotone VI problems came in the form of an extragradient framework \cite{solodov-svaiter-1999}. A related scheme in the form of a linesearch-enabled hyperplane projection scheme has also been suggested \cite[Chapter 12]{facchinei-pang-2003}. The challenge with incorporating a linesearch in stochastic regimes emerges from a key complication: the resulting steplength is a random variable and may be correlated with the direction employed. This concern was the focus of the study by Iusem et al. \cite{iusem-jofre-oliveira-thompson-2019}, by introducing a variance-reduced EG scheme equipped with a linesearch that does not require prior knowledge of $L$. This scheme achieves an optimal iteration complexity of $\mathcal{O}(\epsilon^{-2})$ and a near-optimal sample complexity of $\mathcal{O}(\epsilon^{-4})$ for attaining a point whose natural map has norm at most~$\epsilon$ in expectation. More recently, Li et al. \cite{li-cai-song-ma-2023} further improved upon \cite{iusem-jofre-oliveira-thompson-2019} by employing different stepsizes in the prediction and correction steps. However, we observe that the underlying analysis in \cite{iusem-jofre-oliveira-thompson-2019} relies on significantly involved and complicated analysis leveraging empirical process theory. Such an avenue may impede generalizability and broader applicability, and a natural question is whether simpler avenues are available. \\

\noindent (IV) \textbf{Randomized smoothing and noisy function evaluations.} Gradient sampling \cite{burke-lewis-overton-2005, kiwiel-2007, kiwiel-2010} and randomized smoothing-based zeroth-order methods \cite{lin-zheng-jordon-2022, marrinan-shanbhag-yousefian-2026} provide two avenues for addressing nonconvex nonsmooth optimization. The latter have also been extended to stochastic MPECs \cite{cui-shanbhag-yousefian-2023} and hierarchical federated optimization \cite{qiu-shanbhag-yousefian-2023}. In related work on noisy function evaluations, Berahas et al.\ \cite{berahas-cao-choromanski-scheinberg-2022} analyze finite-difference, interpolation, and noisy gradient approximations, quantifying the sampling requirements and sampling radii needed for convergence guarantees. Berahas et al. \cite{berahas-cao-scheinberg-2021} further establish convergence rates for a modified linesearch method with bounded function-value noise and inexact, possibly random gradient estimates. These developments provide relevant context for the randomized smoothing-based gradient estimation and stochastic biased linesearch method for non-monotone SVIs.

~

\begin{tcolorbox}
    {\bf Gaps:} (I) Can we develop a parameter-free linesearch-based extragradient scheme for solving $\mathrm{SVI}\:(X,F)$ that avoids reliance on the Lipschitz constant $L$ of $F$ and allows a relatively simple analysis when specifying stepsizes? (II) Can we develop a stochastic gradient-type scheme for solving non-monotone $\mathrm{SVI}\:(X,F)$ beyond pseudomonotonicity? If yes, under what condition can convergence and rate statement be established?
\end{tcolorbox}

\subsection{Outline and contributions}

After providing some preliminaries in Section \ref{Sec-2}, our main schemes are captured in Sections \ref{Sec-3} and \ref{Sec-4}. We provide numerical experiments and concluding remarks in Sections \ref{Sec-5} and \ref{Sec-6}, respectively. Our contributions lie in the development of two stochastic periodically restarted linesearch-based schemes, namely, an extragradient method and a gradient method. We summarize our main results in Table \ref{main-contribution-table}. \\

\noindent (I) \textbf{Periodically restarted linesearch.} The periodically restarted linesearch mechanism was first introduced in \cite{wang-shanbhag-xie-2025} for stochastic optimization, motivated by the limitations of existing stochastic linesearch schemes that often require non-increasing steplength sequences or impose restrictive upper bounds of order $\mathcal{O}(1/L)$ on the initial steplength. Specifically, the iteration horizon is divided into a sequence of periods. At the beginning of each period, the initial steplength for the backtracking linesearch is reset to a prescribed maximum value, while within each period, the initial steplength is set to the accepted steplength from the prior iteration. Consequently, the steplength sequence is non-increasing within each period but may increase when a new period begins; see Figure~\ref{period-example} for an illustration. This mechanism allows the scheme to periodically explore using a larger steplength while retaining the adaptation to the local problem geometry achieved by the linesearch within each period. As a result, the generated steplength sequence need not be globally non-increasing and requires neither prior knowledge of the Lipschitz constant nor restrictive bounds of order $\mathcal{O}(1/L)$ on the initial steplength. In this work, we extend this idea to stochastic variational inequality problems.

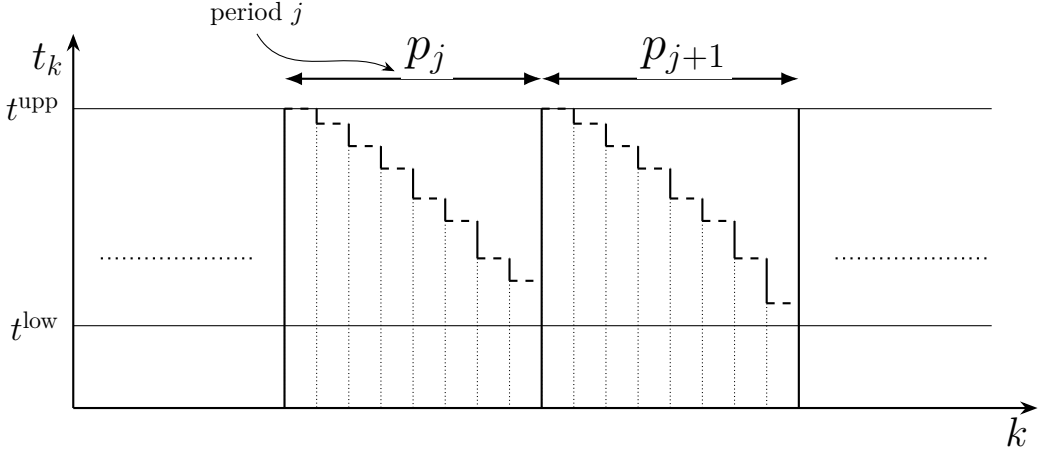
\begin{figure}[t]
\centering
\begin{tikzpicture}[>=Stealth, x=1.35cm, y=1.1cm, scale=0.9, transform shape]

    \draw[->, thick] (0,0) -- (10.5,0) node[below left, font=\LARGE] {$k$};
    \draw[->, thick] (0,0) -- (0,5.0) node[below left, font=\LARGE] {$t_k$};
    
    \draw[thin] (0,4.0) -- (10.0,4.0) node[left, pos=0, font=\Large, xshift=-2pt] {$t^{\mathrm{upp}}$};
    \draw[thin] (0,1.1) -- (10.0,1.1) node[left, pos=0, font=\Large, xshift=-2pt] {$t^{\mathrm{low}}$};

    \draw[thick, dotted] (0.3, 2.0) -- (2.0, 2.0);
    \draw[thick, dotted] (8.3, 2.0) -- (10.0, 2.0);

    \def\xstart{2.3}
    \def\yvals{4.0, 3.8, 3.5, 3.2, 2.8, 2.5, 2.0, 1.7}
    \def\stepw{0.35}
    
    \draw[thick] (\xstart, 0) -- (\xstart, 4.0);
    
    \foreach \y [count=\i] in \yvals {
        \pgfmathsetmacro{\xcurr}{\xstart + (\i-1)*\stepw}
        \pgfmathsetmacro{\xnext}{\xstart + \i*\stepw}
        
        \draw[thick, dashed] (\xcurr, \y) -- (\xnext, \y);
        
        \draw[thin, densely dotted] (\xcurr, \y) -- (\xcurr, 0);
        
        \ifnum\i<8
            \pgfmathsetmacro{\ynext}{{\yvals}[\i]}
            \draw[thick] (\xnext, \y) -- (\xnext, \ynext);
        \fi
    }

    \draw[thick] ({\xstart + 8*\stepw}, 0) -- ({\xstart + 8*\stepw}, 4.0);

    \def\xstarttwo{5.1}
    \def\yvalstwo{4.0, 3.8, 3.5, 3.2, 2.8, 2.5, 2.0, 1.4}
    
    \foreach \y [count=\i] in \yvalstwo {
        \pgfmathsetmacro{\xcurr}{\xstarttwo + (\i-1)*0.35}
        \pgfmathsetmacro{\xnext}{\xstarttwo + \i*0.35}
        
        \draw[thick, dashed] (\xcurr, \y) -- (\xnext, \y);
        
        \ifnum\i<8
            \pgfmathsetmacro{\ynext}{{\yvalstwo}[\i]}
            \draw[thick] (\xnext, \y) -- (\xnext, \ynext);
        \fi
        
        \draw[thin, densely dotted] (\xcurr, \y) -- (\xcurr, 0);
    }

    \draw[thick] ({\xstarttwo + 8*0.35}, 0) -- ({\xstarttwo + 8*0.35}, 4.0);

    \draw[<->, >=Latex, thick] (\xstart, 4.4) -- (5.1, 4.4);
    \node[above, font=\huge, fill=white, inner sep=2pt] at ({( \xstart + 5.4 ) / 2}, 4.4) {$p_j$};
    
    \draw[<->, >=Latex, thick] (5.1, 4.4) -- ({\xstarttwo + 8*0.35}, 4.4);
    \node[above, font=\huge, fill=white, inner sep=2pt] at ({( 5.4 + \xstarttwo + 8*0.35 ) / 2}, 4.4) {$p_{j+1}$};

    \node[font=\normalsize, anchor=south] (period_label) at (2.0, 5.0) {period $j$};
    \draw[thin, ->] (period_label.south) to[out=-60, in=155] (3.5, 4.5);

\end{tikzpicture}
\caption{The reset of the initial stepsize according to the period sequence $\{p_{j}\}$.}
\label{period-example}
\end{figure}

~

\noindent (II) \textbf{A stochastic extragradient method with linesearch.} Existing studies on extragradient methods for solving SVIs either rely on the Lipschitz constant \cite{iusem-jofre-oliveira-thompson-2017} or require a more involved analysis \cite{iusem-jofre-oliveira-thompson-2019}. In this paper, we draw inspiration from the recent study \cite{wang-shanbhag-xie-2025} on a linesearch-based SGD method for stochastic nonconvex smooth optimization, and propose a \emph{parameter-free} stochastic extragradient method with linesearch (\textbf{SEG-LS}). Our algorithm possesses two main features. One is the mini-batch sampling technique to handle stochasticity. The other is the periodical reset mechanism. We reset the initial stepsize for backtracking linesearch at the beginning of each cycle, while the initial stepsize is set to the accepted stepsize from the prior iteration within each cycle; see Figure \ref{period-example} for an illustration. In Section \ref{Sec-3}, we discuss the performances of \textbf{SEG-LS} on two classes of SVIs. (II-A) \emph{Pseudomonotone SVIs.} In contrast to Iusem et al. \cite{iusem-jofre-oliveira-thompson-2019}, our work incorporates a periodically restarted linesearch mechanism without prior knowledge of $L$, while the accompanying analysis is significantly simpler and relies only on standard conditioning arguments. The resulting scheme is characterized by an iteration complexity of $\mathcal{O}(\epsilon^{-2})$ and a sample complexity of $\mathcal{O}(\epsilon^{-4})$ for computing an $\epsilon$-solution, in terms of the expected norm of a suitably defined residual function. (II-B) \emph{Monotone SVIs.} Under the monotonicity, we adopt the expected dual gap function as the convergence measure and show that the iteration and sample complexities are improved to $\mathcal{O}(\epsilon^{-1})$ and $\mathcal{O}(\epsilon^{-2})$, respectively. Notably, the resulting linesearch scheme allows for non-monotone steplength sequences that require neither knowledge of $L$ nor necessitate imposing restrictive bounds of $\mathcal{O}(1/L)$ on the initial steplength.

~

\noindent (III) \textbf{A biased stochastic gradient method with linesearch.} In Section \ref{Sec-4}, we consider a more general class of non-monotone SVIs beyond pseudomonotonicity and propose a biased stochastic linesearch method (\textbf{B-SLAM}). Unlike in Section \ref{Sec-3}, our stochastic gradient method is applied to a regularized gap function rather than the map $F$ itself. The main technical challenge lies in the inexactness of the lower-level solution, which makes unbiased function and gradient oracles unavailable. To address this, we employ the randomized smoothing and develop a stochastic biased Armijo linesearch procedure that performs backtracking using biased function and zeroth-order gradient oracles. Unlike the unbiased setting in Section \ref{Sec-3}, the biased linesearch procedure requires knowledge of $L$ to explicitly set the stepsize lower bound $t^{\mathrm{low}}$. We also show that it terminates in finite time with probability one by incorporating a multiple-batch setting. We show that, under appropriate parameter choices, the method converges to a solution of the SVI under strict copositivity. The iteration and sample complexities are $\mathcal{O}(\epsilon^{-2})$ and $\mathcal{O}(\epsilon^{-4})$, respectively. It should be noted, however, that a limitation of this method is the requirement of prior knowledge of the Lipschitz constant and Lipschitz smoothness constant of the regularized gap function. \\

\begin{table}[!htb]
  \centering
  \footnotesize
  {\renewcommand{\arraystretch}{1.6}
  \begin{tabular}{|c|c|c|c|c|}
    \hline
    \textbf{scheme} & \textbf{parameter-free} & \textbf{condition} & \textbf{iteration complexity} & \textbf{sample complexity} \\
    \hline
    \hline
    
    \tabincell{c}{\textbf{SEG-LS}}
    &
    \tabincell{c}{\ding{51}}
    &
    \tabincell{c}{pseudomonotonicity}
    &
    \tabincell{c}{%
    $\mathcal{O}(\epsilon^{-2})$ (Theorem \ref{SEG-main-theorem})%
    }
    &
    \tabincell{c}{%
    $\mathcal{O}(\epsilon^{-4})$ (Theorem \ref{SEG-main-theorem})%
    }
    \\
    \hline

    \tabincell{c}{\textbf{SEG-LS}}
    &
    \tabincell{c}{\ding{51}}
    &
    \tabincell{c}{monotonicity}
    &
    \tabincell{c}{%
    $\mathcal{O}(\epsilon^{-1})$ (Theorem \ref{SEG-monotone-main-theorem})%
    }
    &
    \tabincell{c}{%
    $\mathcal{O}(\epsilon^{-2})$ (Theorem \ref{SEG-monotone-main-theorem})%
    }
    \\
    \hline

    \tabincell{c}{\textbf{B-SLAM}}
    &
    \tabincell{c}{\ding{55}}
    &
    \tabincell{c}{strict copositivity}
    &
    \tabincell{c}{%
    $\mathcal{O}(\epsilon^{-2})$ (Theorem \ref{slam-main-result})%
    }
    &
    \tabincell{c}{%
    $\mathcal{O}(\epsilon^{-4})$ (Theorem \ref{slam-main-result})%
    }
    \\
    \hline
    
  \end{tabular}}

  \caption{A summary of main results.}
  \label{main-contribution-table}
\end{table}

\noindent\textbf{Notation.} $\Pi_{X}[x]$ denotes the Euclidean projection of $x$ onto set $X$. The expectation and a realization of a random variable $\bxi$ are denoted by $\mathbb{E}[\bxi]$ and $\xi$, respectively. The notation ${\bf 1}_X(x)$ denotes the indicator function of $X$, i.e., ${\bf 1}_{X}(x) = 1$ if $x\in X$ and $0$ otherwise. Let $\mathbf{I}_{n} \in \mathbb{R}^{n\times n}$ and $\mathbf{1}_{n} \in \mathbb{R}^{n\times 1}$ be the identity matrix and the vector of all ones, respectively.

\section{Preliminaries}\label{Sec-2}
We first recall a sufficient guarantee for the existence of a solution of $\mathrm{SVI}\:(X,F)$.

\begin{theorem}[Existence of solution \cite{ravat-shanbhag-2011}]\em
    Consider $\mathrm{SVI}\:(X,F)$ where $X$ is closed and convex. Then the solution set $\mathrm{SOL}\:(X, F)$ is nonempty if (i) or (ii) holds: (i) If $X$ is compact; (ii) If there exists $x^{\mathrm{ref}}\in X$ such that $\liminf\limits_{\|x\|\to\infty, x\in X} \Tilde{F}(x, \xi)^{\top}(x-x^{\mathrm{ref}})\geq 0$ holds almost surely. $\hfill \Box$
\end{theorem}

Next we provide some mathematical preliminaries. In particular, we will present two gap functions as solution measures of $\mathrm{SVI}\:(X, F)$ in subsection.

\subsection{Dual gap function and regularized gap function}\label{Sec-2.1}

Different from \cite{facchinei-pang-2003} in which the gap function is employed, here we consider the dual gap function \cite{facchinei-pang-2003,juditsky-nemirovski-tauvel-2011}. Under monotonicity and continuity, a feasible zero of the dual gap function is a solution to $\mathrm{SVI}\:(X,F)$, and vice versa.

\begin{definition}[Dual gap function \mbox{\cite[p. 18]{juditsky-nemirovski-tauvel-2011}}]
    Suppose $F: X \to \mathbb{R}^{n}$. Then the dual gap function, denoted by $\theta$, is defined as
    \begin{equation}\label{weak-gap-function}
        \theta(x) \triangleq \sup_{y\in X} F(y)^{\top} (x-y).
    \end{equation}
\end{definition}

\begin{proposition}[\mbox{\cite[p. 18]{juditsky-nemirovski-tauvel-2011}}]\label{weak-gap-function-property}\em
    Consider $\mathrm{SVI}\:(X,F)$ where $X$ is closed and convex. Suppose that $F$ is monotone and continuous. For $x\in X$, we have that $\theta(x) = 0$ if and only if $x \in \mathrm{SOL}\:(X,F)$. $\hfill \Box$
\end{proposition}

Another concept that has been used for solving non-monotone VIs is the regularized gap function, which is defined below.

\begin{definition}[Regularized gap function \mbox{\cite[Definition 10.2.2]{facchinei-pang-2003}}]
 Suppose $F: X \to \mathbb{R}^{n}$ and $c > 0$. Then the associated regularized gap function, denoted by $\theta_{c}$, is defined as
    \begin{equation}\label{regularized-gap-function}
        \theta_{c}(x) \triangleq \sup_{y \in X} \left( F(x)^{\top} (x-y) - \frac{c}{2}(x-y)^{\top} (x-y) \right).
    \end{equation}
\end{definition}

\begin{lemma}[\mbox{\cite[Proposition 3.1]{li-ng-2009}} and \mbox{\cite[Theorem 10.2.3]{facchinei-pang-2003}}]\label{lemma-regularized-gap-function}\em
    Consider $\mathrm{SVI}\:(X,F)$ where $X$ is closed and convex and the associated regularized gap function $\theta_{c}(x)$. Then the following hold.
    
    \noindent (i) For any $x\in X$, there exists a unique vector $y_{c}(x)\in X$ at which the supremum in \eqref{regularized-gap-function} is attained at $y_{c}(x) = \Pi_{X}[x-\frac{1}{c}F(x)]$, based on
    \begin{align}
        y_{c}(x) \, &\triangleq \, \argmax_{y\in X} \, \left(\, F(x)^\top (x-y) - \tfrac{c}{2}(x-y)^\top (x-y) \,\right) \,\label{yc-x-1} \\
        &= \, \argmin_{y\in X} \, \mathbb{E}_{\bxi} \left[ \underbrace{\Tilde{F}(x, \bxi)^\top (y-x) + \tfrac{c}{2}(y-x)^\top (y-x)}_{\triangleq \, \widehat{F}(x, \, y, \, \bxi) } \right]. \label{yc-x-2}
    \end{align}
    
    \noindent (ii) $y_{c}$ and $\theta_{c}$ are continuous on $\mathbb{R}^{n}$ and $\theta_{c}(x)\geq 0$ for all $x\in X$.

    \noindent (iii) $\left[\, x\, \in \, X,\:\theta_{c}(x)\, =\, 0\, \right] \, \iff \,  \left[\, x = y_{c}(x)\, \right] \, \iff \, \left[\, x\, \in \, \mathrm{SOL}\:(X,F) \, \right]$.

    \noindent (iv) If $F$ is locally Lipschitz on $\mathbb{R}^{n}$, then $y_{c}$ and $\theta_{c}$ are also locally Lipschitz on $\mathbb{R}^{n}$.
    
    \noindent (v) If $F$ is $\mathrm{C}^1$, then $\theta_{c}$ is $\mathrm{C}^1$ and $\nabla \theta_{c}(x)$ can be expressed as  
    \begin{equation}\label{regularized-gap-function-gradient}
        \nabla\theta_{c}(x) = F(x) + \mathbf{J} F(x)^\top (x-y_{c}(x)) - c(x-y_{c}(x)),
    \end{equation}
    where $\mathbf{J} F(x)$ denotes the Jacobian of $F$ at $x$. $\hfill \Box$
\end{lemma}

By \eqref{yc-x-1}, we may observe that the regularized gap function $\theta_{c}$ in \eqref{regularized-gap-function} can be written as
\begin{equation*}
    \theta_{c}(x)\, = \, F(x)^\top (x-y_{c}(x)) - \frac{c}{2}(x-y_{c}(x))^{\top}(x-y_{c}(x)).
\end{equation*}
The regularized gap function $\theta_{c}$ may well lose convexity. Suppose that $\mathbf{J}F(x)$ exists, implying that $\theta_{c}$ is C$^{1}$. By Lemma \ref{lemma-regularized-gap-function} (ii)-(iii), our goal lies in finding a feasible zero of $\theta_{c}$, which effectively requires globally minimizing $\theta$ over the set $X$. However, gradient-based methods will \emph{at best} provide guarantees for computing a stationary point of $\theta_{c}$, i.e., $0\in \nabla \theta_{c}(x) + {\cal N}_{X}(x)$ holds, where ${\cal N}_{X}(x)$ is the normal cone to $X$ at $x$. But under some conditions, a stationary point of $\theta_{c}$ on $X$ can indeed  be shown to be a feasible zero. We elaborate on such a condition below after recalling the definition of strict copositivity on a cone.

\begin{definition}[\mbox{\cite[Definition 2.5.1]{facchinei-pang-2003}}]
    We say a matrix $M$ is strictly copositive on a cone $C$ if for all $x\in C\setminus \{0\}$ we have $x^{\top}Mx>0$.
\end{definition}

\begin{lemma}[\mbox{\cite[Corollary 10.2.7]{facchinei-pang-2003}}]\label{copositivity_2}\em
    Consider $\mathrm{SVI}\:(X,F)$ where $X$ is closed and convex and $F$ is $\mathrm{C}^1$ on an open set containing $X$. Suppose $x$ is a stationary point of $\theta_{c}$ on $X$, i.e., $0\in \nabla \theta_{c}(x) + \mathcal{N}_{X}(x)$. If $\mathbf{J} F(x)$ is strictly copositive on $\mathcal{T}_{X}(x)\cap (-F(x))^{\ast}$ for any $x\in X$, where $\mathcal{T}_{X}(x)$ denotes the tangent cone to $X$ at $x$ and the dual cone $(-F(x))^{\ast}$ is defined as $(-F(x))^{\ast} \triangleq \{ d \,\vert\, d^{\top}F(x)\leq 0 \}$, then $x$ is an isolated solution of $\mathrm{SVI}\:(X,F)$. $\hfill \Box$
\end{lemma}

\subsection{Randomized smoothing}\label{Sec-2.2}

We now discuss the form of smoothing that allows for constructing tractable gradient estimators of $\theta_c$, as $\nabla \theta_c(x)$ in \eqref{regularized-gap-function-gradient} is not directly accessible due to the unavailability of $\mathbf{J}F(x)$ and $y_c(x)$ in the stochastic setting ($\mathbf{J}F(x)$ is a random matrix). Let $f: X\subseteq \mathbb{R}^{n} \to \mathbb{R}$ be locally Lipschitz continuous. Recall that the randomized smoothing \cite{steklov-1907} of $f$ with $\eta > 0$ is defined as 
\begin{equation}\label{RS-function-value}
    f^{\eta}(x) \triangleq \mathbb{E}_{\uu\in \mathbb{B}} \left[\, f(x+\eta \uu)\, \right],
\end{equation}
where $\mathbb{B}$ denotes the unit ball and $u$ is uniformly distributed over $\mathbb{B}$. We recall some basic properties of randomized smoothing in the following lemma, where we denote the surface of $\mathbb{B}$ by $\mathbb{S}$ and the Minkowski sum of $X$ and $\eta \mathbb{B}$ by $X_{\eta} \triangleq X+\eta \mathbb{B}$.

\begin{lemma}[\mbox{\cite[Lemma 2.4]{marrinan-shanbhag-yousefian-2026}}]\label{RS-property-grad}\em
    Consider $f: X\subseteq \mathbb{R}^{n}\to \mathbb{R}$ and its randomized smoothing $f^{\eta}$, where $\eta>0$. Then the following hold.
    
    \noindent (i) $f^{\eta}$ is C$^{1}$ over $X$ and 
    \begin{align}\label{RS-gradient}
       \nabla_{x} f^{\eta}(x) \, = \, \left(\frac{n}{2\eta}\right)\mathbb{E}_{\vv\in \eta\mathbb{S}}\left[\, (f(x+\vv)-f(x-\vv))\frac{\vv}{\|\vv\|}\, \right]
    \end{align}
     holds for all $x\in X$. Suppose $f$ is $L_{0}$-Lipschitz continuous on $X_{\eta}$. For any $x, y\in X$, (ii)-(v) hold.

    \noindent (ii) $|f^{\eta}(x)-f^{\eta}(y)|\leq L_{0}\|x-y\|$. \emph{(iii)} $|f^{\eta}(x)-f(x)|\leq L_{0}\eta$.
    
    \noindent (iv) $\|\nabla_{x}f^{\eta}(x)-\nabla_{x}f^{\eta}(y)\|\leq \frac{L_{0}\sqrt{n}}{\eta}\|x-y\|$.

    \noindent (v) $f$ is $L_{1}$-smooth on $X_{\eta} \implies  \forall x \in X_{\eta}$, $\| \nabla_{x}f^{\eta}(x) - \nabla_{x}f(x) \|\leq \eta L_{1} n$.

    \noindent (vi) For any $x\in X$, we have $\mathbb{E}_{\vv \in \eta\mathbb{S}}[\|g(x, \vv)\|^{2}]\leq 16\sqrt{2\pi}L^{2}_{0}n$ where $g(x, \vv) \triangleq \left( \frac{n(f(x + \vv)-f(x - \vv)) \vv}{2\eta\| \vv \|} \right)$. $\hfill \Box$
\end{lemma}

The $\delta$-Clarke subdifferential \cite{goldstein-1977} of a Lipschitz continuous function $f: \mathbb{R}^{n}\to \mathbb{R}$ is defined as
\begin{equation*}
    \partial^{C}_{\delta} f(x) \, \triangleq \, \mbox{conv} \left\{ \, \zeta \, \mid \,  \zeta\in \partial^{C} f(y),\; \left\|\, x-y\, \right\|\, \leq \, \delta \, \right\} \, = \, \mbox{conv} \left(\, \bigcup_{ \, \|y-x\| \le \delta} \partial^C f(y) \right).
\end{equation*}
The following lemma is useful for our approximation analysis, showing the inclusion relation between the smoothed gradient and the $\delta$-Clarke subdifferential.

\begin{lemma}[\mbox{\cite[Proposition 2.5]{marrinan-shanbhag-yousefian-2026}}]\label{RS-property-inclusion}\em
    Consider the optimization problem $\min_{x\in X} f(x)$ where $f$ is locally Lipschitz continuous and $X \subseteq \mathbb{R}^{n}$ is closed, convex, and bounded. Then we have: (i) For any $\eta>0$ and any $x\in \mathbb{R}^{n}$, $\nabla f^{\eta}(x)\in \partial^{C}_{\eta} f(x)$; and (ii) For any $\eta>0$ and any $x\in X$, if $0\in \nabla f^{\eta}(x) + \mathcal{N}_{X}(x)$ holds, we have $0\in \partial^{C}_{\eta} f(x) + \mathcal{N}_{X}(x)$. $\hfill \Box$
\end{lemma}

\section{SEG-LS under Pseudomonotonicity and Monotonicity}\label{Sec-3}

In this section, we develop a parameter-free stochastic extragradient scheme equipped with a linesearch procedure for solving pseudomonotone and monotone SVIs. Given that our problem of interest is characterized by an expectation-valued map, we always assume that for any given $x$, there exists a stochastic oracle such that it generates $\Tilde{F}(x, \xi)$, a sample realization associated with $F(x) \triangleq \mathbb{E}[\tilde{F}(x,\bxi)]$.

\subsection{SEG-LS for pseudomonotone SVIs}\label{Sec-3.2}

In this subsection, we consider the following standard assumption. For simplicity, here we assume that $\Tilde{F}(\bullet, \xi)$ is $L$-Lipschitz for any $\xi\in \Xi$. We may obtain analogues of our main results by instead assuming that $\Tilde{F}(\bullet, \xi)$ is $\tilde{L}(\xi)$-Lipschitz under a suitable integrability condition on $\tilde{L}(\bxi)$.

\vspace{4pt}

\begin{custombox}
    \textbf{Assumption $\mathrm{A}$.} \\
    (A1) $F(\bullet)$ is pseudomonotone, i.e. $(x-y)^{\top}F(y) \geq 0 \implies (x-y)^{\top}F(x) \geq 0$, $\forall x, y\in X$. \\
    (A2) For every $\xi\in \Xi$, $\Tilde{F}(\bullet, \xi)$ is $L$-Lipschitz, i.e., $\| \Tilde{F}(x, \xi) - \Tilde{F}(y, \xi) \|\leq L\| x - y\|$, $\forall x, y\in X$. \\
    (A3) The set $X$ is compact with diameter $D_{X} \triangleq \sup_{x, y\in X} \| x - y \| < \infty$ and there exists $D_{F} > 0$ such that $\left\|\, F(x)\, \right\|\, \leq \, D_{F}$ holds over $X$. \\
    (A4) For any $x\, \in \, X$, there exists $\sigma > 0$ such that the following hold almost surely: (i) $\mathbb{E} [\, \Tilde{F}(x, \bxi) \:\vert\: x\, ] \, = \, F(x)$; and (ii) $\mathbb{E} [\, \| \, \Tilde{F}(x, \bxi) - F(x) \, \|^{2} \:\vert\: x\, ] \, \leq \, \sigma^{2}$.
\end{custombox}

\vspace{4pt}

\begin{remark}\em
    We provide a brief commentary on the assumptions. Regarding (A1), observe that $F$ is pseudomonotone over $X$ if it satisfies monotonicity on $X$, i.e., for any $x, y \in X$, $(F(x)-F(y))^\top(x-y) \ge 0$. While (A2) requires Lipschitzian properties on $\tilde{F}(\bullet,\xi)$, this is essential since we will be employing a linesearch scheme on the sampled problem.  The compactness requirement in (A3) can potentially be weakened in prior work. The standard unbiasedness and second moment bound assumptions are presented in (A4). $\hfill \Box$
\end{remark}

\subsubsection{SEG-LS scheme}

Our proposed scheme requires generating independent batches of $N_k$ i.i.d. realizations of the random variable $\bxi$, as specified by $\{ \xi^{k}_{i} \}_{i=1}^{N_{k}}$ and $\{ \zeta^{k}_{i} \}_{i=1}^{N_{k}}$. We may define the associated sample means as
\begin{equation}\label{sample-mean-F}
    \bar{F}_{k}(\bullet, \xi^{k}) \triangleq \frac{1}{N_{k}}\sum_{i=1}^{N_{k}} \Tilde{F}(\bullet, \xi^{k}_{i}) \;\;\text{and}\;\; \bar{F}_{k}(\bullet, \zeta^{k}) \triangleq \frac{1}{N_{k}}\sum_{i=1}^{N_{k}} \Tilde{F}(\bullet, \zeta^{k}_{i}),
\end{equation}
respectively. This allows us to define the history. Specifically, $\mathcal{F}_0 \triangleq \{x_0\}$ and for any $k > 0$, we have 
\begin{align}\label{SEG-history}
    \mathcal{F}_{k} \triangleq \left\{ x^{0}, \left\{ \{\xi^{j}_{i}\}_{i=1}^{N_{j}} \right\}_{j=0}^{k-1}, \left\{ \{ \zeta^{j}_{i} \}_{i=1}^{N_{j}} \right\}_{j=0}^{k-1} \right\}.
\end{align}
In addition, we define $\mathcal{F}_{k+1/2}$ as  
\begin{align}\label{SEG-history-2}
\mathcal{F}_{k+1/2} \triangleq \mathcal{F}_{k} \cup \{\xi_{i}^{k}\}_{i=1}^{N_{k}}.
\end{align}
The following stochastic Armijo line search is defined via the sample-mean objectives defined above.

\vspace{4pt}

\begin{custombox}
\noindent\textbf{stochastic Armijo linesearch}: Given $\alpha, \beta \in (0, 1)$. At iteration $k$, we reduce the stepsize $t_{k}$ at a rate $\beta>0$ from some initial value $t^{\mathrm{init}}_{k} > 0$ until the following condition is satisfied:
\begin{equation}\label{SEG-Armijo}
    t_{k} \left\| \, \bar{F}_{k}\left(\, \Pi_{X}\left[\, x^{k}-t_{k}\bar{F}_{k}(x^{k}, \xi^{k})\, \right], \xi^{k}\, \right) - \bar{F}_{k}(x^{k}, \xi^{k}) \, \right\| \, \leq \, \alpha \left\| \, \Pi_{X}\left[\, x^{k}-t_{k}\bar{F}_{k}(x^{k}, \xi^{k})\, \right] - x^{k} \, \right\|.
\end{equation}
\end{custombox}

\vspace{4pt}

The proposed \textbf{SEG-LS} scheme is given in Algorithm \ref{SEG-LS}. It can be observed that $k$ is the actual number of iterations, while the counter $j$ denotes the cycle number, and $d$ represents the number of iterations within each cycle, which is reset to $1$ at the beginning of the next cycle.

\begin{algorithm}
\caption{~\textbf{SEG-LS}}\label{SEG-LS}
\begin{algorithmic}[1]
\State \textbf{Input}: $x^{0}\in X$, $\alpha \in (0,1)$, $\beta \in (0,1)$, $\{p_{j}\}$ where $p_{j}\in \mathbb{N}$, $t^{\mathrm{upp}}\in (0, \infty)$.
\State Set $j\leftarrow -1$;
\For{$k = 0, 1, \cdots, K-1$}
    \State Generate i.i.d. samples $\{ \xi^{k}_{i} \}_{i=1}^{N_{k}}$ and $\{ \zeta^{k}_{i} \}_{i=1}^{N_{k}}$;
    \State \textbf{if} \: $k = 0$ or $d = p_{j}$ \: \textbf{then}
    \State \quad\quad Set $t^{\mathrm{init}}_{k} \leftarrow t^{\mathrm{upp}}$;\; Set $d \leftarrow 1$, $j \leftarrow j+1$;
    \State \textbf{else}
    \State \quad\quad Set $t^{\mathrm{init}}_{k} \leftarrow t_{k-1}$;\; Set $d \leftarrow d+1$;
    \State Set $t_{k} \leftarrow t^{\mathrm{init}}_{k}$;
    \State \textbf{while} \: \eqref{SEG-Armijo} is not satisfied \: \textbf{do}
    \State \quad\quad Set $t_{k} \gets \beta t_{k}$;
    \State Update $z^{k}$ and $x^{k+1}$ as follows, where $\bar{F}_{k}(\bullet, \xi^{k})$ and $\bar{F}_{k}(\bullet, \zeta^{k})$ are defined in \eqref{sample-mean-F}.
    \begin{align}
        z^{k} &\, \triangleq \, \Pi_{X}\left[\, x^{k} - t_{k} \bar{F}_{k}(x^{k}, \xi^{k})\, \right], \label{SEG-LS-eqn1} \\
        x^{k+1} &\, \triangleq \, \Pi_{X}\left[\, x^{k} - t_{k} \bar{F}_{k}(z^{k}, \zeta^{k})\, \right]. \label{SEG-LS-eqn2}
    \end{align}
\EndFor
\State \textbf{Output}: $x^{R_{K}}$ as final estimate where $R_{K}$ is a random variable with $\mathbb{P}[R_{K} = j] = 1/K$ for any $j\in \{ 0, 1, \cdots, K-1 \}$.
\end{algorithmic}
\end{algorithm}

\subsubsection{Supporting lemmas}

In this subsection, we consider the residual $G_t$, given by $G_{t}(x) \triangleq t^{-1}\left( x - \Pi_{X}\left[ x - t F(x) \right] \right)$. The following lemma establishes an equivalent relationship between $\left\|\, G_{t}(x) \,\right\|$ and $\mathrm{SOL}\:(X, F)$.

\begin{lemma}[\mbox{\cite[Proposition 1.5.8]{facchinei-pang-2003}}]\em
    For any $x^{\ast}\in X$, we have that $\left\|\, G_{t}(x^{\ast}) \,\right\| = 0$ if and only if $x^{\ast}$ is a solution of $\mathrm{VI}\:(X, F)$. $\hfill \Box$
\end{lemma}

The following lemma derives lower and upper bounds for the stepsize sequence $\{ t_{k} \}^{K - 1}_{k\geq 0}$.

\begin{lemma}[Bounds on $t_{k}$]\label{SEG-stepsize}\em
    Consider the stepsize sequence $\{ t_{k} \}^{K - 1}_{k\geq 0}$ generated by \textbf{SEG-LS}. Suppose that Assumption $\mathrm{A}$ holds. Then for any $k = 0, 1, \cdots, K-1$, we always have
    \begin{equation*}
        \min \left\{ \, t^{\mathrm{upp}}, \tfrac{\alpha\beta}{L} \, \right\} \, \triangleq t^{\mathrm{low}} \, \leq \,  t_{k} \, \leq \,  t^{\mathrm{upp}}.
    \end{equation*}
\end{lemma}
\begin{proof}
    The upper bound requirement, given by $t_{k}\leq t^{\mathrm{upp}}$ for any $k$, follows immediately from the initialization at the beginning of each cycle (Line 6) and the fact that the stepsize sequence is are non-increasing (Lines 8 and 11) within each cycle. Next, we establish the lower bound $t_{k}\, \geq \, t^{\mathrm{low}} \triangleq \, \min \left\{\, t^{\mathrm{upp}}, \frac{\alpha\beta}{L} \, \right\}$. We consider two possibilities. (i) If no backtracking (Line 11) occurs from the beginning of the cycle containing the $k$-th iteration up to iteration $k$, we have that $t_{k} = t^{\mathrm{upp}} \geq t^{\mathrm{low}}$. (ii) Otherwise, we need to implement at least one backtracking step (Line 11) within such a cycle. Then the accepted stepsize $t_{k}$ satisfying the stochastic Armijo condition \eqref{SEG-Armijo} would be less than $t^{\mathrm{upp}}$ and the prior stepsize $t_{k}/\beta$ does not satisfy \eqref{SEG-Armijo}, i.e.,
    \begin{equation}\label{SEG-stepsize-eqn2}
        \begin{aligned}
            \frac{t_{k}}{\beta} \left\|\, \bar{F}_{k} \left( \Pi_{X} \left[\, x^{k}-\frac{t_{k}}{\beta}\bar{F}_{k}(x^{k}, \xi^{k}) \, \right], \xi^{k} \, \right) - \bar{F}_{k}(x^{k}, \xi^{k}) \right\| \, > \, \alpha \left\| \, \Pi_{X} \left[ x^{k}-\frac{t_{k}}{\beta}\bar{F}_{k}(x^{k}, \xi^{k})\, \right] - x^{k} \, \right\|.
        \end{aligned}
    \end{equation}
    On the other hand, since $\Tilde{F}(\bullet, \xi^{k}_{i})$ is $L$-Lipschitz by assumption $\mathrm{(A2)}$, we have that $\bar{F}_{k}(\bullet, \xi^{k})$ is $L$-Lipschitz and
    \begin{equation}\label{SEG-stepsize-eqn3}
        \begin{aligned}
            \left\| \bar{F}_{k} \left( \Pi_{X} \left[ x^{k}-\frac{t_{k}}{\beta}\bar{F}_{k}(x^{k}, \xi^{k}) \right], \xi^{k} \right) - \bar{F}_{k}(x^{k}, \xi^{k}) \right\| \leq L \left\| \Pi_{X} \left[ x^{k}-\frac{t_{k}}{\beta}\bar{F}_{k}(x^{k}, \xi^{k}) \right] - x^{k} \right\|
        \end{aligned}
    \end{equation}
    holds. By comparing \eqref{SEG-stepsize-eqn2} and \eqref{SEG-stepsize-eqn3}, we deduce that $t_{k} \geq \frac{\alpha \beta}{L} \geq t^{\mathrm{low}}$. Combing these two possibilities, it follows that $t_{k}\geq t^{\mathrm{low}} \triangleq \min \{ t^{\mathrm{upp}}, \frac{\alpha\beta}{L} \}$ for any $k = 0, 1, \cdots, K-1$.
\end{proof}

Figure \ref{period-example} demonstrates an example of the associated periods. Given the generated stepsize sequence $\{ t_{k} \}_{k = 0}^{K-1}$, we denote the number of times that the stepsize changes during $K$ iterations by $T(K)$, i.e.,
\begin{equation*}
    T(K) \triangleq \sum_{k=0}^{K-1} \mathbf{1}_{\{t_{k}\neq t_{k-1}\}} \;\;\text{with}\;\; t_{-1} \triangleq t^{\mathrm{upp}},
\end{equation*}
where $\mathbf{1}_{\{t_{k}\neq t_{k-1}\}} = 1$ if $t_k \ne t_{k-1}$ and $0$ otherwise. By leveraging Lemma \ref{SEG-stepsize}, we derive a deterministic upper bound of $T(K)$. The proof of the following lemma is similar to that of \cite[Lemma 4]{wang-shanbhag-xie-2025} hence we omit it here.

\begin{lemma}[Bound on $T(K)$]\label{bound-TK-lemma}\em
    Consider the stepsize sequence $\{ t_{k} \}_{k=0}^{K-1}$ generated by \textbf{SEG-LS}. Suppose that Assumption $\mathrm{A}$ holds. Then we can bound $T(K)$ from above as follows.
    \begin{equation*}
        T(K)\leq (J(K)+1) \left\lceil \left[ \log_{1/\beta}\left(\frac{t^{\mathrm{upp}}L}{\alpha}\right) \right]_{+} + 2 \right\rceil,
    \end{equation*}
    where $J(K)$ is defined as
    \begin{equation*}
        J(K)\, \triangleq \, \min \left\{ l\in \mathbb{N} \;\;\middle\vert\;\; \sum_{j=0}^{l} p_{j} \geq K \right\},
    \end{equation*}
    i.e., $J(K)+1$ is the total number of cycles during $K$ iterations. $\hfill \Box$
\end{lemma}

Note that $J(K)$ is a deterministic quantity since $\{p_{j}\}$ and $K$ are prescribed. The following lemma is useful in our subsequent analysis.

\begin{lemma}[\mbox{\cite[Lemma 2.1]{iusem-jofre-oliveira-thompson-2019}}]\label{useful-lemma-SEG}\em
    Suppose that $X\subseteq \mathbb{R}^{n}$ is a closed and convex set. Let $v\in \mathbb{R}^{n}$, $x\in X$, and $z \triangleq \Pi_{X}[x-v]$. Then for all $u\in X$,  $2v^{\top}(z-u) \, \leq \, \left\|\, x-u\,\right\|^{2} - \left\|\, z-u \, \right\|^{2} - \left\| \, z-x \, \right\|^{2}$. $\hfill \Box$
\end{lemma}

We define $w^{k}_{1}$, $w^{k}_{2}$, and $w^{k}_{3}$ as 
\begin{align}
    w^{k}_{1} &\triangleq \bar{F}_{k}(x^{k}, \xi^{k}) - F(x^{k}), \label{w_k_1} \\
    w^{k}_{2} &\triangleq \bar{F}_{k}(z^{k}, \zeta^{k}) - F(z^{k}), \label{w_k_2} \\
    w^{k}_{3} &\triangleq \bar{F}_{k}(z^{k}, \xi^{k}) - F(z^{k}). \label{w_k_3}
\end{align}

By Assumption $\mathrm{(A4)}$, we know that
\begin{equation}\label{wk_1-wk_2-bounds}
    \mathbb{E}[\|w^{k}_{1}\|^{2}] \leq \frac{\sigma^{2}}{N_{k}} \;\;\text{and}\;\; \mathbb{E}[\|w^{k}_{2}\|^{2}] \leq \frac{\sigma^{2}}{N_{k}}.
\end{equation}
However, such a second moment bound does not hold for $w^{k}_{3} \triangleq \bar{F}_{k}(z^{k}, \xi^{k}) - F(z^{k})$ since $z^{k}$ is dependent of $\xi^{k}$ by \eqref{SEG-LS-eqn1}. To analyze $w^{k}_{3}$, we introduce the auxiliary iterate $\bar{z}^{k}$ as
\begin{equation*}
    \bar{z}^{k} \triangleq \Pi_{X}[x^{k} - t_{k-1}F(x^{k})] \;\;\text{for all}\;\; k\in \mathbb{N} \;\;\text{and}\;\; t_{-1} \triangleq t^{\mathrm{upp}}.
\end{equation*}
Observe that $\bar{z}^{k}$ is independent of the samples generated at iteration $k$ and is adapted to $\mathcal{F}_k$, as specified in \eqref{SEG-history}. Next, we derive  a moment bound of $w^{k}_{3}$ by leveraging $\bar{z}^{k}$.

\begin{lemma}[Second moment bound of $w^{k}_{3}$]\label{w3-bound-lemma}\em
    Consider the sequence $\{ x^{k} \}_{k=0}^{K-1}$ generated by \textbf{SEG-LS}. Suppose that Assumption $\mathrm{A}$ holds. We may bound $w^{k}_{3} \triangleq \bar{F}_{k}(z^{k}, \xi^{k}) - F(z^{k})$ from above
    \begin{equation}\label{w3-bound-lemma-main-result}
        \mathbb{E}[\|w^{k}_{3}\|^{2}] \leq (12(t^{\mathrm{upp}})^{2}L^{2}+3)\frac{\sigma^{2}}{N_{k}} + 12(t^{\mathrm{upp}}-t^{\mathrm{low}})^{2}D^{2}_{F}L^{2}\mathbb{P}\{t_{k}\neq t_{k-1}\}.
    \end{equation}
\end{lemma}
\begin{proof}
    By assumption $\mathrm{(A2)}$, we have that
    \begin{align}\label{w3-bound-lemma-eqn1}
        \| w^{k}_{3} \|^{2} &= \| F(z^{k}) - \bar{F}_{k}(z^{k}, \xi^{k}) \|^{2} \notag \\
        &= \| F(z^{k}) - F(\bar{z}^{k}) + F(\bar{z}^{k}) - \bar{F}_{k}(\bar{z}^{k}, \xi^{k}) + \bar{F}_{k}(\bar{z}^{k}, \xi^{k}) - \bar{F}_{k}(z^{k}, \xi^{k}) \|^{2} \notag \\
        &\leq 3\| F(z^{k}) - F(\bar{z}^{k})\|^{2} + 3\| F(\bar{z}^{k}) - \bar{F}_{k}(\bar{z}^{k}, \xi^{k}) \|^{2} + 3 \| \bar{F}_{k}(\bar{z}^{k}, \xi^{k}) - \bar{F}_{k}(z^{k}, \xi^{k}) \|^{2} \notag \\
        &\leq 6L^{2} \| z^{k} - \bar{z}^{k} \|^{2} + 3\| F(\bar{z}^{k}) - \bar{F}_{k}(\bar{z}^{k}, \xi^{k}) \|^{2}.
    \end{align}
    By taking the unconditional expectation on both sides of \eqref{w3-bound-lemma-eqn1}, it leads to
    \begin{equation}\label{w3-bound-lemma-eqn2}
        \mathbb{E}[\| w^{k}_{3} \|^{2}] \leq \underbrace{6L^{2} \mathbb{E}[\| z^{k} - \bar{z}^{k} \|^{2}]}_{(i)} \,+\, \underbrace{3\mathbb{E}[\| F(\bar{z}^{k}) - \bar{F}_{k}(\bar{z}^{k}, \xi^{k}) \|^{2}]}_{(ii)}.
    \end{equation}
    We know from assumption $\mathrm{(A4)}$ that
    \begin{equation}\label{w3-bound-lemma-eqn3}
        (ii) \leq 3\frac{\sigma^{2}}{N_{k}}.
    \end{equation}   
    Next we bound $(i)$ via a conditioning argument as follows. For $\mathbb{E}[\| \bar{z}^{k} - z^{k} \|^{2}]$, we have that
    \allowdisplaybreaks
    \begin{align*}
        &\mathbb{E} [\| \bar{z}^{k} - z^{k} \|^{2}] = \mathbb{E}[\| \Pi_{X}[x^{k} - t_{k-1}F(x^{k})] - \Pi_{X}[x^{k} - t_{k}\bar{F}_{k}(x^{k}, \xi^{k})] \|^{2}] \\
        &= \mathbb{E} \left[ \| \Pi_{X}[x^{k} - t_{k-1}F(x^{k})] - \Pi_{X}[x^{k} - t_{k}F(x^{k})] + \Pi_{X}[x^{k} - t_{k}F(x^{k})] - \Pi_{X}[x^{k} - t_{k}\bar{F}_{k}(x^{k}, \xi^{k})] \|^{2} \right] \\
        &\leq \mathbb{E} \left[ 2 \| \Pi_{X}[x^{k} - t_{k-1}F(x^{k})] - \Pi_{X}[x^{k} - t_{k}F(x^{k})]\|^{2} \:\vert\: t_{k}\neq t_{k-1} \right]\mathbb{P}\left\{ t_{k}\neq t_{k-1} \right\} \\
        &+ \mathbb{E} \left[ 2 \| \Pi_{X}[x^{k} - t_{k-1}F(x^{k})] - \Pi_{X}[x^{k} - t_{k}F(x^{k})]\|^{2} \:\vert\: t_{k} = t_{k-1} \right]\mathbb{P}\left\{ t_{k} = t_{k-1} \right\} \\
        &+ \mathbb{E} \left[ 2 \| \Pi_{X}[x^{k} - t_{k}F(x^{k})] - \Pi_{X}[x^{k} - t_{k}\bar{F}_{k}(x^{k}, \xi^{k})] \|^{2} \right] \\
        &\leq \mathbb{E} \left[ 2 \|(t_{k}-t_{k-1})F(x^{k})\|^{2} \:\vert\: t_{k}\neq t_{k-1} \right]\mathbb{P}\left\{ t_{k}\neq t_{k-1} \right\} + \mathbb{E} \left[ 2t^{2}_{k} \| F(x^{k}) - \bar{F}_{k}(x^{k}, \xi^{k}) \|^{2} \right] \\
        &\leq 2(t^{\mathrm{upp}}-t^{\mathrm{low}})^{2} D^{2}_{F}\mathbb{P}\left\{ t_{k}\neq t_{k-1} \right\} + 2(t^{\mathrm{upp}})^{2}\frac{\sigma^{2}}{N_{k}},
    \end{align*}
    where the penultimate inequality is due to nonexpansiveness of the projection operator, and the last inequality is due to Lemma \ref{SEG-stepsize} and Assumption $\mathrm{A}$. Therefore, we may bound (i) from above
    \begin{equation}\label{w3-bound-lemma-eqn4}
        (i) \leq 12(t^{\mathrm{upp}}-t^{\mathrm{low}})^{2} D^{2}_{F} L^{2} \mathbb{P}\left\{ t_{k}\neq t_{k-1} \right\} + 12 (t^{\mathrm{upp}})^{2} L^{2} \frac{\sigma^{2}}{N_{k}}.
    \end{equation}
    By combing \eqref{w3-bound-lemma-eqn3} with \eqref{w3-bound-lemma-eqn4}, we may arrive at \eqref{w3-bound-lemma-main-result}.
\end{proof}

The following lemma is crucial in our main theorem.

\begin{lemma}\label{sum-x_k+1-minus-x_star-lemma}\em
    Consider the sequence $\{x^{k}\}_{k=0}^{K-1}$ generated by \textbf{SEG-LS}. Suppose that Assumption $\mathrm{A}$ holds. Then, for any $k = 0, 1, \cdots, K-1$ and any $x^{\ast}\in \mathrm{SOL}\:(X, F)$, we have that
    \begin{equation*}
        \sum_{k=0}^{K-1} \mathbb{E} \left[ \| x^{k+1} - x^{\ast} \|^{2} \left( \frac{1}{t^{2}_{k+1}} - \frac{1}{t^{2}_{k}} \right) \right] \leq D^{2}_{X} \left( \frac{1}{(t^{\mathrm{low}})^{2}} - \frac{1}{(t^{\mathrm{upp}})^{2}} \right) (J(K)+1) \left\lceil \left[ \log_{1/\beta}\left(\frac{t^{\mathrm{upp}}L}{\alpha}\right) \right]_{+} + 3 \right\rceil.
    \end{equation*}
\end{lemma}
\begin{proof}
    For $k\geq 0$, we define three events, $I_{k,1} \triangleq \{ t_{k+1} > t_{k} \}$, $I_{k,2} \triangleq \{ t_{k+1} = t_{k} \}$, and $I_{k,3} \triangleq \{ t_{k+1} < t_{k} \}$, respectively. By the law of total expectation, we have that
    \begin{align*}
        &\mathbb{E} \left[ \| x^{k+1} - x^{\ast} \|^{2} \left( \frac{1}{t^{2}_{k+1}} - \frac{1}{t^{2}_{k}} \right) \right] \\
        &= \underbrace{\mathbb{E} \left[ \| x^{k+1} - x^{\ast} \|^{2} \left( \frac{1}{t^{2}_{k+1}} - \frac{1}{t^{2}_{k}} \right) \:\middle\vert\: I_{k,1} \right]}_{\leq 0} \mathbb{P}[I_{k,1}] + \underbrace{\mathbb{E} \left[ \| x^{k+1} - x^{\ast} \|^{2} \left( \frac{1}{t^{2}_{k+1}} - \frac{1}{t^{2}_{k}} \right) \:\middle\vert\: I_{k,2} \right]}_{= 0} \mathbb{P}[I_{k,2}] \\
        &+ \mathbb{E} \left[ \| x^{k+1} - x^{\ast} \|^{2} \left( \frac{1}{t^{2}_{k+1}} - \frac{1}{t^{2}_{k}} \right) \:\middle\vert\: I_{k,3} \right] \mathbb{P}[I_{k,3}] \\
        &\leq \mathbb{E} \left[ \| x^{k+1} - x^{\ast} \|^{2} \left( \frac{1}{t^{2}_{k+1}} - \frac{1}{t^{2}_{k}} \right) \:\middle\vert\: I_{k,3} \right] \mathbb{P}[I_{k,3}] \leq D^{2}_{X} \left( \frac{1}{(t^{\mathrm{low}})^{2}} - \frac{1}{(t^{\mathrm{upp}})^{2}} \right) \mathbb{P}\{t_{k+1}\neq t_{k}\},
    \end{align*}
    where the last inequality arises from Assumption $\mathrm{A}$, Lemma \ref{SEG-stepsize}, and the fact that $\mathbb{P}[I_{k,3}] \leq \mathbb{P}\{t_{k+1}\neq t_{k}\}$. By summing from $k = 0, \cdots, K-1$, we have that
    \begin{align*}
        &\sum_{k=0}^{K-1} \mathbb{E} \left[ \| x^{k+1} - x^{\ast} \|^{2} \left( \frac{1}{t^{2}_{k+1}} - \frac{1}{t^{2}_{k}} \right) \right]\\
        &\leq D^{2}_{X} \left( \frac{1}{(t^{\mathrm{low}})^{2}} - \frac{1}{(t^{\mathrm{upp}})^{2}} \right) \sum_{k=0}^{K-1} \mathbb{P}\{t_{k+1}\neq t_{k}\} = D^{2}_{X} \left( \frac{1}{(t^{\mathrm{low}})^{2}} - \frac{1}{(t^{\mathrm{upp}})^{2}} \right) \sum_{k=0}^{K-1} \mathbb{E}[\mathbf{1}_{\{ t_{k+1} \neq t_{k} \}}] \\
        &= D^{2}_{X} \left( \frac{1}{(t^{\mathrm{low}})^{2}} - \frac{1}{(t^{\mathrm{upp}})^{2}} \right) \sum_{k=1}^{K} \mathbb{E}[\mathbf{1}_{\{ t_{k} \neq t_{k-1} \}}] \leq D^{2}_{X} \left( \frac{1}{(t^{\mathrm{low}})^{2}} - \frac{1}{(t^{\mathrm{upp}})^{2}} \right) (\mathbb{E}[T(K)] + 1) \\
        &\leq D^{2}_{X} \left( \frac{1}{(t^{\mathrm{low}})^{2}} - \frac{1}{(t^{\mathrm{upp}})^{2}} \right) \left( (J(K)+1) \left\lceil \left[ \log_{1/\beta}\left(\frac{t^{\mathrm{upp}}L}{\alpha}\right) \right]_{+} + 2 \right\rceil + 1 \right) \\
        &\leq D^{2}_{X} \left( \frac{1}{(t^{\mathrm{low}})^{2}} - \frac{1}{(t^{\mathrm{upp}})^{2}} \right) (J(K)+1) \left\lceil \left[ \log_{1/\beta}\left(\frac{t^{\mathrm{upp}}L}{\alpha}\right) \right]_{+} + 3 \right\rceil,
    \end{align*}
    where the last inequality is due to the fact that $J(K)\geq 0$. This completes the proof.
\end{proof}

\subsubsection{Convergence analysis}

In this subsection, we will show that the expected norm of natural map diminishes at a rate of $\mathcal{O}(1/\sqrt{K})$, where $K$ is the iteration budget. Based on this result, the iteration and sample complexities of \textbf{SEG-LS} are shown to be $\mathcal{O}(\epsilon^{-2})$ and $\mathcal{O}(\epsilon^{-4})$, respectively. We begin with a useful proposition.

\begin{proposition}\label{SEG-key-recursion}\em
    Suppose Assumption $\mathrm{A}$ holds. Suppose that $\alpha \in (0, \frac{1}{\sqrt{3}})$. Then the sequence $\{ x^{k} \}_{k=0}^{K-1}$ generated by \textbf{SEG-LS} satisfies, for any $x^{\ast}\in \mathrm{SOL}\:(X, F)$ and $k\geq 0$,
    \begin{align}\label{SEG-key-recursion-eqn0}
        \| x^{k+1}-x^{\ast} \|^{2} &\leq \| x^{k}-x^{\ast} \|^{2} - \frac{1-3\alpha^{2}}{2}t^{2}_{k} \left\| G_{t^{\mathrm{upp}}}(x^{k}) \right\|^{2} + 2t_{k}(w^{k}_{2})^{\top}(x^{\ast}-z^{k}) \notag \\
        &+ (1-3\alpha^{2})t^{2}_{k}\|w^{k}_{1}\|^{2} + 3t^{2}_{k}\|w^{k}_{2}\|^{2} + 3t^{2}_{k}\|w^{k}_{3}\|^{2},
    \end{align}
    where $w^{k}_{1}$, $w^{k}_{2}$, and $w^{k}_{3}$ are defined in \eqref{w_k_1}-\eqref{w_k_3}, respectively.
\end{proposition}
\begin{proof}
    By invoking Lemma \ref{useful-lemma-SEG}, we have
    \begin{equation}\label{SEG-key-recursion-eqn1}
        2t_{k}\bar{F}_{k}(x^{k}, \xi^{k})^{\top}(z^{k}-x) \leq \| x^{k} - x \|^{2} - \| z^{k} - x \|^{2} - \| z^{k} - x^{k} \|^{2}, ~ \forall x\in X
    \end{equation}
    and
    \begin{equation}\label{SEG-key-recursion-eqn2}
        2t_{k}\bar{F}_{k}(z^{k}, \zeta^{k})^{\top}(x^{k+1}-x) \leq \| x^{k} - x \|^{2} - \| x^{k+1} - x \|^{2} - \| x^{k+1} - x^{k} \|^{2}, ~ \forall x\in X.
    \end{equation}
    By letting $x = x^{k+1}$ in \eqref{SEG-key-recursion-eqn1}, and then adding \eqref{SEG-key-recursion-eqn1} to \eqref{SEG-key-recursion-eqn2}, we have that
    \begin{align}\label{SEG-key-recursion-eqn3}
        \mathrm{Term}\;A &\triangleq 2t_{k}\bar{F}_{k}(x^{k}, \xi^{k})^{\top}(z^{k}-x^{k+1}) + 2t_{k}\bar{F}_{k}(z^{k}, \zeta^{k})^{\top}(x^{k+1}-x) \notag \\
        &\leq \| x^{k} - x \|^{2} - \| x^{k+1} - x \|^{2} - \| z^{k} - x^{k+1} \|^{2} - \| z^{k} - x^{k} \|^{2}.
    \end{align}
    We can see that $\mathrm{Term}\;A$ can be written as
    \begin{align}\label{SEG-key-recursion-eqn4}
        \mathrm{Term}\;A &= 2t_{k}(\bar{F}_{k}(x^{k}, \xi^{k})-\bar{F}_{k}(z^{k}, \zeta^{k}))^{\top}(z^{k}-x^{k+1}) + 2t_{k}\bar{F}_{k}(z^{k}, \zeta^{k})^{\top}(z^{k}-x) \notag \\
        &= 2t_{k}(\bar{F}_{k}(x^{k}, \xi^{k})-\bar{F}_{k}(z^{k}, \zeta^{k}))^{\top}(z^{k}-x^{k+1}) + 2t_{k}F(z^{k})^{\top}(z^{k}-x) + 2t_{k} (w^{k}_{2})^{\top}(z^{k}-x).
    \end{align}
    By \eqref{SEG-key-recursion-eqn3} and \eqref{SEG-key-recursion-eqn4}, it leads to
    \begin{equation*}
        \begin{aligned}
            2t_{k}F(z^{k})^{\top}(z^{k}-x) &\leq 2t_{k}(\bar{F}_{k}(x^{k}, \xi^{k})-\bar{F}_{k}(z^{k}, \zeta^{k}))^{\top}(x^{k+1}-z^{k}) + 2t_{k}(w^{k}_{2})^{\top}(x-z^{k}) \\
            &~~~+ \| x^{k} - x \|^{2} - \| x^{k+1} - x \|^{2} - \| z^{k} - x^{k+1} \|^{2} - \| z^{k} - x^{k} \|^{2} \\
            &\leq 2t_{k} \| \bar{F}_{k}(x^{k}, \xi^{k})-\bar{F}_{k}(z^{k}, \zeta^{k}) \| \| x^{k+1}-z^{k} \| + 2t_{k}(w^{k}_{2})^{\top}(x-z^{k}) \\
            &~~~+ \| x^{k} - x \|^{2} - \| x^{k+1} - x \|^{2} - \| z^{k} - x^{k+1} \|^{2} - \| z^{k} - x^{k} \|^{2} \\
            &\leq 2 \left( \frac{t^{2}_{k}}{2} \| \bar{F}_{k}(x^{k}, \xi^{k})-\bar{F}_{k}(z^{k}, \zeta^{k}) \|^{2} + \frac{\| x^{k+1} - z^{k} \|^{2}}{2} \right) + 2t_{k}(w^{k}_{2})^{\top}(x-z^{k}) \\
            &~~~+ \| x^{k} - x \|^{2} - \| x^{k+1} - x \|^{2} - \| z^{k} - x^{k+1} \|^{2} - \| z^{k} - x^{k} \|^{2},
        \end{aligned}
    \end{equation*}
    implying that
    \begin{align}\label{SEG-key-recursion-eqn5}
         2t_{k}F(z^{k})^{\top}(z^{k}-x) \notag  & \leq t^{2}_{k}\| \bar{F}_{k}(x^{k}, \xi^{k})-\bar{F}_{k}(z^{k}, \zeta^{k}) \|^{2} 
            + 2t_{k}(w^{k}_{2})^{\top}(x-z^{k})\\
                &~~~+ \| x^{k} - x \|^{2} - \| x^{k+1} - x \|^{2} - \| z^{k} - x^{k} \|^{2}.
    \end{align}
    Note that we cannot directly derive a bound on $t^{2}_{k}\| \bar{F}_{k}(x^{k}, \xi^{k})-\bar{F}_{k}(z^{k}, \zeta^{k}) \|^{2}$ since these use different samples. Instead, we have that
    \begin{align}\label{SEG-key-recursion-eqn6}
        & t^{2}_{k} \| \bar{F}_{k}(x^{k}, \xi^{k})-\bar{F}_{k}(z^{k}, \zeta^{k}) \|^{2} \notag \\
        &= t^{2}_{k} \| \bar{F}_{k}(x^{k}, \xi^{k}) - \bar{F}_{k}(z^{k}, \xi^{k}) + \bar{F}_{k}(z^{k}, \xi^{k}) - F(z^{k}) + F(z^{k}) - \bar{F}_{k}(z^{k}, \zeta^{k}) \|^{2} \notag \\
        &\leq  t^{2}_{k} ( 3\| \bar{F}_{k}(x^{k}, \xi^{k}) - \bar{F}_{k}(z^{k}, \xi^{k}) \|^{2} + 3\| \bar{F}_{k}(z^{k}, \xi^{k}) - F(z^{k}) \|^{2} + 3\| F(z^{k}) - \bar{F}_{k}(z^{k}, \zeta^{k}) \|^{2} ) \notag \\
        &= t^{2}_{k} ( 3\| \bar{F}_{k}(x^{k}, \xi^{k}) - \bar{F}_{k}(z^{k}, \xi^{k}) \|^{2} + 3\| w^{k}_{3} \|^{2} + 3\| w^{k}_{2} \|^{2} ) \notag \\
        &\overset{\eqref{SEG-Armijo}}{\leq} t^{2}_{k} \left( 3\cdot \frac{\alpha^{2}}{t^{2}_{k}}\|x^{k} - z^{k}\|^{2} + 3\| w^{k}_{3} \|^{2} + 3\| w^{k}_{2} \|^{2} \right) \notag \\
        &= 3\alpha^{2}\|x^{k} - z^{k}\|^{2} + 3t^{2}_{k}\| w^{k}_{2} \|^{2} + 3t^{2}_{k}\| w^{k}_{3} \|^{2}.
    \end{align}
    %where the last inequality is due to \eqref{SEG-Armijo}. 
    By choosing $x = x^{\ast}\in X^{\ast}$ and noting that $F$ is pseudomonotone by assumption $\mathrm{(A1)}$, we have that 
    \begin{equation}\label{SEG-key-recursion-eqn7}
        F(x^{\ast})^{\top}(z^{k}-x^{\ast})\, \geq \, 0 \, \implies \, F(z^{k})^{\top}(z^{k}-x^{\ast})\geq 0.
    \end{equation}
    By \eqref{SEG-key-recursion-eqn5}, \eqref{SEG-key-recursion-eqn6}, and \eqref{SEG-key-recursion-eqn7}, it follows that
    \begin{align}\label{SEG-key-recursion-eqn8}
        \| x^{k+1} - x^{\ast} \|^{2} &\leq \| x^{k} - x^{\ast} \|^{2} - \| x^{k} - z^{k} \|^{2} -2t_k F(z^{k})^\top(z^{k}-x^\ast) + 2t_{k}(w^{k}_{2})^{\top}(x^{\ast}-z^{k}) \notag \\
        &~~~ + 3\alpha^{2}\|x^{k} - z^{k}\|^{2} + 3t^{2}_{k}\| w^{k}_{2} \|^{2} + 3t^{2}_{k}\| w^{k}_{3} \|^{2} \notag \\
         &\leq \| x^{k} - x^{\ast} \|^{2} - \| x^{k} - z^{k} \|^{2} + 2t_{k}(w^{k}_{2})^{\top}(x^{\ast}-z^{k}) + 3\alpha^{2}\|x^{k} - z^{k}\|^{2} + 3t^{2}_{k}\| w^{k}_{2} \|^{2} + 3t^{2}_{k}\| w^{k}_{3} \|^{2} \notag \\
        &= \| x^{k} - x^{\ast} \|^{2} - (1-3\alpha^{2}) \| x^{k} - z^{k} \|^{2} + 2t_{k}(w^{k}_{2})^{\top}(x^{\ast}-z^{k}) + 3t^{2}_{k}(\| w^{k}_{2} \|^{2}+\| w^{k}_{3} \|^{2}).
    \end{align}
    By \cite[Lemma 5]{wang-shanbhag-xie-2025}, we know that $\| G_{t^{\mathrm{upp}}}(x) \| \leq \| G_{t}(x) \|$ holds for any $t\leq t^{\mathrm{upp}}$. Therefore, by Lemma \ref{SEG-stepsize}, we have that
    \begin{align*}
            t^{2}_{k} \left\| G_{t^{\mathrm{upp}}}(x^{k}) \right\|^{2} &\leq t^{2}_{k} \left\| G_{t_{k}}(x^{k}) \right\|^{2} = \left\| x^{k} - \Pi_{X}\left[x^{k} - t_{k}F(x^{k})\right] \right\|^{2} \\
            &= \left\| x^{k} - \underbrace{\Pi_{X}[x^{k} - t_{k}\bar{F}_{k}(x^{k}, \xi^{k})]}_{=z^{k}} + \Pi_{X}\left[x^{k} - t_{k}\bar{F}_{k}(x^{k}, \xi^{k})\right] - \Pi_{X}\left[x^{k} - t_{k}F(x^{k})\right] \right\|^{2} \\
            &\leq 2 \| x^{k} - z^{k} \|^{2} + 2t^{2}_{k} \| \bar{F}_{k}(x^{k}, \xi^{k}) - F(x^{k}) \|^{2} = 2 \| x^{k} - z^{k} \|^{2} + 2t^{2}_{k} \| w^{k}_{1} \|^{2},
        \end{align*}
    where the penultimate inequality is due to nonexpansiveness of the projection operator, implying that
    \begin{equation}\label{SEG-key-recursion-eqn9}
        -\| x^{k} - z^{k} \|^{2} \, \leq \, -\frac{1}{2} t^{2}_{k} \left\| G_{t^{\mathrm{upp}}}(x^{k}) \right\|^{2} + t^{2}_{k} \| w^{k}_{1} \|^{2}.
    \end{equation}
    Since $\alpha\in \left(0, \frac{1}{\sqrt{3}}\right)$, we have that $1 - 3\alpha^{2} > 0$. By \eqref{SEG-key-recursion-eqn8} and \eqref{SEG-key-recursion-eqn9}, it follows that
    \begin{equation*}
        \begin{aligned}
            \| x^{k+1}-x^{\ast} \|^{2} &\leq \| x^{k}-x^{\ast} \|^{2} - \frac{1-3\alpha^{2}}{2} t^{2}_{k} \left\| G_{t^{\mathrm{upp}}}(x^{k}) \right\|^{2} + 2t_{k}(w^{k}_{2})^{\top}(x^{\ast}-z^{k}) \\
            &~~~+ (1-3\alpha^{2})t^{2}_{k}\|w^{k}_{1}\|^{2} + 3t^{2}_{k}\|w^{k}_{2}\|^{2} + 3t^{2}_{k}\|w^{k}_{3}\|^{2},
        \end{aligned}
    \end{equation*}
    which completes the proof.
\end{proof}

To proceed, we recall the definition of $\mathcal{F}_{k+1/2}$ from \eqref{SEG-history-2}, whereby $\mathbb{E} [ w^{k}_{2} \:\vert\: \mathcal{F}_{k+1/2} ] = 0$ since $w^{k}_{2} \triangleq \bar{F}_{k}(z^{k}, \zeta^{k}) - F(z^{k})$.

\begin{proposition}\label{SEG-main-proposition}\em
    Consider the sequence $\{x^{k}\}_{k=0}^{K-1}$ generated by \textbf{SEG-LS}. Suppose Assumption $\mathrm{A}$ holds. Suppose that $\alpha \in (0, \frac{1}{\sqrt{3}})$. Then the following holds for any $K$, where $x^{R_K}$ is defined in Algorithm~\ref{SEG-LS}.
    \begin{align}\label{SEG-main-proposition-main-result}
        &\frac{1-3\alpha^{2}}{2} \mathbb{E} \left[ \left\|G_{t^{\mathrm{upp}}}(x^{R_{K}}) \right\|^{2} \right] \leq \frac{D^{2}_{X}}{K(t^{\mathrm{low}})^{2}} + \frac{1}{K} (36(t^{\mathrm{upp}})^{2}L^{2}+13-3\alpha^{2}) \sum_{k=0}^{K-1} \frac{\sigma^{2}}{N_{k}} \notag \\
        &+ \frac{1}{K} \left( D^{2}_{X} \left( \frac{1}{(t^{\mathrm{low}})^{2}} - \frac{1}{(t^{\mathrm{upp}})^{2}} \right) + 36(t^{\mathrm{upp}}-t^{\mathrm{low}})^{2}D^{2}_{F}L^{2} \right) (J(K)+1) \left\lceil \left[ \log_{1/\beta}\left( \frac{t^{\mathrm{upp}}L}{\alpha} \right) \right]_{+} + 3 \right\rceil.
    \end{align}
\end{proposition}
\begin{proof}
    By dividing both sides of \eqref{SEG-key-recursion-eqn0} by $t^{2}_{k}$ and then rearranging terms, it leads to
    \begin{align}\label{SEG-main-proposition-eqn0}
        \frac{1-3\alpha^{2}}{2} \left\| G_{t^{\mathrm{upp}}}(x^{k}) \right\|^{2} 
        &\leq \frac{\|x^{k}-x^{\ast}\|^{2}}{t^{2}_{k}} - \frac{\|x^{k+1}-x^{\ast}\|^{2}}{t^{2}_{k}} \notag \\
        & ~~~ + 2 \frac{1}{t_{k}} (w^{k}_{2})^{\top}(x^{\ast}-z^{k}) + (1-3\alpha^{2})\|w^{k}_{1}\|^{2}+ 3 \|w^{k}_{2}\|^{2} + 3\|w^{k}_{3}\|^{2} \notag \\
        &=  \frac{\|x^{k}-x^{\ast}\|^{2}}{t^{2}_{k}} - \frac{\|x^{k+1}-x^{\ast}\|^{2}}{t^{2}_{k+1}} + \frac{\|x^{k+1}-x^{\ast}\|^{2}}{t^{2}_{k+1}} - \frac{\|x^{k+1}-x^{\ast}\|^{2}}{t^{2}_{k}} \notag \\
        &~~~ + 2 \frac{1}{t_{k}} (w^{k}_{2})^{\top}(x^{\ast}-z^{k}) + (1-3\alpha^{2}) \|w^{k}_{1}\|^{2} + 3 \|w^{k}_{2}\|^{2} + 3\|w^{k}_{3}\|^{2}.
    \end{align}
    Since $t_{k}$ are $z^{k}$ are adapted to $\mathcal{F}_{k+1/2}$, we may observe that
    \begin{align}\label{SEG-main-proposition-eqn0.5}
        \mathbb{E} \left[ 2 \frac{1}{t_{k}} (w^{k}_{2})^{\top}(x^{\ast}-z^{k}) \right] &= \mathbb{E} \left[ \mathbb{E} \left[ 2 \frac{1}{t_{k}} (w^{k}_{2})^{\top}(x^{\ast}-z^{k}) \:\middle\vert\: \mathcal{F}_{k+1/2} \right] \right] \notag \\
        &= \mathbb{E} \left[ 2 \frac{1}{t_{k}} \mathbb{E} \left[ w^{k}_{2} \:\middle\vert\: \mathcal{F}_{k+1/2} \right]^{\top}(x^{\ast}-z^{k}) \right] = 0.
    \end{align}
    Therefore, by taking unconditional expectations on both sides of \eqref{SEG-main-proposition-eqn0}, invoking \eqref{wk_1-wk_2-bounds} and \eqref{w3-bound-lemma-main-result}, we obtain that
    \begin{align*}
        &\frac{1-3\alpha^{2}}{2} \mathbb{E} \left[ \left\| G_{t^{\mathrm{upp}}}(x^{k}) \right\|^{2} \right] \leq \mathbb{E} \left[ \frac{\|x^{k}-x^{\ast}\|^{2}}{t^{2}_{k}} - \frac{\|x^{k+1}-x^{\ast}\|^{2}}{t^{2}_{k+1}} \right] + \mathbb{E} \left[ \frac{\|x^{k+1}-x^{\ast}\|^{2}}{t^{2}_{k+1}} - \frac{\|x^{k+1}-x^{\ast}\|^{2}}{t^{2}_{k}} \right] \\
        &+ (4-3\alpha^{2})\frac{\sigma^{2}}{N_{k}} + (36(t^{\mathrm{upp}})^{2}L^{2}+9)\frac{\sigma^{2}}{N_{k}} + 36(t^{\mathrm{upp}}-t^{\mathrm{low}})^{2}D^{2}_{F}L^{2}\mathbb{P}\{t_{k}\neq t_{k-1}\} \\
        &= \mathbb{E} \left[ \frac{\|x^{k}-x^{\ast}\|^{2}}{t^{2}_{k}} - \frac{\|x^{k+1}-x^{\ast}\|^{2}}{t^{2}_{k+1}} \right] + \mathbb{E} \left[ \frac{\|x^{k+1}-x^{\ast}\|^{2}}{t^{2}_{k+1}} - \frac{\|x^{k+1}-x^{\ast}\|^{2}}{t^{2}_{k}} \right] \\
        &+ (36(t^{\mathrm{upp}})^{2}L^{2}+13-3\alpha^{2})\frac{\sigma^{2}}{N_{k}} + 36(t^{\mathrm{upp}}-t^{\mathrm{low}})^{2}D^{2}_{F}L^{2}\mathbb{P}\{t_{k}\neq t_{k-1}\}.
    \end{align*}
    By taking the sum over $k = 0, \cdots, K-1$ and dividing both sides by $K$, it follows that
    \begin{align}\label{SEG-main-proposition-eqn1}
        &\frac{1-3\alpha^{2}}{2} \frac{1}{K} \sum_{k=0}^{K-1} \mathbb{E} \left[\left\| G_{t^{\mathrm{upp}}}(x^{k}) \right\|^{2}\right] \notag \\
        &\leq \underbrace{\frac{1}{K} \sum_{k=0}^{K-1} \mathbb{E} \left[ \frac{\|x^{k}-x^{\ast}\|^{2}}{t^{2}_{k}} - \frac{\|x^{k+1}-x^{\ast}\|^{2}}{t^{2}_{k+1}} \right]}_{(i)} + \underbrace{\frac{1}{K} \sum_{k=0}^{K-1} \mathbb{E} \left[ \frac{\|x^{k+1}-x^{\ast}\|^{2}}{t^{2}_{k+1}} - \frac{\|x^{k+1}-x^{\ast}\|^{2}}{t^{2}_{k}} \right]}_{(ii)} \notag \\
        &+ \frac{1}{K}(36(t^{\mathrm{upp}})^{2}L^{2}+13-3\alpha^{2})\sum_{k=0}^{K-1}\frac{\sigma^{2}}{N_{k}} + \underbrace{\frac{1}{K} 36(t^{\mathrm{upp}}-t^{\mathrm{low}})^{2}D^{2}_{F}L^{2} \sum_{k=0}^{K-1} \mathbb{P}\{t_{k}\neq t_{k-1}\}}_{(iii)}.
    \end{align}
    By telescoping, we can see that
    \begin{equation*}
        (i) = \frac{1}{K} \mathbb{E} \left[ \frac{\|x^{0}-x^{\ast}\|^{2}}{t^{2}_{0}} - \frac{\|x^{K}-x^{\ast}\|^{2}}{t^{2}_{K}} \right] \leq \frac{1}{K} \mathbb{E} \left[ \frac{\|x^{0}-x^{\ast}\|^{2}}{t^{2}_{0}} \right] \leq \frac{\mathbb{E}[\|x^0-x^\ast\|^2]}{K (t^{\mathrm{low}})^{2}} \leq \frac{D_{X}^2}{K (t^{\mathrm{low}})^{2}}.
    \end{equation*}
    By Lemma \ref{sum-x_k+1-minus-x_star-lemma}, we have
    \begin{equation*}
        (ii) \leq \frac{1}{K} D^{2}_{X} \left( \frac{1}{(t^{\mathrm{low}})^{2}} - \frac{1}{(t^{\mathrm{upp}})^{2}} \right) (J(K)+1) \left\lceil \left[ \log_{1/\beta}\left(\frac{t^{\mathrm{upp}}L}{\alpha}\right) \right]_{+} + 3 \right\rceil.
    \end{equation*}
    We may bound (iii) from above by noting that
    \begin{align*}
        (iii) &= \frac{36(t^{\mathrm{upp}}-t^{\mathrm{low}})^{2}D^{2}_{F}L^{2}}{K} \sum_{k=0}^{K-1} \mathbb{P}\{t_{k}\neq t_{k-1}\} = \frac{36(t^{\mathrm{upp}}-t^{\mathrm{low}})^{2}D^{2}_{F}L^{2}}{K} \sum_{k=0}^{K-1} \mathbb{E}[\mathbf{1}_{\{ t_{k}\neq t_{k-1} \}}] \\
        &= \frac{36(t^{\mathrm{upp}}-t^{\mathrm{low}})^{2}D^{2}_{F}L^{2}}{K} \mathbb{E}[T(K)] \leq \frac{36(t^{\mathrm{upp}}-t^{\mathrm{low}})^{2}D^{2}_{F}L^{2}}{K} (J(K)+1) \left\lceil \left[ \log_{1/\beta}\left(\frac{t^{\mathrm{upp}}L}{\alpha}\right) \right]_{+} + 2 \right\rceil \\
        &\leq \frac{36(t^{\mathrm{upp}}-t^{\mathrm{low}})^{2}D^{2}_{F}L^{2}}{K} (J(K)+1) \left\lceil \left[ \log_{1/\beta}\left(\frac{t^{\mathrm{upp}}L}{\alpha}\right) \right]_{+} + 3 \right\rceil,
    \end{align*}
    where the penultimate inequality is due to Lemma \ref{bound-TK-lemma}. By plugging the upper bounds of $(i)$, $(ii)$, and $(iii)$ into \eqref{SEG-main-proposition-eqn1} and recalling the definition of $R_{K}$, we may arrive the final conclusion \eqref{SEG-main-proposition-main-result}. 
\end{proof}

Next, by specifying the parameters $p_{j}$ and $N_{k}$, we derive a nonasymptotic sublinear rate in terms of $K$ and the associated complexity guarantees.

\begin{theorem}\label{SEG-main-theorem}\em
    Consider the sequence generated by \textbf{SEG-LS}. Suppose Assumption $\mathrm{A}$ holds. Suppose that $\alpha \in (0, \frac{1}{\sqrt{3}})$. For a given $K$, we set $p_{j} \triangleq \left\lceil c_{p}K \right\rceil$ and $N_{k} \triangleq \left\lceil c_{N}K \right\rceil$, where $c_{p}\in (0, 1)$ and $c_{N} > 0$. Then the following hold. \\
    (i) We have that
    \begin{align}\label{SEG-main-theorem-eqn1}
        &\mathbb{E} \left[ \left\| G_{t^{\mathrm{upp}}}(x^{R_{K}}) \right\|^{2} \right] \leq \frac{2D^{2}_{X}}{(1-3\alpha^{2})(t^{\mathrm{low}})^{2}K} + \frac{2(36(t^{\mathrm{upp}})^{2}L^{2}+13-3\alpha^{2})\sigma^{2}}{(1-3\alpha^{2})c_{N}K} \notag \\
        &+ \frac{2\left( D^{2}_{X} \left( \frac{1}{(t^{\mathrm{low}})^{2}} - \frac{1}{(t^{\mathrm{upp}})^{2}} \right) + 36(t^{\mathrm{upp}}-t^{\mathrm{low}})^{2}D^{2}_{F}L^{2} \right)(c_{p}+1)\left\lceil \left[ \log_{1/\beta}\left(\frac{t^{\mathrm{upp}}L}{\alpha}\right) \right]_{+} + 3 \right\rceil}{(1-3\alpha^{2})c_{p}K} \notag \\
        &= \mathcal{O} \left( \frac{ D^{2}_{X} + (t^{\mathrm{upp}})^{2}L^{2}\sigma^{2} + \left( D^{2}_{X} \left( \frac{1}{(t^{\mathrm{low}})^{2}} - \frac{1}{(t^{\mathrm{upp}})^{2}} \right) + (t^{\mathrm{upp}}-t^{\mathrm{low}})^{2}D^{2}_{F}L^{2} \right) [\log{(t^{\mathrm{upp}}L)}]_{+} }{(t^{\mathrm{low}})^{2}K} \right).
    \end{align} 
    (ii) The iteration complexity $K_{\epsilon}$ and the sample complexity $S_{\epsilon} = \sum_{k=0}^{K_{\epsilon}-1} N_{k}$ that guarantee $\mathbb{E}[\|G_{t^{\mathrm{upp}}}(x^{R_{K}})\|]\leq \epsilon$ are $\mathcal{O}(\epsilon^{-2})$ and $\mathcal{O}(\epsilon^{-4})$, respectively.
\end{theorem}
\begin{proof}
    (i) By \cite[Lemmas 6 \& 7]{wang-shanbhag-xie-2025}, we know that $\sum_{k=0}^{K-1}\frac{1}{N_{k}}\leq \frac{1}{c_{N}}$ and $J(K)+1 \leq \frac{1}{c_{p}}+1$. Then the derivation of \eqref{SEG-main-theorem-eqn1} follows from Proposition \ref{SEG-main-proposition} directly. (ii) The expression of $K_{\epsilon}$ is a direct result of \eqref{SEG-main-theorem-eqn1}, and the expression of $S_{\epsilon}$ follows from $S_{\epsilon} = \sum_{k=0}^{K_{\epsilon}-1} N_{k}$ directly.
\end{proof}

\begin{remark}\em
    We make two remarks here. First, our new analysis is significantly simpler than that of Iusem et al. \cite{iusem-jofre-oliveira-thompson-2019}, obviating the heavier machinery of empirical process theory but instead relying only on standard conditioning arguments. This has the additional benefit of allowing for generalization to other settings. Second, by prescribing a constant stepsize with $t^{\mathrm{upp}} = t^{\mathrm{low}} \equiv \mathcal{O}(1/L)$, we derive from \eqref{SEG-main-theorem-eqn1} a convergence rate of $\mathcal{O}((D_X^2L^2+\sigma^2)/K)$, which matches the non-adaptive regime in Iusem et al. \cite{iusem-jofre-oliveira-thompson-2017} after rescaling their residual measure $r_t(x)\triangleq\|x-\Pi_X[x-tF(x)]\|$ by $1/t$. $\hfill \Box$
\end{remark}

\subsection{SEG-LS for monotone SVIs}

In this subsection, we consider the setting where the mapping $F$ is monotone over $X$, i.e., $(F(x)-F(y))^{\top}(x-y)\geq 0$ for any $x, y \in X$. Based on Proposition \ref{weak-gap-function-property}, we will show that the expected dual gap function diminishes at an improved rate of $\mathcal{O}(1/K)$, where $K$ is the iteration budget. Therefore, the resulting iteration and sample complexities are $\mathcal{O}(\epsilon^{-1})$ and $\mathcal{O}(\epsilon^{-2})$, respectively.

\begin{proposition}\label{SEG-monotone-proposition}\em
    Suppose Assumption $\mathrm{A}$ holds, except that the pseudomonotonicity assumption $\mathrm{(A1)}$ is replaced by monotonicity of $F$ on $X$. Then the sequence $\{ z^{k} \}_{k=0}^{K-1}$ generated by \textbf{SEG-LS} satisfies
    \begin{align}\label{SEG-monotone-main-result}
        \mathbb{E}[\theta(z^{R_{K}})] &\leq \frac{1}{2K} \left( \frac{1}{t^{\mathrm{low}}} D^{2}_{X} + D^{2}_{X} \right) + \frac{1}{2K} ( 3t^{\mathrm{upp}} + 1 ) \sum_{k=0}^{K-1} \frac{\sigma^{2}}{N_{k}} + \frac{1}{2K} ( 36(t^{\mathrm{upp}})^{3}L^{2} + 9t^{\mathrm{upp}} ) \sum_{k=0}^{K-1} \frac{\sigma^{2}}{N_{k}} \notag \\
        &+ \frac{1}{2K} D^{2}_{X} \left( \frac{1}{t^{\mathrm{low}}} - \frac{1}{t^{\mathrm{upp}}} \right) (J(K)+1) \left\lceil \left[ \log_{1/\beta}\left(\frac{t^{\mathrm{upp}}L}{\alpha}\right) \right]_{+} + 2 \right\rceil \notag \\
        &+ \frac{1}{2K} \cdot 36 t^{\mathrm{upp}} (t^{\mathrm{upp}}-t^{\mathrm{low}})^{2} D^{2}_{F} L^{2} (J(K)+1) \left\lceil \left[ \log_{1/\beta}\left(\frac{t^{\mathrm{upp}}L}{\alpha}\right) \right]_{+} + 2 \right\rceil,
    \end{align}
    where $z^{R_{K}}$ is a random variable uniformly distributed on $\{ z^{0}, z^{1}, \cdots, z^{K-1} \}$.
\end{proposition}
\begin{proof}
   By \eqref{SEG-key-recursion-eqn5} and \eqref{SEG-key-recursion-eqn6}, we arrive that  
   \begin{align}\label{SEG-monotone-proposition-eqn1}
        2 t_{k} F(z^{k})^\top(z^{k}-x) &\leq \| x^{k} - x \|^{2} - \| x^{k+1} - x \|^{2}- (1-3\alpha^{2}) \| x^{k} - z^{k} \|^{2} \notag \\
        &~~~+ 2t_{k}(w^{k}_{2})^{\top}(x-z^{k}) + 3t^{2}_{k}(\| w^{k}_{2} \|^{2}+\| w^{k}_{3} \|^{2}).
    \end{align}
    By leveraging monotonicity, we have that
    \begin{align}\label{SEG-monotone-proposition-eqn2}
        2 t_{k} F(z^{k})^\top(z^{k}-x) = \underbrace{2 t_{k} (F(z^{k}) - F(x))^\top(z^{k}-x)}_{\geq 0} \, + \,  2t_k F(x)^{\top}(z^{k}-x) \geq 2t_k F(x)^{\top}(z^{k}-x). 
    \end{align}
    It follows from \eqref{SEG-monotone-proposition-eqn1} and \eqref{SEG-monotone-proposition-eqn2}, together with $\alpha \in (0, \frac{1}{\sqrt{3}})$, that
    \begin{align*}
        2t_k F(x)^\top(z^{k}-x) &\leq  \| x^{k} - x \|^{2} - \| x^{k+1} - x \|^{2} - \underbrace{(1-3\alpha^{2}) \| x^{k} - z^{k} \|^{2}}_{\geq 0} \notag \\
        &~~~+ 2t_{k}(w^{k}_{2})^{\top}(x-z^{k}) + 3t^{2}_{k}(\| w^{k}_{2} \|^{2}+\| w^{k}_{3} \|^{2}),
    \end{align*}
    which implies that
    \begin{equation*}
        2 F(x)^\top(z^{k}-x) \leq  \frac{1}{t_k}\| x^{k} - x \|^{2} - \frac{1}{t_k}\| x^{k+1} - x \|^{2} + 2(w^{k}_{2})^{\top}(x-z^{k}) + 3t_{k}(\| w^{k}_{2} \|^{2}+\| w^{k}_{3} \|^{2})
    \end{equation*}
    since $t_{k} \geq t^{\mathrm{low}} > 0$. Summing the above inequality from $k = 0, \hdots, K-1$, we obtain
    \begin{align*}
      & 2 \sum_{k=0}^{K-1} F(x)^\top(z^{k}-x) \\
      & \leq \sum_{k=0}^{K-1} \left(\frac{1}{t_k}\| x^{k} - x \|^{2} - \frac{1}{t_k}\| x^{k+1} - x \|^{2} \right) + \sum_{k=0}^{K-1} \left( 2(w^{k}_{2})^{\top}(x-z^{k}) + 3t_{k}(\| w^{k}_{2} \|^{2}+\| w^{k}_{3} \|^{2}) \right) \\
      &\leq \frac{1}{t_0}\| x^{0} - x \|^{2} + \sum_{k=0}^{K-2} \left( \frac{1}{t_{k+1}}- \frac{1}{t_k}\right)\| x^{k+1} - x \|^{2} + \sum_{k=0}^{K-1}\left(2(w^{k}_{2})^{\top}(x-z^{k}) + 3t_{k}(\| w^{k}_{2} \|^{2}+\| w^{k}_{3} \|^{2})\right).
    \end{align*}
    By defining $z^{R_{K}} \triangleq \frac{1}{K} \sum_{k=0}^{K-1}z^{k}$, we obtain 
    \begin{align}\label{SEG-monotone-proposition-eqn3}
        & 2 K F(x)^{\top} (z^{R_{K}}-x) \notag \\
        &\leq \frac{1}{t_0}\| x^{0} - x \|^{2} + \sum_{k=0}^{K-2} \left( \frac{1}{t_{k+1}}- \frac{1}{t_k}\right)\| x^{k+1} - x \|^{2} + \sum_{k=0}^{K-1} \left( 2(w^{k}_{2})^{\top}(x-z^{k}) + 3t_{k}(\| w^{k}_{2} \|^{2}+\| w^{k}_{3} \|^{2}) \right).
    \end{align}
    To derive an upper bound on $\sum_{k=0}^{K-1} 2(w^{k}_{2})^{\top}(x-z^{k})$, we define the sequence $\{ u^{k} \}_{k = 0}^{K}$ as follows. Choose any $u^0 \in X$ and let $u^{k+1} \triangleq \Pi_{X} [ u^{k} + w^{k}_{2} ]$ for each $k \ge 0$. Then we have  
    \begin{align*} 
        \| u^{k+1} - x \|^2 & = \left\| \Pi_{X}[ u^{k} + w^{k}_{2} ] - \Pi_X[x]\right\|^2 \leq \left\| u^{k} - x + w^{k}_{2} \right\|^2 \\ 
        & = \| u^k - x \|^2 + 2 (w^{k}_{2})^\top ( u^k - x ) + \| w^{k}_{2} \|^2,
    \end{align*}
    implying that
    \begin{equation*}
       -2 (w^{k}_{2})^\top ( u^k - x ) \leq \| u^k - x \|^2 - \| u^{k+1} - x \|^2 + \| w^{k}_{2} \|^2.
    \end{equation*}
    Consequently, we may obtain
    \begin{align*}
        (2w^{k}_{2})^{\top}(x-z^{k}) &= (2w^{k}_{2})^{\top}(x-u^{k}) + (2w^{k}_{2})^{\top}(u^{k}-z^{k}) \\
        & \leq \| u^k - x \|^2 - \| u^{k+1} - x \|^2 + \| w^{k}_{2} \|^2 + (2w^{k}_{2})^{\top}(u^{k}-z^{k})
    \end{align*}
    and hence, summing over $k = 0, \cdots, K-1$, we obtain that
    \begin{equation}\label{SEG-monotone-proposition-eqn4}
        \sum_{k=0}^{K-1} 2(w^{k}_{2})^{\top}(x-z^{k}) \leq \| u^{0} - x \|^2 + \sum_{k=0}^{K-1} \| w^{k}_{2} \|^2 + \sum_{k=0}^{K-1} (2w^{k}_{2})^{\top}(u^{k}-z^{k}).
    \end{equation}
    By \eqref{SEG-monotone-proposition-eqn3} and \eqref{SEG-monotone-proposition-eqn4}, we may then obtain
    \begin{align*}
        2 K F(x)^{\top} (z^{R_{K}}-x) &\leq \frac{1}{t_0}\| x^{0} - x \|^{2} + \sum_{k=0}^{K-2} \left( \frac{1}{t_{k+1}}- \frac{1}{t_k}\right)\| x^{k+1} - x \|^{2} \\
        &+ \| u^{0} - x \|^2 + \sum_{k=0}^{K-1} \| w^{k}_{2} \|^2 + \sum_{k=0}^{K-1} (2w^{k}_{2})^{\top}(u^{k}-z^{k}) + \sum_{k=0}^{K-1} 3t_{k}(\| w^{k}_{2} \|^{2}+\| w^{k}_{3} \|^{2}).
    \end{align*}
    Taking the supremum over $X$, we obtain
    \begin{align*}
         2K \theta(z^{R_{K}}) &\leq \sup_{x\in X} \frac{1}{t_0} \| x^{0} - x \|^{2} + \sup_{x\in X} \left[ \sum_{k=0}^{K-2} \left( \frac{1}{t_{k+1}}- \frac{1}{t_k}\right)\| x^{k+1} - x \|^{2} \right] \\
         &+ \sup_{x\in X} \| u^{0} - x \|^2 + \sum_{k=0}^{K-1} \| w^{k}_{2} \|^2 + \sum_{k=0}^{K-1} (2w^{k}_{2})^{\top}(u^{k}-z^{k}) + \sum_{k=0}^{K-1} 3t_{k}(\| w^{k}_{2} \|^{2}+\| w^{k}_{3} \|^{2}).
    \end{align*}
    Then, by taking expectations and dividing both sides by $2K$, we obtain
    \begin{align}\label{SEG-monotone-proposition-eqn5}
        &\mathbb{E}[\theta(z^{R_{K}})] \leq \frac{1}{2K} \sup_{x\in X} \frac{1}{t_0} \| x^{0} - x \|^{2} + \frac{1}{2K} \underbrace{\mathbb{E} \left[ \sup_{x\in X} \left[ \sum_{k=0}^{K-2} \left( \frac{1}{t_{k+1}}- \frac{1}{t_k}\right)\| x^{k+1} - x \|^{2} \right] \right]}_{(i)} \notag \\
        &+ \frac{1}{2K} \sup_{x\in X} \| u^{0} - x \|^2 + \frac{1}{2K} \underbrace{\sum_{k=0}^{K-1} \mathbb{E} [ (2w^{k}_{2})^{\top}(u^{k}-z^{k}) ]}_{(ii)} + \frac{1}{2K} \sum_{k=0}^{K-1} (3t_{k} + 1) \mathbb{E} [\| w^{k}_{2} \|^{2}] + \frac{1}{2K} \underbrace{\sum_{k=0}^{K-1} 3t_{k} \mathbb{E} [\| w^{k}_{3} \|^{2}]}_{(iii)}.
    \end{align}
    Similar to \eqref{SEG-main-proposition-eqn0.5}, we may derive that $(ii) = 0$. For $(i)$, similar to the proof of Lemma \ref{sum-x_k+1-minus-x_star-lemma}, it follows from the law of total expectation that
    \begin{align}\label{SEG-monotone-proposition-eqn6}
        (i) &\leq \mathbb{E} \left[ \sum_{k=0}^{K-2} \sup_{x\in X} \left[ \left( \frac{1}{t_{k+1}}- \frac{1}{t_k}\right)\| x^{k+1} - x \|^{2} \right] \right] = \sum_{k=0}^{K-2} \mathbb{E} \left[ \sup_{x\in X} \left[ \left( \frac{1}{t_{k+1}}- \frac{1}{t_k}\right)\| x^{k+1} - x \|^{2} \right] \right] \notag \\
        &\leq \sum_{k=0}^{K-2} D^{2}_{X} \left( \frac{1}{t^{\mathrm{low}}} - \frac{1}{t^{\mathrm{upp}}} \right) \mathbb{P}\{ t_{k+1} \neq t_{k} \} = D^{2}_{X} \left( \frac{1}{t^{\mathrm{low}}} - \frac{1}{t^{\mathrm{upp}}} \right) \sum_{k=0}^{K-2} \mathbb{E}[\mathbf{1}\{ t_{k+1} \neq t_{k} \}] \notag \\
        &\leq D^{2}_{X} \left( \frac{1}{t^{\mathrm{low}}} - \frac{1}{t^{\mathrm{upp}}} \right) \sum_{k=1}^{K-1} \mathbb{E}[\mathbf{1}\{ t_{k} \neq t_{k-1} \}] \leq D^{2}_{X} \left( \frac{1}{t^{\mathrm{low}}} - \frac{1}{t^{\mathrm{upp}}} \right) \mathbb{E}[T(K)] \notag \\
        &\leq D^{2}_{X} \left( \frac{1}{t^{\mathrm{low}}} - \frac{1}{t^{\mathrm{upp}}} \right) (J(K)+1) \left\lceil \left[ \log_{1/\beta}\left(\frac{t^{\mathrm{upp}}L}{\alpha}\right) \right]_{+} + 2 \right\rceil.
    \end{align}
    For (iii), by Lemma \ref{w3-bound-lemma}, we have that
    \begin{align}\label{SEG-monotone-proposition-eqn7}
        & (iii) \leq 3 t^{\mathrm{upp}} \sum_{k=0}^{K-1} \mathbb{E}[\| w^{k}_{3} \|^{2}] \notag \\
        &\leq ( 36(t^{\mathrm{upp}})^{3}L^{2} + 9t^{\mathrm{upp}} ) \sum_{k=0}^{K-1} \frac{\sigma^{2}}{N_{k}} + 36 t^{\mathrm{upp}} (t^{\mathrm{upp}}-t^{\mathrm{low}})^{2} D^{2}_{F} L^{2} \underbrace{\sum_{k=0}^{K-1} \mathbb{P} \{ t_{k} \neq t_{k-1} \}}_{\,=\, \mathbb{E}[T(K)]} \notag \\
        &\leq ( 36(t^{\mathrm{upp}})^{3}L^{2} + 9t^{\mathrm{upp}} ) \sum_{k=0}^{K-1} \frac{\sigma^{2}}{N_{k}} + 36 t^{\mathrm{upp}} (t^{\mathrm{upp}}-t^{\mathrm{low}})^{2} D^{2}_{F} L^{2} (J(K)+1) \left\lceil \left[ \log_{1/\beta}\left(\frac{t^{\mathrm{upp}}L}{\alpha}\right) \right]_{+} + 2 \right\rceil.
    \end{align}
    By substituting \eqref{wk_1-wk_2-bounds}, \eqref{SEG-monotone-proposition-eqn6}, and \eqref{SEG-monotone-proposition-eqn7} into \eqref{SEG-monotone-proposition-eqn5}, and using the fact that $t^{\mathrm{low}} \le t_k \le t^{\mathrm{upp}}$, we obtain
    \allowdisplaybreaks
    \begin{align*}
        \mathbb{E}[\theta(z^{R_{K}})] &\leq \frac{1}{2K} \left( \frac{1}{t^{\mathrm{low}}} D^{2}_{X} + D^{2}_{X} \right) + \frac{1}{2K} ( 3t^{\mathrm{upp}} + 1 ) \sum_{k=0}^{K-1} \frac{\sigma^{2}}{N_{k}} + \frac{1}{2K} ( 36(t^{\mathrm{upp}})^{3}L^{2} + 9t^{\mathrm{upp}} ) \sum_{k=0}^{K-1} \frac{\sigma^{2}}{N_{k}} \\
        &+ \frac{1}{2K} D^{2}_{X} \left( \frac{1}{t^{\mathrm{low}}} - \frac{1}{t^{\mathrm{upp}}} \right) (J(K)+1) \left\lceil \left[ \log_{1/\beta}\left(\frac{t^{\mathrm{upp}}L}{\alpha}\right) \right]_{+} + 2 \right\rceil \\
        &+ \frac{1}{2K} \cdot 36 t^{\mathrm{upp}} (t^{\mathrm{upp}}-t^{\mathrm{low}})^{2} D^{2}_{F} L^{2} (J(K)+1) \left\lceil \left[ \log_{1/\beta}\left(\frac{t^{\mathrm{upp}}L}{\alpha}\right) \right]_{+} + 2 \right\rceil,
    \end{align*}
    as desired.
\end{proof}

Similar to Theorem \ref{SEG-main-theorem}, by specifying the parameters $p_{j}$ and $N_{k}$, we may derive the following iteration and sample complexities.

\begin{theorem}\label{SEG-monotone-main-theorem}\em
    Consider the sequence generated by \textbf{SEG-LS}. Suppose Assumption $\mathrm{A}$ holds, except that the pseudomonotonicity assumption $\mathrm{(A1)}$ is replaced by monotonicity.
    For a given $K$, we set $p_{j} \triangleq \left\lceil c_{p}K \right\rceil$ and $N_{k} \triangleq \left\lceil c_{N}K \right\rceil$, where $c_{p}\in (0, 1]$ and $c_{N} > 0$. Then the iteration complexity $K_{\epsilon}$ and the sample complexity $S_{\epsilon} = \sum_{k=0}^{K_{\epsilon}-1} N_{k}$ that guarantee $\mathbb{E}[\theta(z^{R_{K}})]\leq \epsilon$ are $\mathcal{O}(\epsilon^{-1})$ and $\mathcal{O}(\epsilon^{-2})$, respectively.
\end{theorem}
\begin{proof}
    The proof is similar to that of Theorem \ref{SEG-main-theorem} and we omit it here.
\end{proof}

\begin{remark}\em
    We make several remarks here. First, as we mentioned before, it is an encouraging result that our iteration and sample complexities are improved from $\mathcal{O}(\epsilon^{-2})$ and $\mathcal{O}(\epsilon^{-4})$ to $\mathcal{O}(\epsilon^{-1})$ and $\mathcal{O}(\epsilon^{-2})$, respectively under monotonicity. Second, with the a prescribed constant stepsize $t^{\mathrm{upp}} = t^{\mathrm{low}} \equiv \mathcal{O}(1/L)$, we obtain the bound $\mathcal{O}(L D_{X}^2 /K)$. This matches the first term in the bound $\mathcal{O}(LD_X^2/K+\sigma D_X/\sqrt{K})$ established in Juditsky et al. \cite{juditsky-nemirovski-tauvel-2011} with a fixed batch size. The additional $\mathcal{O}(\sigma D_X/\sqrt{K})$ term accounts for stochastic noise under a fixed batch size, whereas we employ increasing batch sizes for variance reduction. Third, it is worth noting that the above result on monotone SVIs also applies to stochastic convex optimization with expectation-valued constraints \eqref{SOEC}, which is not addressed in \cite{wang-shanbhag-xie-2025}. $\hfill \Box$
\end{remark}

\section{B-SLAM under Strict Copositivity}\label{Sec-4}

In many cases, the pseudomonotonicity assumption $\mathrm{(A1)}$ does not always hold. In this section, we focus on examining the non-monotone setting where $F$ does not necessarily satisfy pseudomonotonicity. However, we do impose a requirement that $F$ is C$^1$ and  assume the Jacobian of $F$, defined as $\mathbf{J} F(x)$ for a given $x \, \in \, X$, always exists. In other words, the regularized gap function $\theta_{c}$ is C$^{1}$ by \eqref{regularized-gap-function-gradient}, where
\begin{equation}\label{stochastic-gap-function}
    \theta_{c}(x) = \mathbb{E}_{\bxi} \left[ \underbrace{\Tilde{F}(x, \bxi)^\top (x-y_{c}(x)) - \frac{c}{2}(x-y_{c}(x))^{\top}(x-y_{c}(x))}_{\triangleq \, \Tilde{\theta}_{c}(x, y_{c}(x), \bxi)} \right],
\end{equation}
where $y_c(x) \triangleq \displaystyle \argmin_{y\in X}~\mathbb{E}[\widehat{F}(x,y,\bxi)]$ and $\widehat{F}(x,y,\xi) \, \triangleq \, \Tilde{F}(x,\xi)^\top (y-x) + \frac{c}{2}(y-x)^\top (y-x)$. In the stochastic setting, the Jacobian at any  $x \, \in \, X$, given by $\mathbf{J}F(x)$ is expectation-valued and therefore difficult to estimate. Consequently, the gradient $\nabla \theta_c(x)$ in \eqref{regularized-gap-function-gradient} is not directly accessible. We contend with this unavailability by considering a framework reliant on biased values of $\theta_c$ via a  randomized smoothing framework
\begin{equation}\label{gap-function-RS}
    \theta^{\eta}_{c}(x) \triangleq \mathbb{E}_{\uu\in \mathbb{B}}\left[\, \theta_{c}(x+\eta \uu)\, \right] = \mathbb{E}_{\uu, \bxi}\left[\, \Tilde{\theta}_{c}(x+\eta \uu, y_{c}(x+\eta \uu), \bxi)\, \right],
\end{equation} 
where $\mathbb{B}$ denotes the unit ball. Computing a zeroth-order gradient estimator requires evaluating $\tilde{\theta}_c(x,y_c(x),\xi)$. However, this object cannot be computed in finite time since $y_c(x)$ necessitates evaluating the projection of an expectation-valued metric. Instead, we employ an approximation of $y_c(x)$, leading to the bias. Accordingly, we develop a biased stochastic linesearch framework for minimizing $\theta_{c}$ over $X$. The main idea is formalized next.

\vspace{4pt}

\begin{tcolorbox}
\textbf{Main idea.} Given the challenge in estimating the expectation-valued Jacobian matrix $\mathbf{J}F(x)$ as well as $\tilde{\theta}_c(x,y_c(x),\xi)$, we leverage the biased \emph{zeroth-order} gradient descent and a \emph{biased} stochastic Armijo linesearch on the smoothed approximation $\theta^{\eta_k}_{c}$ with a diminishing sequence of smoothing parameters $\{\eta_k\}_{k=0}^{K-1}$ upper bounded by $\bar{\eta} > 0$, a diminishing error sequence $\{\epsilon_k\}_{k=0}^{K-1}$ and an increasing mini-batch size sequence $\{N_{k}\}^{K-1}_{k = 0}$.
\end{tcolorbox}

\vspace{4pt}

Our biased SLAM (\textbf{B-SLAM}) is described in Algorithm \ref{SLAM}. We impose the following standard assumption throughout this section.

\vspace{4pt}

\begin{custombox}
    \textbf{Assumption $\mathrm{B}$.} \\
    (B1) Each batch $\{(u^{i,k,r},v^{i,k,r},\xi^{i,k,r})\}_{i=1}^{N_k}$ is i.i.d. and is independent of the history $\mathcal{H}^{r}_{k}$. \\
    (B2) For any $\xi \in \Xi$, $\Tilde{\theta}_{c}(\bullet, y_{c}(\bullet), \xi)$ is $L_{\nabla \theta}$-smooth and $L_{\theta}$-Lipschitz. \\
    (B3) For any $x\in X$ and $\xi \in \Xi$, $\Tilde{\theta}_{c}(x, \bullet, \xi)$ is $L^{y}_{\theta}$-Lipschitz. \\
    (B4) The set $X$ is compact with diameter $D_{X} \triangleq \sup_{x, y\in X} \| x - y \| < \infty$. \\
    (B5) For any $x \in X$, there exists $\sigma_{F} > 0$ such that the following hold almost surely: (i) $\mathbb{E} [\, \Tilde{F}(x, \bxi) \:\vert\: x\, ] \, = \, F(x)$; and (ii) $\mathbb{E} [\, \| \, \Tilde{F}(x, \bxi) - F(x) \, \|^{2} \:\vert\: x\, ] \, \leq \, \sigma^{2}_{F}$. \\
    (B6) Consider Algorithm \ref{SA-lower}. For any $k, t \geq 0$, any $r \leq \tau_{k}$, we have that (i) $\mathbb{E}[\nabla_{y}\widehat{F}(\hat{x}^{k}, y^{t}, \xi^{r, t}) \mid \hat{x}^{k}, y^{t}] \, = \, \nabla_{y}\mathbb{E}[\widehat{F}(\hat{x}^{k}, y^{t}, \bxi) \mid \hat{x}^{k}, y^{t}]$; (ii) $\mathbb{E}[ \| \nabla_{y}\widehat{F}(\hat{x}^{k}, y^{t}, \xi^{r, t}) - \nabla_{y}\mathbb{E}[\widehat{F}(\hat{x}^{k}, y^{t}, \bxi)] \|^{2} \mid \hat{x}^{k}, y^{t}] \, \leq \, v^{2}_{F}$ holds a.s. for some $v_{F} > 0$; and (iii) the bound $\| \nabla_{y}\mathbb{E}[\widehat{F}(x, y, \bxi)] \| \, \leq \, c_{F}$ holds for some $c_{F} > 0$ and any $x, y\in X$.
\end{custombox}

\vspace{4pt}

The Lipschitz continuity assumptions $\mathrm{(B2)}$ and $\mathrm{(B3)}$ can be guaranteed under suitable requirements on $\tilde{F}(\bullet,\xi)$ and the set $X$. We formalize these relations in Lemma \ref{VI-to-Lipschitz}, provided next.

\begin{lemma}[\mbox{\cite[Lemma 5.7]{xiao-shanbhag-2025-GR}}]\label{VI-to-Lipschitz}\em
    Suppose $X$ is compact with diameter $D_{X} > 0$ and $\mathbf{J} F(x)$ exists for all $x\in X$. Assume that there exists some $L_{F} > 0$, $M_{F} > 0$, $M_{J} > 0$, and $L_{J} > 0$ such that for any $x, y \in X$ and any realization $\xi$, we have that $\| \Tilde{F}(x, \xi) - \Tilde{F}(y, \xi) \| \leq L_{F} \| x - y \|$, $\| \Tilde{F}(x, \xi) \| \leq M_{F}$, $\| \mathbf{J} \Tilde{F}(x, \xi) \| \leq M_{J}$, and $\| \mathbf{J} \Tilde{F}(x, \xi) - \mathbf{J} \Tilde{F}(y, \xi) \| \leq L_{J} \| x - y \|$. Then the following hold: (i) $y_{c}(\bullet)$ is $(1 + \tfrac{L_{F}}{c})$-Lipschitz; (ii) Each $\Tilde{\theta}_{c}(\bullet, y(\bullet), \xi)$ is $L_{\theta}$-Lipschitz, where $L_{\theta} \triangleq 2D_{X}(L_{F}+c) + M_{F}(2 + \tfrac{L_{F}}{c})$; (iii) Each $\Tilde{\theta}_{c}(\bullet, y(\bullet), \xi)$ is $L_{\nabla \theta}$-smooth, where $L_{\nabla \theta} \triangleq 2(L_{F}+M_{J}+c) + D_{X}L_{J} + \tfrac{M_{J}L_{F}}{c}$; and (iv) Each $\Tilde{\theta}_{c}(x, \bullet, \xi)$ is $L^{y}_{\theta}$-Lipschitz, where $L^{y}_{\theta} \triangleq M_{F} + c D_{X}$. $\hfill \Box$
\end{lemma}

\subsection{A stochastic linesearch framework}

We begin by defining unbiased sampled objective $\bar{\theta}^{\eta_{k}, r}_{c}(\bullet)$, unbiased sampled zeroth-order gradient estimator $\bar{g}^{\eta_{k}, r}_{c}(\bullet)$, the (biased) sampled objective $\bar{\theta}^{\eta_{k}, r}_{c, \epsilon_{k}}(\bullet)$, and the (biased) sampled zeroth-order gradient estimator $\bar{g}^{\eta_{k}, r}_{c, \epsilon_{k}}(\bullet)$ as
\allowdisplaybreaks
\begin{align}
    \bar{\theta}^{\eta_{k}, r}_{c}(\bullet) &\triangleq \frac{1}{N_{k}} \sum_{i=1}^{N_{k}} \Tilde{\theta}_{c}(\bullet+\eta_{k} u^{i,k,r}, y_{c}(\bullet+\eta_{k} u^{i,k,r}), \xi^{i,k,r}), \label{theta_bar} \\
    \bar{g}^{\eta_{k}, r}_{c}(\bullet) &\triangleq \frac{n}{2\eta_{k} N_{k}} \sum_{i=1}^{N_{k}}( \Tilde{\theta}_{c}(\bullet + v^{i,k,r}, y_{c}(\bullet+v^{i,k,r}), \xi^{i,k,r}) - \Tilde{\theta}_{c}(\bullet-v^{i,k,r}, y_{c}(\bullet-v^{i,k,r}), \xi^{i,k,r}) ) \frac{v^{i,k,r}}{\|v^{i,k,r}\|}, \label{g_bar} \\
    \bar{\theta}^{\eta_{k}, r}_{c, \epsilon_{k}}(\bullet) &\triangleq \frac{1}{N_{k}} \sum_{i=1}^{N_{k}} \Tilde{\theta}_{c}(\bullet+\eta_{k} u^{i,k,r}, y_{c}^{\epsilon_{k}}(\bullet+\eta_{k} u^{i,k,r}), \xi^{i,k,r}), \mbox{ and } \label{theta_bar_eps} \\
    \bar{g}^{\eta_{k}, r}_{c, \epsilon_{k}}(\bullet) &\triangleq \frac{n}{2\eta_{k} N_{k}} \sum_{i=1}^{N_{k}}( \Tilde{\theta}_{c}(\bullet+v^{i,k,r}, y_{c}^{\epsilon_{k}}(\bullet+v^{i,k,r}), \xi^{i,k,r}) - \Tilde{\theta}_{c}(\bullet-v^{i,k,r}, y_{c}^{\epsilon_{k}}(\bullet-v^{i,k,r}), \xi^{i,k,r}) ) \frac{v^{i,k,r}}{\|v^{i,k,r}\|}, \label{g_bar_eps}
\end{align}
respectively. Here the batch index $r$ specifies the batch $\{u^{i,k,r},v^{i,k,r},\xi^{i,k,r}\}_{i=1}^{N_k}$. For any $k$, the batches $\{u^{\bullet,k,r},v^{\bullet,k,r},\xi^{\bullet,k,r}\}$ are iid batches for any positive integer $r$. Due to the unavailability of exact lower-level solutions, we propose a stochastic \emph{biased} Armijo linesearch by employing $\bar{\theta}^{\eta_{k}, r}_{c, \epsilon_{k}}(\bullet)$ and $\bar{g}^{\eta_{k}, r}_{c, \epsilon_{k}}(\bullet)$ defined above. During the implementation of our \textbf{B-SLAM}, we use a stochastic approximation (\textbf{SA}) scheme (see Algorithm \ref{SA-lower}) to compute inexact lower-level solutions $y_{c}^{\epsilon_k}(\widehat{x}^{k})$ that approximate $y_c(\widehat{x}^{k})$, where $\widehat{x}^{k}$ denotes the upper-level parameter at the $k$-th outer iteration. Note that given iteration $k$, different batch $r$ uses the same lower-level SA iteration times $T_{k}$ but different realizations. The convergence of the lower-level \textbf{SA} scheme is captured in Lemma \ref{lower-level-error}.

\begin{remark}\em
    We slightly abuse the notation here. The approximate solution $y_c(\widehat{x}^{k})$ satisfying \eqref{SA-inexactness} is not unique and is obtained by implementing lower-level \textbf{SA} scheme (Algorithm \ref{SA-lower}). $\hfill \Box$
\end{remark}

\vspace{4pt}

\begin{custombox}
\noindent\textbf{Biased stochastic Armijo linesearch.} Specifically at iteration $k$ and when batch $r$ is employed, we reduce the stepsize $t_{k}$ at a geometric rate $\beta>0$ from some initial value $t^{\mathrm{upp}}>0$, i.e., $t_k\gets\max\{\beta t_k,t^{\mathrm{low}}\}$, until the following Armijo-type condition is satisfied
\begin{equation}\label{sto-biased-Armijo}
\begin{aligned}
    \bar{\theta}^{\eta_k,r}_{c,\epsilon_k}(\Pi_X[x^k-t_k\bar{g}^{\eta_k,r}_{c,\epsilon_k}(x^k)]) - \bar{\theta}^{\eta_k,r}_{c,\epsilon_k}(x^k) \leq -\frac{\alpha}{t_k} \|x^k-\Pi_X[x^k-t_k\bar{g}^{\eta_k,r}_{c,\epsilon_k}(x^k)]\|^2 + e^k,
\end{aligned}
\end{equation}
where $e^k$ is prescribed in \eqref{SLAM-noise} and $t^{\mathrm{low}}$ is defined by
\begin{equation}\label{SLAM-floor}
    t^{\mathrm{low}} \triangleq\min\left\{t^{\mathrm{upp}},\frac{2(1-\alpha)\beta}{L_{\nabla\theta}}\right\}>0
\end{equation}
for given $\alpha, \beta \in (0,1)$. If $t_{k} = t^{\mathrm{low}}$ fails, we discard this batch and repeat the procedure with a new batch from the same $t_k^{\mathrm{init}} > 0$.
\end{custombox}

\vspace{4pt}

\begin{remark}\em
    We provide some comments on the modifications made relative to the unbiased case to address the effects of bias. First, unlike the periodically restarted linesearch in Section \ref{Sec-3}, the above linesearch requires knowledge of $L_{\nabla\theta}$ to set the lower bound $t^{\mathrm{low}}$. Second, we allow for generating multiple independent batches at a given epoch $k$, each of which is indexed by $r$,  and will proceed to demonstrate that the above linesearch terminates in finite time with probability one. It bears reminding that the total number of batches $\tau_{k}$ at the $k$-th iteration is a random variable. Third, the bias $e^{k}$ defined in \eqref{SLAM-noise} diminishes to zero as $k$ increases, providing the possibility of approximate descent of the function value from \eqref{sto-biased-Armijo}. $\hfill \Box$
\end{remark}

\begin{algorithm}[htb]
\caption{~\textbf{B-SLAM}}\label{SLAM}
\begin{algorithmic}[1]
\State \textbf{Input}: $x^{0}\in X$, $\alpha \in (0,1)$, $\beta \in (0,1)$, $\{p_{j}\}$ where $p_{j}\in \mathbb{N}$, $t^{\mathrm{upp}} > 0$, $L_{\nabla \theta} > 0$, $\{ \eta_{k}, \epsilon_{k}, N_{k} \}_{k\geq 0}$.
\State Set $t^{\mathrm{low}}$ by \eqref{SLAM-floor} and $j\leftarrow -1$;
\For{$k = 0, 1, \cdots, K-1$} 
    \State \textbf{if} \: $k = 0$ or $d = p_{j}$ \: \textbf{then}
    \State \quad\quad Set $t^{\mathrm{init}}_{k} \leftarrow t^{\mathrm{upp}}$;\; Set $d \leftarrow 1$, $j \leftarrow j+1$;
    \State \textbf{else}
    \State \quad\quad Set $t^{\mathrm{init}}_{k} \leftarrow t_{k-1}$;\; Set $d \leftarrow d+1$;
    \State \textbf{end if}
    \State Set $r \leftarrow 0$;
    \State \textbf{repeat} \hfill $\triangleright$ \emph{a multiple-batch setting}
    \State \quad\quad Set $r \leftarrow r+1$; generate a batch $\{ u^{i,k,r}, v^{i,k,r}, \xi^{i,k,r} \}_{i=1}^{N_k}$;
    \State \quad\quad Set $t_{k} \leftarrow t^{\mathrm{init}}_{k}$;
    \State \quad\quad \textbf{loop}
    \State \quad\quad \quad\quad \textbf{if} \: \eqref{sto-biased-Armijo} holds for this batch \: \textbf{then} \hfill $\triangleright$ \emph{compute inexact estimators via} \textbf{SA}
    \State \quad\quad \quad\quad \quad\quad \textbf{break};
    \State \quad\quad \quad\quad \textbf{else if} \: $t_{k} = t^{\mathrm{low}}$ \: \textbf{then}
    \State \quad\quad \quad\quad \quad\quad \textbf{break}; \hfill $\triangleright$ \emph{discard this batch and retry with a new batch}
    \State \quad\quad \quad\quad \textbf{else}
    \State \quad\quad \quad\quad \quad\quad Update $t_{k} \leftarrow \max\{ \beta t_{k}, t^{\mathrm{low}} \}$;
    \State \quad\quad \quad\quad \textbf{end if}
    \State \quad\quad \textbf{end loop}
    \State \textbf{until} \: \eqref{sto-biased-Armijo} is satisfied
    \State Set $\tau_{k} \leftarrow r$; \hfill $\triangleright$ $\tau_{k}$: \emph{total number of batches used at $k$-th iteration}
    \State Update $x^{k+1} \gets \Pi_{X}[x^{k}-t_{k}\bar{g}^{\eta_{k}, \tau_{k}}_{c, \epsilon_{k}}(x^{k})]$;
\EndFor
\State \textbf{Output}: $x^{R_{K}}$ where $R_{K}$ is a random variable with $\mathbb{P}[R_{K} = j] = 1/K$ for $j\in \{ 0, 1, \cdots, K-1 \}$.
\end{algorithmic}
\end{algorithm}

\begin{algorithm}[t]
\caption{~\textbf{SA} scheme (at $k$-th iteration of \textbf{B-SLAM} when batch $r$ is employed)}\label{SA-lower}
\begin{algorithmic}[1]
\State \textbf{Input}: $y^{0}$, $\widehat{x}^{k}\in X_{\bar{\eta}}$, iteration times $T_{k}$, and stepsize $\alpha_{t} = \frac{\alpha_{0}}{t+\Gamma}$ where $\alpha_{0} > \frac{1}{2c}$ and $\Gamma>0$.
\For{$t = 0, 1, \cdots, T_{k}-1$}  
    \State Generate an unbiased gradient realization $\nabla_{y}\widehat{F}(\widehat{x}^{k}, y^{t}, \xi^{r, t})$ of $\mathbb{E}[\widehat{F}(\widehat{x}^{k}, y, \bxi)]$ at $y = y^{t}$;
    \State Update $y^{t+1} \leftarrow \Pi_{X}[y^{t}-\alpha_{t} \nabla_{y}\widehat{F}(\hat{x}^{k}, y^{t}, \xi^{r, t})]$;
\EndFor
\State \textbf{Output}: $y^{T_{k}}$ as final estimate and let $y_{c}^{\epsilon_{k}}(\widehat{x}^{k}) \triangleq y^{T_{k}}$.
\end{algorithmic}
\end{algorithm}

\begin{lemma}\label{lower-level-error}\em
Consider Algorithm \ref{SA-lower} for solving the strongly convex stochastic optimization problem \eqref{yc-x-2}. Suppose assumption $\mathrm{(B6)}$ holds. Given $k > 0$ and $\hat{x}^{k}\in X$, let $y_{c}(\hat{x}^{k})$ be the unique exact solution and $y^{\epsilon_{k}}_{c}(\hat{x}^{k})$ be the output of Algorithm \ref{SA-lower} after $T_{k}$ iterations. Then for any $k \geq 0$, we have that
\begin{equation}\label{SA-inexactness}
    \mathbb{E}[\|y^{\epsilon_{k}}_{c}(\hat{x}^{k})-y_{c}(\hat{x}^{k})\|^{2} \mid \hat{x}^{k}]\leq \epsilon_{k} \triangleq \frac{c_{\epsilon}}{T_{k}+\Gamma},
\end{equation}
where $\Gamma > 0$ and $c_{\epsilon} \triangleq \max \left\{\tfrac{(c^{2}_{F}+v^{2}_{F})\alpha^{2}_{0}}{2c\alpha_{0}-1}, \: \Gamma\sup_{y\in X}\|y_{0}-y\|^{2} \right\} > 0$.
\end{lemma}
\begin{proof}
    The proof is similar to that of \cite[Theorem 2-(a)]{cui-shanbhag-yousefian-2023}, except that the cited result solves a strongly monotone stochastic VI, whereas here we solve the strongly convex stochastic optimization problem \eqref{yc-x-2}. We omit the proof here.
\end{proof}

\begin{remark}\em
    The unbiasedness and second moment conditions in assumption $\mathrm{(B6)}$ can be guaranteed under suitable requirements on $\mathrm{SVI}\:(X, F)$. Under the unbiasedness assumption $\mathbb{E}[\tilde{F}(\hat{x}, \xi) \mid \hat{x}] = F(\hat{x})$, it can be observed that $\nabla_{y}\widehat{F}(\hat{x}, \bullet\, \xi) \triangleq \tilde{F}(\hat{x}, \xi) + c( \bullet - \hat{x} )$ is an unbiased estimator. Also we have $\mathbb{E}[ \| \nabla_{y}\widehat{F}(\hat{x}, y, \bxi) - \nabla_{y}\mathbb{E}[\widehat{F}(\hat{x}, y, \bxi)] \|^{2} \mid \hat{x}, y] \, = \, \mathbb{E} [\, \| \, \Tilde{F}(x, \bxi) - F(x) \, \|^{2} \:\vert\: x\, ] \, \leq \, \sigma^{2}_{F}$ where $\sigma_{F}$ is defined in assumption $\mathrm{(B5)}$. Therefore, we may set $v_{F} \triangleq \sigma_{F}$. $\hfill \Box$
\end{remark}

\subsection{Supporting lemmas}

In this subsection, we provide some supporting lemmas that are helpful in our subsequent analysis. To proceed, we define the history $\mathcal{F}_k$ as
\begin{equation*}
    \mathcal{F}_k \triangleq \{ x^{0} \} \, \cup \, \left( \cup_{t=0}^{k-1} \left\{ \cup_{r=1}^{\tau_{t}} \left\{ \cup_{i=1}^{N_{t}} \left\{ u^{i,t,r}, v^{i,t,r}, \xi^{i,t,r} \right\} \right\} \right\} \right), \quad \mathcal{F}_{0} \triangleq \{ x^{0} \}
\end{equation*}
and define the history $\mathcal{H}^{r}_{k}$ $(1\leq r \leq \tau_{k})$ as
\begin{equation*}
    \mathcal{H}^{r}_{k} \triangleq \mathcal{F}_k \cup \left( \cup_{j=1}^{r-1} \left\{ \cup_{i=1}^{N_{k}} \left\{ u^{i,k,j}, v^{i,k,j}, \xi^{i,k,j} \right\} \right\} \right), \quad \mathcal{H}^{1}_{k} \triangleq \mathcal{F}_k.
\end{equation*}
It may be noted that $\mathcal{H}^r_k$ denotes the history up to and including the batches generated at index $r-1$. Then we may see that $x^k$ and $t_k^{\mathrm{init}}$ are $\mathcal{F}_k$-measurable, and the event $\{\tau_k\geq r\}\in\mathcal{H}^{r}_{k}$. Lemmas \ref{slam-suport-lemma-1}--\ref{slam-suport-lemma-3} concern properties of batch index $r$ before the final acceptance of the approximate descent requirement.

\begin{lemma}\label{slam-suport-lemma-1}\em
Suppose Assumption $\mathrm{B}$ holds. For any $k\geq 0$, $r\leq \tau_k$, and any $\mathcal{H}^{r}_{k}$-measurable $x\in X$, we have that
\begin{equation*}
    \mathbb{E} \left[\left\| \nabla\Tilde{\theta}_c(x+\eta_k u^{i,k,r},y_c(x+\eta_k u^{i,k,r}),\xi^{i,k,r}) - \nabla\theta^{\eta_k}_{c}(x)\right\|^2\,\middle|\,\mathcal{H}^{r}_{k}\right] \leq L_{\theta}^2,
\end{equation*}
where $L_{\theta}$ is given in assumption $\mathrm{(B2)}$.
\end{lemma}
\begin{proof}
     Denote $\zeta^{i,k,r} \triangleq \nabla\Tilde{\theta}_c(x+\eta_k u^{i,k,r},y_c(x+\eta_k u^{i,k,r}),\xi^{i,k,r})$. Then we can see that
    \begin{equation*}
        \mathbb{E}\left[ \zeta^{i,k,r} \,\middle\vert\, \mathcal{H}^{r}_{k} \right] = \nabla \theta^{\eta_{k}}_{c}(x).
    \end{equation*}
    Since the Lipschitz constant of a smooth function provides a bound on its gradient norm, we have
    \begin{equation*}
        \mathbb{E} \left[ \left\| \zeta^{i,k,r} - \mathbb{E}\left[ \zeta^{i,k,r} \,\middle\vert\, \mathcal{H}^{r}_{k} \right] \right\|^2 \,\middle\vert\, \mathcal{H}^{r}_{k} \right] = \mathbb{E}\left[ \left\| \zeta^{i,k,r}\right\|^2 \,\middle\vert\, \mathcal{H}^{r}_{k} \right] - \left\|\mathbb{E} \left[ \zeta^{i,k,r} \,\middle\vert\, \mathcal{H}^{r}_{k} \right]\right\|^2 \leq L_{\theta}^2
    \end{equation*}
    where $L_{\theta}$ is given in assumption $\mathrm{(B2)}$.
\end{proof}

The following lemma is a direct consequence of Lemma \ref{slam-suport-lemma-1}.

\begin{lemma}\label{slam-suport-lemma-2}\em
    Suppose Assumption $\mathrm{B}$ holds. For any $k\geq 0$, $r\leq \tau_k$, and any $\mathcal{H}^{r}_{k}$-measurable $x\in X$, we have that
    \begin{equation*}
         \mathbb{E} \left[ \nabla\bar{\theta}^{\eta_k,r}_{c}(x) \,\middle\vert\, \mathcal{H}^{r}_{k} \right] = \nabla\theta^{\eta_k}_{c}(x), \quad \mathbb{E} \left[ \left\| \nabla \bar{\theta}^{\eta_k,r}_{c}(x) - \nabla\theta^{\eta_k}_{c}(x) \right\|^2 \,\middle\vert\, \mathcal{H}^{r}_{k} \right] \leq\frac{L_{\theta}^2}{N_k},
    \end{equation*}
    where $L_{\theta}$ is given in assumption $\mathrm{(B2)}$.
\end{lemma}
\begin{proof}
    The first unbiasedness equality follows from the i.i.d. assumption $\mathrm{(B1)}$ and the definition of randomized smoothing. The second moment bound follows immediately from the unbiasedness equality, the definition of $\bar{\theta}^{\eta_k,r}_{c}(\bullet)$ in \eqref{theta_bar}, and Lemma \ref{slam-suport-lemma-1}.
\end{proof}

Two increments that will be frequently used in the subsequent analysis are defined as
\begin{equation}\label{delta-Delta}
    \delta^{k,r}\triangleq \nabla\bar{\theta}^{\eta_k,r}_{c}(x^k)-\bar{g}^{\eta_k,r}_{c,\epsilon_k}(x^k) \mbox{ and } \Delta^{k,r} \triangleq  \nabla\theta^{\eta_k}_{c}(x^k)-\bar{g}^{\eta_k,r}_{c,\epsilon_k}(x^k),
\end{equation}
respectively.  The following lemma establishes bound on the first and second conditional moments of these two increments and plays an important role in our subsequent analysis.

\begin{lemma}[Second moments of two increments]\label{slam-suport-lemma-3}\em
    Suppose Assumption $\mathrm{B}$ holds. For any $k\geq 0$ and $r\leq \tau_k$, let $\delta^{k,r}$ and $\Delta^{k,r}$ be defined as in \eqref{delta-Delta}. Then we have
    \begin{align}
        \mathbb{E} \left[ \| \delta^{k,r} \|^{2} \,\middle\vert\, \mathcal{H}^{r}_{k} \right] & \leq d_{k} \triangleq \frac{2( L_\theta^2 + 16\sqrt{2\pi}L_\theta^2n )}{N_k}+\frac{2n^2}{\eta_k^2}(L^y_\theta)^2\epsilon_k, \label{dk-def} \\
        \mathbb{E} \left[ \| \Delta^{k,r} \|^2 \,\middle\vert\, \mathcal{H}^{r}_{k} \right] &\leq v_k \triangleq \frac{32\sqrt{2\pi}L_\theta^2n}{N_k} +\frac{2n^2(L^y_\theta)^2\epsilon_k}{\eta_k^2} \label{vk-def}
    \end{align}
    where $L_{\theta}$ and $L^{y}_{\theta}$ are given in assumptions $\mathrm{(B2)}$ and $\mathrm{(B3)}$, respectively, and $\epsilon_k$ is given in \eqref{SA-inexactness}.
\end{lemma}
\begin{proof}
   By the definitions of $\bar{\theta}^{\eta_k,r}_{c}$ in \eqref{theta_bar} and $\bar{g}^{\eta_k,r}_{c}$ in \eqref{g_bar}, we may see that
   \begin{equation}\label{bar-g-gradient-bar-theta}
       \mathbb{E}_{\vv} \left[ \bar{g}^{\eta_{k}, r}_{c}(x^{k}) \,\middle\vert\, \mathcal{H}^{r}_{k}, \{ \xi^{i,k,r} \}_{i=1}^{N_{k}} \right] = \mathbb{E}_{\uu} \left[ \nabla \bar{\theta}^{\eta_{k}, r}_{c}(x^{k}) \,\middle\vert\, \mathcal{H}^{r}_{k}, \{ \xi^{i,k,r} \}_{i=1}^{N_{k}} \right].
   \end{equation}
   Therefore, by Lemma \ref{slam-suport-lemma-1} and Lemma \ref{RS-property-grad}-(vi), we may deduce that
   \allowdisplaybreaks
    \begin{align}\label{slam-suport-lemma-3-eqn1}
        & \mathbb{E} \left[ \left\| \nabla \bar{\theta}^{\eta_{k}, r}_{c} (x^{k}) - \bar{g}^{\eta_{k}, r}_{c}(x^{k}) \right\|^{2} \,\middle\vert\, \mathcal{H}^{r}_{k} \right] \notag \\
        &= \mathbb{E}_{\bxi} \left[ \mathbb{E}_{\uu, \vv} \left[ \left\| \nabla \bar{\theta}^{\eta_{k}, r}_{c} (x^{k}) - \bar{g}^{\eta_{k}, r}_{c}(x^{k}) \right\|^{2} \,\middle\vert\, \mathcal{H}^{r}_{k}, \{ \xi^{i,k,r} \}_{i=1}^{N_{k}} \right] \,\middle\vert\, \mathcal{H}^{r}_{k} \right] \notag \\
        &= \mathbb{E}_{\bxi} \left[ \mathbb{E}_{\uu} \left[ \left\| \nabla \bar{\theta}^{\eta_{k}, r}_{c} (x^{k}) \right\|^{2} \,\middle\vert\, \mathcal{H}^{r}_{k}, \{ \xi^{i,k,r} \}_{i=1}^{N_{k}} \right] - 2\mathbb{E}_{\uu, \vv} \left[ \nabla \bar{\theta}^{\eta_{k}, r}_{c} (x^{k})^{\top} \bar{g}^{\eta_{k}, r}_{c}(x^{k}) \,\middle\vert\, \mathcal{H}^{r}_{k}, \{ \xi^{i,k,r} \}_{i=1}^{N_{k}} \right] \,\middle\vert\, \mathcal{H}^{r}_{k} \right] \notag \\
        &+ \mathbb{E}_{\bxi} \left[ \mathbb{E}_{\vv} \left[ \left\| \bar{g}^{\eta_{k}, r}_{c}(x^{k}) \right\|^{2} \,\middle\vert\, \mathcal{H}^{r}_{k}, \{ \xi^{i,k,r} \}_{i=1}^{N_{k}} \right] \,\middle\vert\, \mathcal{H}^{r}_{k} \right] \notag \\
        &\overset{\eqref{bar-g-gradient-bar-theta}}{=} \mathbb{E}_{\bxi} \left[ \mathbb{E}_{\uu} \left[ \left\| \nabla \bar{\theta}^{\eta_{k}, r}_{c} (x^{k}) \right\|^{2} \,\middle\vert\, \mathcal{H}^{r}_{k}, \{ \xi^{i,k,r} \}_{i=1}^{N_{k}} \right] - 2 \left\| \mathbb{E}_{\uu} \left[ \nabla \bar{\theta}^{\eta_{k}, r}_{c} (x^{k}) \,\middle\vert\, \mathcal{H}^{r}_{k}, \{ \xi^{i,k,r} \}_{i=1}^{N_{k}} \right] \right\|^{2} \,\middle\vert\, \mathcal{H}^{r}_{k} \right] \notag \\
        &+ \mathbb{E}_{\bxi} \left[ \mathbb{E}_{\vv} \left[ \left\| \bar{g}^{\eta_{k}, r}_{c}(x^{k}) \right\|^{2} \,\middle\vert\, \mathcal{F}_{k}, \{ \xi^{i,k} \}_{i=1}^{N_{k}} \right] \,\middle\vert\, \mathcal{F}_{k} \right] \leq \frac{L_{\theta}^{2} + 16\sqrt{2\pi}L^{2}_{\theta}n}{N_{k}},
    \end{align}
    where the penultimate inequality is from the law of total conditional variance, Lemma \ref{slam-suport-lemma-2}, and Lemma \ref{RS-property-grad}-(vi). Therefore, we have
    \begin{align*}
        \mathbb{E} \left[ \|\delta^{k,r}\|^2 \,\middle\vert\, \mathcal{H}^{r}_{k} \right] &\leq 2\mathbb{E}\left[ \left\| \nabla\bar{\theta}_c^{\eta_k,r}(x^k) - \bar{g}_c^{\eta_k,r}(x^k) \right\|^2 \,\middle\vert\, \mathcal{H}^{r}_{k} \right] + 2\mathbb{E} \left[ \left\| \bar{g}_c^{\eta_k,r}(x^k) -\bar{g}_{c,\epsilon_k}^{\eta_k,r}(x^k) \right\|^2 \,\middle\vert\, \mathcal{H}^{r}_{k} \right] \\
        &\leq \frac{2 ( L_{\theta}^2 + 16\sqrt{2\pi} L_\theta^2 n )} {N_k} + \frac{2 n^2 (L_\theta^y)^2}{\eta_k^2} \epsilon_k,
    \end{align*}
    where the last inequality is from \eqref{slam-suport-lemma-3-eqn1} and Lemma \ref{lower-level-error}. Similarly, by Lemma \ref{RS-property-grad}-(vi) and Lemma \ref{lower-level-error} again, we may establish
    \begin{align*}
        \mathbb{E}\left[ \|\Delta^{k,r}\|^2 \,\middle\vert\, \mathcal{H}^{r}_{k} \right] & \leq 2\mathbb{E} \left[ \left\| \nabla\theta_c^{\eta_k}(x^k) -\bar{g}_c^{\eta_k,r}(x^k) \right\|^2 \,\middle\vert\, \mathcal{H}^{r}_{k} \right] + 2\mathbb{E}\left[ \left\| \bar{g}_c^{\eta_k,r}(x^k) - \bar{g}_{c,\epsilon_k}^{\eta_k,r}(x^k) \right\|^2 \,\middle\vert\, \mathcal{H}^{r}_{k} \right]\\
        &\leq \frac{32\sqrt{2\pi}L_\theta^2n}{N_k} +\frac{2n^2(L_\theta^y)^2\epsilon_k}{\eta_k^2}.
    \end{align*}
    Then we complete the proof.
\end{proof}

Based on Lemma \ref{slam-suport-lemma-3}, we now specify the bias $e^{k}$ in the backtracking procedure \eqref{sto-biased-Armijo}. By the definition of $t^{\mathrm{low}}$ in \eqref{SLAM-floor}, we know that
\begin{equation}\label{c0-def}
    c_0\triangleq \frac{1-\alpha}{t^{\mathrm{low}}}-\frac{L_{\nabla \theta}}{2} \geq \frac{(1-\beta)L_{\nabla \theta}}{2\beta}>0.
\end{equation}
Based on Lemma \ref{slam-suport-lemma-3}, we define $B_{k}$ as
\begin{equation}\label{Bk-def}
    B_k \triangleq 2L^y_\theta\sqrt{\epsilon_k}+\frac{d_k}{4c_0},
\end{equation}
where $d_{k}$ is defined in \eqref{dk-def} and prescribe the linesearch tolerance as
\begin{equation}\label{SLAM-noise}
    e^k\triangleq 2B_k.
\end{equation}
By the above definition, we know that $e^{k} \to 0$ as $k$ goes to infinity. In our next result, we prove that linesearch scheme terminates with probability one. In addition, we derive a bound on the expected number of batches at any epoch $k$.

\begin{lemma}[Stepsize bounds and finite termination]\label{SLAM-stepsize-lemma}\em
    Suppose Assumption $\mathrm{B}$ holds. Consider the stepsizes $\{ t_{k} \}_{k\geq 0}$ and batches $\{ \tau_{k} \}_{k\geq 0}$ generated by \textbf{B-SLAM}. Then every iteration terminates almost surely, and
    \begin{equation*}
        \mathbb{P} \left[ \, t^{\mathrm{low}}\leq t_k\leq t^{\mathrm{upp}}, \;\; \forall k = 0, \ldots, K-1 \, \right] = 1.
    \end{equation*}
    Moreover, we have
    \begin{equation}\label{tauk-P-E}
        \mathbb{P}\left[\,\tau_k>r\mid\mathcal{F}_k\, \right]\leq2^{-r} \;\; (r=0,1,\ldots) \quad \textrm{and} \quad \mathbb{E}[\tau_k\mid\mathcal{F}_k] = \sum_{r=0}^{\infty} \mathbb{P}\left[\, \tau_k>r\mid\mathcal{F}_k\, \right] \leq 2.
    \end{equation}
\end{lemma}
\begin{proof}
    For convenience, we denote $x^{k+1}_{\mathrm{low}} \triangleq \Pi_X[x^k-t^{\mathrm{low}} \bar{g}^{\eta_k,r}_{c,\epsilon_k}(x^k)]$. We also define $e_1^{k,r}$, $e_2^{k,r}$, and $R^{r}_{k}$ as 
    \begin{equation}\label{e1-e2-Rrk}
        e_1^{k,r}\triangleq \bar{\theta}^{\eta_k,r}_{c}(x^{k+1}_{\mathrm{low}})-\bar{\theta}^{\eta_k,r}_{c,\epsilon_k}(x^{k+1}_{\mathrm{low}}), \;\; e_2^{k,r}\triangleq \bar{\theta}^{\eta_k,r}_{c}(x^k)-\bar{\theta}^{\eta_k,r}_{c,\epsilon_k}(x^k), \;\; \textrm{and} \;\; R^{r}_{k} \triangleq |e_1^{k,r}|+|e_2^{k,r}| + \frac{\| \delta^{k,r} \|^2}{4 c_0},
    \end{equation}
    respectively, where $\delta^{k,r}$ is defined in \eqref{delta-Delta}. By the projection property, we have that
    \begin{align}\label{g-zo-inner-product}
        \left( x^{k}-t^{\mathrm{low}} \bar{g}^{\eta_{k}}_{c, \epsilon_{k}}(x^{k}) - x^{k+1}_{\mathrm{low}} \right)^{\top} \left( x^{k} - x^{k+1}_{\mathrm{low}}\right) \leq 0 & \implies \| x^{k} - x^{k+1}_{\mathrm{low}} \|^{2} - t^{\mathrm{low}} \bar{g}^{\eta_{k}}_{c, \epsilon_{k}}(x^{k})^{\top} \left( x^{k} - x^{k+1}_{\mathrm{low}} \right) \leq 0 \notag \\
        \implies & \left\langle \bar{g}^{\eta_{k}}_{c, \epsilon_{k}}(x^{k}), x^{k+1}_{\mathrm{low}} - x^{k} \right\rangle \leq -\frac{1}{t^{\mathrm{low}}} \| x^{k} - x^{k+1}_{\mathrm{low}} \|^{2}.
    \end{align}
    Recall that $\delta^{k,r}\triangleq \nabla\bar{\theta}^{\eta_k,r}_{c}(x^k)-\bar{g}^{\eta_k,r}_{c,\epsilon_k}(x^k)$ is defined in \eqref{delta-Delta}. Since $\bar{\theta}^{\eta_k,r}_{c}$ is $L_{\nabla \theta}$-smooth by assumption $\mathrm{(B2)}$, the descent lemma yields
    \begin{align*}
        \bar{\theta}^{\eta_{k}, r}_{c} ( x^{k+1}_{\mathrm{low}} ) &\leq \bar{\theta}^{\eta_{k}, r}_{c} (x^{k}) + \left\langle \bar{g}^{\eta_{k}}_{c, \epsilon_{k}}(x^{k}), x^{k+1}_{\mathrm{low}} - x^{k} \right\rangle + \left\langle \delta^{k, r}, x^{k+1}_{\mathrm{low}} - x^{k} \right\rangle + \frac{L_{\nabla \theta}}{2} \| x^{k} - x^{k+1}_{\mathrm{low}} \|^{2} \notag \\
        &\overset{\eqref{g-zo-inner-product}}{\leq} \bar{\theta}^{\eta_{k}, r}_{c} (x^{k}) + \left( \frac{L_{\nabla \theta}}{2} - \frac{1}{t^{\mathrm{low}}} \right) \| x^{k} - x^{k+1}_{\mathrm{low}} \|^{2} + \left\langle \delta^{k, r}, x^{k+1}_{\mathrm{low}} - x^{k} \right\rangle.
    \end{align*}
    Therefore, by the definition of $c_{0}$ given in \eqref{c0-def}, we have
    \allowdisplaybreaks
    \begin{align}\label{SLAM-floor-descent}
        &\bar{\theta}^{\eta_k,r}_{c,\epsilon_k}(x^{k+1}_{\mathrm{low}})-\bar{\theta}^{\eta_k,r}_{c,\epsilon_k}(x^k) + \frac\alpha{t^{\mathrm{low}}} \|x^{k+1}_{\mathrm{low}}-x^k \|^2 \leq -c_0 \| x^{k+1}_{\mathrm{low}}-x^k \|^2 + \left\langle \delta^{k,r}, x^{k+1}_{\mathrm{low}}-x^k\right\rangle \notag \\
        &\qquad \leq -c_0 \| x^{k+1}_{\mathrm{low}} - x^k \|^2 + \left\langle \delta^{k,r}, x^{k+1}_{\mathrm{low}}-x^k\right\rangle + |e_1^{k,r}| + |e_2^{k,r}| \notag \\
        &\qquad \leq -c_0\| x^{k+1}_{\mathrm{low}} - x^k \|^2 + \frac{\| \delta^{k,r} \|^2}{4 c_0} + c_0\| x^{k+1}_{\mathrm{low}} - x^k \|^2 + |e_1^{k,r}| + |e_2^{k,r}| = R_k^r,
    \end{align}
    where the last inequality follows from $c_0>0$ and Young's inequality, $\alpha \in (0,1)$ is defined in the backtracking procedure \eqref{sto-biased-Armijo}, and $e_1^{k,r}$, $e_2^{k,r}$, and $R_k^r$ are defined in \eqref{e1-e2-Rrk}. By assumption $\mathrm{(B3)}$ and Lemmas \ref{lower-level-error} and \ref{slam-suport-lemma-3}, we have
    \begin{equation}\label{SLAM-floor-error-moment}
        \mathbb{E}[|e_1^{k,r}|\mid\mathcal{H}^{r}_{k}] \leq L^y_\theta\sqrt{\epsilon_k}, \quad \mathbb{E}[|e_2^{k,r}|\mid\mathcal{H}^{r}_{k}]\leq L^y_\theta\sqrt{\epsilon_k}, \quad \mathbb{E}[R^{r}_{k}\mid\mathcal{H}^{r}_{k}]\leq B_k,
    \end{equation}
    where $B_{k}$ is defined in \eqref{Bk-def}. If the biased stochastic Armijo linesearch is not satisfied when $t_{k} = t^{\mathrm{low}}$, then batch $r$ is rejected. By \eqref{SLAM-floor-descent} this implies $R^{r}_{k} > e^{k} \triangleq 2B_k$. We denote by $A^{r}_{k}$ the event that the batch $r$ is rejected. The conditional Markov's inequality gives
    \begin{equation}\label{SLAM-batch-failure}
        \mathbb{P} \left[\, A^{r}_{k} \mid \mathcal{H}^{r}_{k} \, \right] \leq \mathbb{P} \left[\, R^{r}_{k} > 2 B_k \mid \mathcal{H}^{r}_{k} \, \right] \leq \frac{\mathbb{E}[R^{r}_{k} \mid \mathcal{H}^{r}_{k}]}{2 B_{k}} \overset{\eqref{SLAM-floor-error-moment}}{\leq} \frac12.
    \end{equation}
    By definition, the event that batch $r$ does not lead to descent is given by the event that the batch indexed by $r-1$ does not lead to descent and the event $A_k^r$, which is the event that batch $r$ is rejected; this implies that
    \begin{equation}\label{event-intersection}
        \left\{\, \tau_{k} > r \, \right\} = \left\{ \, \tau_{k} > r-1 \, \right\} \cap A^{r}_{k}.
    \end{equation}
    Then it follows from \eqref{SLAM-batch-failure} and \eqref{event-intersection}, together with the tower property that
    \begin{equation}\label{SLAM-retry-tail}
    \begin{aligned}
        \mathbb{P}\left[\, \tau_k>r \mid \mathcal{F}_k \, \right] &= \mathbb{E} \left[ \mathbf{1}_{\{ \tau_{k} > r-1 \}} \mathbf{1}_{A^{r}_{k}} \mid \mathcal{F}_k \right] \\
        &= \mathbb{E} \left[ \mathbb{E} \left[ \mathbf{1}_{\{ \tau_{k} > r-1 \}} \mathbf{1}_{A^{r}_{k}} \mid \mathcal{H}^{r}_{k} \right] \mid \mathcal{F}_k \right] \qquad (\mathcal{F}_k \subseteq \mathcal{H}^{r}_{k}) \\
        &= \mathbb{E} \left[ \mathbf{1}_{\{ \tau_{k} > r-1 \}} \, \mathbb{E} \left[ \mathbf{1}_{A^{r}_{k}} \mid \mathcal{H}^{r}_{k} \right] \mid \mathcal{F}_k \right] \qquad \{ \tau_{k} > r-1 \}\in \mathcal{H}^{r}_{k} \\
        &= \mathbb{E} \left[ \mathbf{1}_{\{ \tau_{k} > r-1 \}} \, \mathbb{P} \left[\, A^{r}_{k} \mid \mathcal{H}^{r}_{k} \, \right] \mid \mathcal{F}_k \right] \\
        &\overset{\eqref{SLAM-batch-failure}}{\leq} \frac{1}{2} \mathbb{E} \left[ \mathbf{1}_{\{ \tau_{k} > r-1 \}} \mid \mathcal{F}_k \right] = \frac{1}{2} \mathbb{P} \left[\, \tau_k>r-1 \mid \mathcal{F}_k\, \right] \leq \cdots \leq \frac{1}{2^{r}}.
    \end{aligned}
    \end{equation}
    Therefore, $\tau_k<\infty$ almost surely. The algorithm therefore terminates with $t_k\in[t^{\mathrm{low}},t^{\mathrm{upp}}]$. Moreover, we have that
    \begin{equation*}
        \mathbb{E}\left[\, \tau_k\mid\mathcal{F}_k\, \right] = \sum_{r=0}^{\infty} \mathbb{P}\left[\, \tau_k>r\mid\mathcal{F}_k \, \right] \leq 2.
    \end{equation*}
    Then we complete the proof.
\end{proof}

\begin{remark}\em
    In contrast to the unbiased regime \cite{wang-shanbhag-xie-2025} and Section \ref{Sec-3}, here we introduce a multiple-batch sampling mechanism to address the fundamental difficulty arising from biased estimators. We employ the conditional Markov's inequality to establish almost sure upper and lower bounds on the stepsizes. This construction is needed because Lemma \ref{lower-level-error} controls the error between the inexact and exact lower-level solutions only in expectation, without a suitable bound in the almost sure sense. $\hfill \Box$
\end{remark}

Again, we denote the number of times that the stepsize sequence $\{ t_{k} \}_{k=0}^{K-1}$ changes during $K$ iterations by
\begin{equation}\label{Q-definition}
    T(K) \triangleq \sum_{k=0}^{K-1} \mathbf{1}_{\{t_k\neq t_{k-1}\}}, \quad t_{-1} \triangleq t^{\mathrm{upp}}, \quad Q \triangleq 1 + \left\lceil \log_{1/\beta} \frac{t^{\mathrm{upp}}}{t^{\mathrm{low}}} \right\rceil,
\end{equation}
where $Q$ denotes the number of backtracking steps from $t^{\mathrm{upp}}$ to $t^{\mathrm{low}}$. Let $J(K)+1$ denote the total number of complete cycles during $K$ iterations, where
\begin{equation*}
J(K)\triangleq \min\left\{l\in\mathbb{N}_0:\sum_{j=0}^{l}p_j\geq K\right\}.
\end{equation*}
By our biased stochastic Armijo linesearch procedure, we may get the upper bound of $T(K)$ immediately. It is worth noting that, unlike in Section \ref{Sec-3}, since we require knowledge of $t^{\mathrm{low}}$ here, we can directly use $Q$ to obtain a more intuitive upper bound on $T(K)$.

\begin{lemma}[Bounds on $T(K)$]\label{SLAM-bound-TK-lemma}\em
For Algorithm~\ref{SLAM}, we have
\begin{equation}\label{SLAM-step-change-bound}
    T(K) \leq (J(K)+1)Q,
\end{equation}
where $Q$ is defined in \eqref{Q-definition}.
\end{lemma}

For convenience, we denote the final accepted batch by $\tau_k$ throughout the analysis. The following lemma is crucial in our subsequent analysis, which connects moment bounds for a new generated batch index $r$ conditional on $\mathcal{H}^{r}_{k}$, to bounds for the final accepted batch index $\tau_{k}$ conditional on $\mathcal{F}_k$. It is important because the final accepted batch index $\tau_{k}$ is random (since we do not how many batches needed until Armijo's rule is satisfied), so the moment bounds for a new batch cannot be applied directly to $\tau_{k}$. The proposition accounts for this selection effect through the bound $\mathbb{E}[\tau_k \mid \mathcal{F}_k] \leq 2$ we have established in Lemma \ref{SLAM-stepsize-lemma}.

\begin{lemma}\label{general-lemma}\em
    Let $W^{r}_{k}$ be a nonnegative quantity associated with batch $r$ at iteration $k$. If
    \begin{equation*}
        \mathbb{E}[ W^{r}_{k} \mid \mathcal{H}^{r}_{k}]\leq w_k
    \end{equation*}
    for a nonnegative deterministic $w_k$, then we have
    \begin{equation*}
        \mathbb{E}[ W^{\tau_{k}}_{k} \mid \mathcal{F}_k]\leq 2w_k.
    \end{equation*}
\end{lemma}
\begin{proof}
    By the conditional tower property, we have
    \begin{align*}
        \mathbb{E}[W^{\tau_k}_{k}\mid\mathcal{F}_k] &= \sum_{r\geq1} \mathbb{E}[\mathbf1_{\{\tau_k=r\}} W^{r}_{k} \mid\mathcal{F}_k] \\ &\leq \sum_{r\geq1}\mathbb{E}[\mathbf1_{\{\tau_k\geq r\}} W^{r}_{k} \mid\mathcal{F}_k] \\
        &= \sum_{r\geq1}\mathbb{E}[ \mathbb{E} [ \mathbf1_{\{\tau_k\geq r\}}W^{r}_{k} \mid \mathcal{H}^{r}_{k} ] \mid \mathcal{F}_k] \qquad (\mathcal{F}_k \subseteq \mathcal{H}^{r}_{k}) \\
        &= \sum_{r\geq1}\mathbb{E}[\mathbf1_{\{\tau_k\geq r\}} \mathbb{E}[W^{r}_{k} \mid \mathcal{H}^{r}_{k} ] \mid\mathcal{F}_k] \qquad \{ \tau_{k} \geq r \}\in \mathcal{H}^{r}_{k} \\
        &\leq w_k\sum_{r\geq1}\mathbb{P}\left[\, \tau_k\geq r\mid\mathcal{F}_k\, \right] = w_k\mathbb{E}[\tau_k\mid\mathcal{F}_k] \overset{\eqref{tauk-P-E}}{\leq} 2w_k
    \end{align*}
    since the event $\{\tau_k \geq r\}$ is known before batch $r$ is generated when the event $\{\tau_k > r-1\}$ is happened. Then we complete the proof.
\end{proof}

The following corollary is an immediate consequence of Lemma \ref{slam-suport-lemma-3} and Lemma \ref{general-lemma} and hence we omit the proof.

\begin{corollary}\em
    Suppose Assumption $\mathrm{B}$ holds. For any $k\geq 0$, we have the following two bounds
    \begin{equation}\label{SLAM-accepted-gradient-errors}
        \mathbb{E}[ \| \delta^{k,\tau_k} \| \mid \mathcal{F}_k ] \leq 2 a_k, \qquad \mathbb{E} [ \| \Delta^{k,\tau_k} \|^2 \mid \mathcal{F}_k ]\leq 2 v_k,
    \end{equation}
    where $a_{k}$ and $v_{k}$ are defined in \eqref{vk-def}.
\end{corollary}

Similar to Section \ref{Sec-3}, in order to facilitate our analysis, we introduce the auxiliary iterate
\begin{equation}\label{bar-x-k+1}
    \bar{x}^{k+1} \triangleq \Pi_X[x^k-t_{k-1}\nabla\theta_{c}^{\eta_k}(x^k)], \quad t_{-1}= t^{\mathrm{upp}}.
\end{equation}
We can see that $\bar{x}^{k+1}$ is $\mathcal{F}_k$-measurable. The following two lemmas are useful in the analysis of Proposition \ref{slam-convergence-key-proposition-2}.

\begin{lemma}\label{diff-xk+1-barxk+1}\em
    Consider $\bar{x}^{k+1}$ defined in \eqref{bar-x-k+1}. Then for any $k \geq 0$, we have
    \begin{equation}\label{proposition-1.2-eqn1}
        \mathbb{E}\left[\, \left\| x^{k+1}-\bar{x}^{k+1} \right\|^2 \, \right] \leq 2D_X^2 \mathbb{P} \{ t_k\neq t_{k-1} \} + 4( t^{\mathrm{upp}})^2 v_k,
    \end{equation}
    where $v_k$ is given in \eqref{vk-def}.
\end{lemma}
\begin{proof}
    We define the intermediate point
    \begin{equation*}
        \tilde{x}^{k+1} \triangleq \Pi_X \left[ x^k-t_k \nabla\theta_c^{\eta_k}(x^k) \right].
    \end{equation*}
    By nonexpansivity of the projection operator and $t_k\leq t^{\mathrm{upp}}$,
    \begin{equation*}
        \|x^{k+1}-\tilde{x}^{k+1}\| \leq t_k \left\| \bar{g}_{c,\epsilon_k}^{\eta_k,\tau_k}(x^k) -\nabla\theta_c^{\eta_k}(x^k) \right\| =t_k\|\Delta^{k,\tau_k}\| \leq t^{\mathrm{upp}} \| \Delta^{k,\tau_k} \|,
    \end{equation*}
    where $\Delta^{k,\tau_k}$ defined in \eqref{delta-Delta}. Furthermore, we know from \eqref{bar-x-k+1} that $\tilde{x}^{k+1}=\bar{x}^{k+1}$ whenever $t_k=t_{k-1}$. Consequently, we have
    \begin{equation*}
        \|\tilde{x}^{k+1}-\bar{x}^{k+1}\|^2 \leq D_X^2\mathbf{1}_{\{t_k\neq t_{k-1}\}},
    \end{equation*}
    which implies that
    \begin{equation*}
        \|x^{k+1}-\bar{x}^{k+1}\|^2 \leq 2\|x^{k+1}-\tilde{x}^{k+1}\|^2 +2\|\tilde{x}^{k+1}-\bar{x}^{k+1}\|^2 \leq 2(t^{\mathrm{upp}})^2\|\Delta^{k,\tau_k}\|^2 + 2D_X^2\mathbf{1}_{\{t_k\neq t_{k-1}\}}.
    \end{equation*}
    Taking expectations yields
    \begin{equation*}
        \mathbb{E} \left[\|x^{k+1}-\bar{x}^{k+1}\|^2\right] \leq 2D_X^2\mathbb{P} \{ t_k\neq t_{k-1} \} + 2(t^{\mathrm{upp}})^2 \mathbb{E} \left[\|\Delta^{k,\tau_k}\|^2\right] \overset{\eqref{SLAM-accepted-gradient-errors}}{\leq} 2D_X^2 \mathbb{P} \{ t_k\neq t_{k-1} \} + 4( t^{\mathrm{upp}})^2 v_k,
    \end{equation*}
    which completes the proof.
\end{proof}

We denote the difference between $\theta^{\eta_{k}}_{c}$ and its sampled counterpart $\bar{\theta}^{\eta_{k}, r}_{c}$ for batch $r$ by $h^{r}_{k}$. The following lemma provides an upper bound on the expected difference between the absolute values of $h^{r}_{k}$ at $\bar{x}^{k+1}$ and $x^{k}$.

\begin{lemma}\label{two-h-diff}\em
    Consider $\bar{x}^{k+1}$ defined in \eqref{bar-x-k+1}. Suppose Assumption $\mathrm{B}$ holds. Define $h^{r}_{k}(x) \triangleq \theta_{c}^{\eta_k}(x)-\bar{\theta}_c^{\eta_k,r}(x)$ for batch $r$ at iteration $k$. Then for any $k \geq 0$, we have
    \begin{align}\label{proposition-1.2-eqn2}
        \mathbb{E}[|h^{\tau_{k}}_{k}(\bar{x}^{k+1})-h^{\tau_{k}}_{k}(x^k)|] &\leq \frac{\alpha}{2 t^{\mathrm{upp}}} \mathbb{E} [ \| x^{k+1}-x^k \|^2 ] + \frac{2 t^{\mathrm{upp}} L_\theta^2}{\alpha N_k} \notag \\
        &+ L_{\nabla \theta} D_X^2 \mathbb{P} \{ t_k\neq t_{k-1} \} + 2L_{\nabla \theta} ( t^{\mathrm{upp}})^2 v_k + \frac{2L_\theta^2}{ L_{\nabla \theta} N_k}.
    \end{align}
\end{lemma}
\begin{proof}
    Consider iteration $k$ and batch $r$. For $i = 1, \ldots, N_{k}$, we define the sample increment between $\bar{x}^{k+1}$ and $x^{k}$ by
    \begin{equation*}
        D^{i,k,r} \triangleq \tilde{\theta}_c \left( \bar{x}^{k+1}+\eta_k u^{i,k,r}, y_c(\bar{x}^{k+1} + \eta_k u^{i,k,r}), \xi^{i,k,r} \right) - \tilde{\theta}_c \left( x^k+\eta_k u^{i,k,r}, y_c(x^k+\eta_k u^{i,k,r}), \xi^{i,k,r} \right).
    \end{equation*}
    We know from assumption $\mathrm{(B2)}$ that
    \begin{equation*}
       |D^{i,k,r}| \leq L_\theta\|\bar{x}^{k+1}-x^k\|.
    \end{equation*}
    By the definition of $h^{r}_{k}$, we have
    \begin{equation*}
        \mathbb{E}[D^{i,k,r}\mid\mathcal{H}_k^r] = \theta_c^{\eta_k}(\bar{x}^{k+1}) - \theta_c^{\eta_k}(x^k) \implies h^{r}_{k}(\bar{x}^{k+1})-h^{r}_{k}(x^k) = \frac{1}{N_k}\sum_{i=1}^{N_k} \left( \mathbb{E}[D^{i,k,r}\mid\mathcal{H}_k^r] - D^{i,k,r} \right).
    \end{equation*}
    Therefore, we have
    \begin{equation*}
        \mathbb{E} \left[ |h^{r}_{k}(\bar{x}^{k+1})-h^{r}_{k}(x^k) |^2 \,\middle|\,\mathcal{H}_k^r \right] = \frac{1}{N_k^2}\sum_{i=1}^{N_k} \mathrm{Var}(D^{i,k,r}\mid\mathcal{H}_k^r) \leq \frac{1}{N_k^2}\sum_{i=1}^{N_k} \mathbb{E}[|D^{i,k,r}|^2\mid\mathcal{H}_k^r] \leq \frac{L_\theta^2}{N_k} \|\bar{x}^{k+1}-x^k\|^2.
    \end{equation*}
    Finally, the conditional Jensen's inequality gives
    \begin{equation}\label{bd:h_k_r}
        \mathbb{E} \left[ |h^{r}_{k}(\bar{x}^{k+1})-h^{r}_{k}(x^k)| \mid \mathcal{H}_k^r \right] \leq \left( \mathbb{E} \left[ | h^{r}_{k}(\bar{x}^{k+1})-h^{r}_{k}(x^k) |^2 \,\middle|\,\mathcal{H}_k^r \right] \right)^{1/2} \leq \frac{L_\theta}{\sqrt{N_k}} \|\bar{x}^{k+1}-x^k\|.
    \end{equation}
    By applying Lemma \ref{general-lemma} and Young's inequality, we have
    \begin{align*}
        &\mathbb{E} \left[ |h^{\tau_{k}}_{k}(\bar{x}^{k+1})-h^{\tau_{k}}_{k}(x^k)| \mid \mathcal{F}_k \right] \overset{\eqref{bd:h_k_r}, \textrm{Lemma } \ref{general-lemma}}{\leq} \frac{2L_\theta}{\sqrt{N_k}} \| \bar{x}^{k+1}-x^k \|  \\
        & \qquad \leq \frac{2L_\theta}{\sqrt{N_k}} \| x^{k+1}-x^k \| + \frac{2L_\theta}{\sqrt{N_k}} \| x^{k+1}-\bar{x}^{k+1} \| \\
        & \qquad \leq \left( \frac{\alpha}{2 t^{\mathrm{upp}}} \| x^{k+1}-x^k \|^2 + \frac{2 t^{\mathrm{upp}} L_\theta^2}{\alpha N_k} \right) + \left( \frac{ L_{\nabla \theta}}{2} \| x^{k+1}-\bar{x}^{k+1} \|^2\right) + \left(\frac{2L_\theta^2}{ L_{\nabla \theta} N_k} \right).
    \end{align*}
Then by taking unconditional expectations on both sides, we may derive that
\begin{align*}
    &\mathbb{E} \left[ |h^{\tau_{k}}_{k}(\bar{x}^{k+1})-h^{\tau_{k}}_{k}(x^k)| \right] \leq \frac{\alpha}{2 t^{\mathrm{upp}}} \mathbb{E}[ \| x^{k+1}-x^k \|^2 ] + \frac{2 t^{\mathrm{upp}} L_\theta^2}{\alpha N_k} + \frac{ L_{\nabla \theta}}{2} \mathbb{E}[ \| x^{k+1} - \bar{x}^{k+1} \|^2] + \frac{2L_\theta^2}{ L_{\nabla \theta} N_k} \\
    &\qquad \overset{\textrm{Lemma } \ref{diff-xk+1-barxk+1}}{\leq} \frac{\alpha}{2 t^{\mathrm{upp}}} \mathbb{E}[\| x^{k+1}-x^k \|^2] + \frac{2 t^{\mathrm{upp}} L_\theta^2}{\alpha N_k} + \frac{ L_{\nabla \theta}}{2} \left( 2D_X^2 \mathbb{P} \{ t_k\neq t_{k-1} \} + 4( t^{\mathrm{upp}})^2 v_k \right) + \frac{2L_\theta^2}{ L_{\nabla \theta} N_k} \\
    &\qquad = \frac{\alpha}{2 t^{\mathrm{upp}}} \mathbb{E} [ \| x^{k+1}-x^k \|^2] + \frac{2 t^{\mathrm{upp}} L_\theta^2}{\alpha N_k} + L_{\nabla \theta} D_X^2 \mathbb{P} \{ t_k\neq t_{k-1} \} + 2L_{\nabla \theta} ( t^{\mathrm{upp}})^2 v_k + \frac{2L_\theta^2}{ L_{\nabla \theta} N_k},
\end{align*}
which completes the proof.
\end{proof}

\subsection{Convergence analysis}

We define residuals $G_{t}(x)$, $G^{\eta}_{t}(x)$, $\bar{G}^{\eta}_{t}(x)$, and $\bar{G}^{\eta}_{t, \epsilon}(x)$ as
\begin{align*}
    G_{t}(x) &\triangleq t^{-1}(x - \Pi_{X}[x - t \nabla \theta_{c}(x)]), \quad G^{\eta}_{t}(x) \triangleq t^{-1}(x - \Pi_{X}[x - t \nabla \theta^{\eta}_{c}(x)]), \\
    \bar{G}^{\eta, r}_{t}(x) &\triangleq t^{-1}(x - \Pi_{X}[x - t \bar{g}^{\eta, r}_{c}(x)]), \quad \bar{G}^{\eta, r}_{t, \epsilon}(x) \triangleq t^{-1}(x - \Pi_{X}[x - t \bar{g}^{\eta, r}_{c, \epsilon}(x)]),
\end{align*}
respectively. Now we provide a roadmap of the next steps.

~

\begin{mdframed}[
  linewidth=0.8pt,
  innerleftmargin=8pt,
  innerrightmargin=8pt,
  innertopmargin=6pt,
  innerbottommargin=6pt
]
\noindent\textbf{Roadmap.} By taking expectations in \eqref{sto-biased-Armijo} with $r = \tau_{k}$ (accepted batch), we obtain
\begin{equation}\label{exp-sto-biased-Armijo}
    \underbrace{\mathbb{E}\left[\bar{\theta}^{\eta_k,\tau_k}_{c,\epsilon_k}(x^{k+1})-\bar{\theta}^{\eta_k,\tau_k}_{c,\epsilon_k}(x^k)\right]}_{\textrm{Term 1}} \leq -\alpha \underbrace{\mathbb{E}\left[t_k\left\|\bar{G}^{\eta_k,\tau_k}_{t_k,\epsilon_k}(x^k)\right\|^2\right]}_{\textrm{Term 2}} + e^k.
\end{equation}
Our goal is to relate Terms $1$ and $2$ to $\mathbb{E}\left[\theta_{c}(x^{k+1})-\theta_{c}(x^k)\right]$ and $\mathbb{E} \left[ \left\| G_{t_k}(x^k) \right\|^2 \right]$, respectively:
\begin{itemize}
    \item $\textrm{Term 1} \xrightarrow{\text{Proposition \ref{slam-convergence-key-proposition-2}}} \mathbb{E}\left[\theta^{\eta_k}_{c}(x^{k+1})-\theta^{\eta_k}_{c}(x^k)\right] \xrightarrow{\text{Lemma \ref{RS-property-grad}}} \mathbb{E}\left[\theta_{c}(x^{k+1})-\theta_{c}(x^k)\right]$; and
    \item $\textrm{Term 2} \xrightarrow{\text{Proposition \ref{slam-convergence-key-proposition-3}}} \mathbb{E} \left[ \left\| G_{t_k}^{\eta_k}(x^k) \right\|^2 \right] \xrightarrow{\text{Lemma \ref{RS-property-grad}}} \mathbb{E} \left[ \left\| G_{t_k}(x^k) \right\|^2 \right]$.
\end{itemize}
\end{mdframed}

~

Let $\ell^{r}_k$ denote the total number of backtracking required when the batch $r$ is employed at iteration $k$. To facilitate our analysis in Lemma \ref{slam-convergence-key-proposition-1}, we denote the point generated by the $\ell$-th $(\ell = 0, \cdots, \ell^{r}_k)$ backtracking by $z^{r}_{k, \ell}$. In particular, $z^{r}_{k, 0} \triangleq x^k$, $z^{r}_{k, 1}$ and $z^{r}_{k, 2}$ are
\begin{align*}
    z^{r}_{k, 1} &\triangleq \Pi_X \left[ x^k-t_k^{\mathrm{init}} \, \bar{g}_{c,\epsilon_k}^{\eta_k,r}(x^k) \right], ~ \textrm{and} \\
    z^{r}_{k, 2} &\triangleq \Pi_X \left[ x^k-\max\{\beta t_k^{\mathrm{init}},t^{\mathrm{low}}\} \, \bar{g}_{c,\epsilon_k}^{\eta_k,r}(x^k) \right],
\end{align*}
respectively. By definition, it is easy to see that $\ell^{r}_k \leq Q$ where $Q$ is given in \eqref{Q-definition}.

\begin{lemma}\label{slam-convergence-key-proposition-1}\em
    Suppose Assumption $\mathrm{B}$ holds. Then for any $k \geq 0$, we have
    \begin{equation*}
        \mathbb{E} \left[ \left| \bar{\theta}_c^{\eta_k,\tau_k}(x^{k+1})-\bar{\theta}_{c,\epsilon_k}^{\eta_k,\tau_k}(x^{k+1}) \right| + \left| \bar{\theta}_c^{\eta_k,\tau_k}(x^k)-\bar{\theta}_{c,\epsilon_k}^{\eta_k,\tau_k}(x^k) \right| \,\middle\vert\, \mathcal{F}_k\right] \leq 2(Q+1) L^{y}_{\theta}\sqrt{\epsilon_k},
    \end{equation*}
    where $Q$ is defined in \eqref{Q-definition}.
\end{lemma}
\begin{proof}
For batch $r$ and any $\ell=0, \ldots, \ell^{r}_k$, by assumption $\mathrm{(B3)}$, we have
\begin{equation*}
    \left| \bar{\theta}_c^{\eta_k,r}(z^{r}_{k, \ell})  - \bar{\theta}_{c,\epsilon_k}^{\eta_k,r}(z^{r}_{k, \ell}) \right| \leq \frac{L_\theta^y}{N_k}\sum_{i=1}^{N_k} \| y_c(z^{r}_{k,\ell}+\eta_k u^{i,k,r}) -y_c^{\epsilon_k,r}(z^{r}_{k,\ell}+\eta_k u^{i,k,r}) \|.
\end{equation*}
Consequently by Lemma \ref{lower-level-error}, we obtain
\begin{equation*}
    \mathbb{E} \left[ \left| \bar{\theta}_c^{\eta_k,r}(z^{r}_{k, \ell})  - \bar{\theta}_{c,\epsilon_k}^{\eta_k,r}(z^{r}_{k, \ell}) \right| \,\middle\vert\, \mathcal{H}^{r}_{k} \right] \leq \frac{L_\theta^y}{N_k} \sum_{i=1}^{N_k}\sqrt{\epsilon_k} =L_\theta^y\sqrt{\epsilon_k}, \quad \forall \ell = 0, \cdots, \ell^{r}_{k}.
\end{equation*}
We define $W^{r}_{k} \triangleq \sum_{\ell=0}^{\ell^{r}_{k}} \left| \bar{\theta}_c^{\eta_k,r}(z^{r}_{k, \ell})  - \bar{\theta}_{c,\epsilon_k}^{\eta_k,r}(z^{r}_{k, \ell}) \right|$. Then we have
\begin{equation*}
    \mathbb{E} \left[ W^{r}_{k} \,\middle\vert\, \mathcal{H}^{r}_{k} \right] = \sum_{\ell=0}^{\ell^{r}_{k}} \mathbb{E} \left[ \left| \bar{\theta}_c^{\eta_k,r}(z^{r}_{k, \ell})  - \bar{\theta}_{c,\epsilon_k}^{\eta_k,r}(z^{r}_{k, \ell}) \right| \,\middle\vert\, \mathcal{H}^{r}_{k} \right] \leq(\ell^{r}_{k}+1)L_\theta^y\sqrt{\epsilon_k} \leq(Q+1)L_\theta^y\sqrt{\epsilon_k}.
\end{equation*}
By Lemma \ref{general-lemma}, it yields that
\begin{equation*}
    \mathbb{E}[W^{\tau_k}_{k} \mid \mathcal{F}_k] \leq 2(Q+1)L_\theta^y\sqrt{\epsilon_k}.
\end{equation*}
By definition, we know that $x^{k+1} \triangleq z^{\tau_{k}}_{k, \ell^{\tau_{k}}_k}$ and $x^{k} = z^{\tau_{k}}_{k, 0}$ (batch $\tau_{k}-1$ is discarded and no update). Therefore, we may establish that
\begin{equation*}
    \begin{aligned}
        &\mathbb{E} \left[ \left| \bar{\theta}_c^{\eta_k,\tau_k}(x^{k+1})-\bar{\theta}_{c,\epsilon_k}^{\eta_k,\tau_k}(x^{k+1}) \right| + \left| \bar{\theta}_c^{\eta_k,\tau_k}(x^k)-\bar{\theta}_{c,\epsilon_k}^{\eta_k,\tau_k}(x^k) \right| \,\middle\vert\, \mathcal{F}_k \right] \\
        &= \mathbb{E} \left[ \left|\bar{\theta}_c^{\eta_k,\tau_k}(z^{\tau_{k}}_{k, \ell^{\tau_{k}}_k})-\bar{\theta}_{c,\epsilon_k}^{\eta_k,\tau_k}(z^{\tau_{k}}_{k, \ell^{\tau_{k}}_k}) \right| + \left| \bar{\theta}_c^{\eta_k,\tau_k}(z^{\tau_{k}}_{k, 0})-\bar{\theta}_{c,\epsilon_k}^{\eta_k,\tau_k}(z^{\tau_{k}}_{k, 0}) \right| \,\middle\vert\, \mathcal{F}_k\right] \\
        &\leq \mathbb{E}[W^{\tau_k}_{k} \mid \mathcal{F}_k] \leq 2(Q+1)L_\theta^y\sqrt{\epsilon_k},
    \end{aligned}
\end{equation*}
which completes the proof.
\end{proof}

The following three propositions lay the foundation for the establishment of the final main result of our \textbf{B-SLAM} scheme.

\begin{proposition}[Comparison of objective increments]\label{slam-convergence-key-proposition-2}\em
Suppose Assumption $\mathrm{B}$ holds. Define
\begin{equation*}
\begin{aligned}
    C_N &\triangleq \frac{2 t^{\mathrm{upp}} L_\theta^2}{\alpha} +\frac{2L_\theta^2+ L_\theta^2}{ L_{\nabla \theta}},\\
    E_{k} &\triangleq 2(Q+1) L^{y}_{\theta}\sqrt{\epsilon_k}+\frac{C_N}{N_k} + 8 L_{\nabla \theta}( t^{\mathrm{upp}})^2v_k+4 L_{\nabla \theta} D_X^2\mathbb{P} \{ t_k\neq t_{k-1} \}.
\end{aligned}
\end{equation*}
Then for any $k \geq 0$, we have
\begin{equation}\label{slam-convergence-key-proposition-2-bound}
    \mathbb{E} \left[ \theta_{c}^{\eta_k}(x^{k+1})-\theta_{c}^{\eta_k}(x^k) \right] \leq \mathbb{E} \left[ \bar{\theta}_{c,\epsilon_k}^{\eta_k,\tau_k}(x^{k+1})-\bar{\theta}_{c,\epsilon_k}^{\eta_k,\tau_k}(x^k) \right] +\frac{\alpha}{2 t^{\mathrm{upp}}} \mathbb{E} \left[ \| x^{k+1}-x^k \|^2 \right] + E_{k}.
\end{equation}
\end{proposition}
\begin{proof}
Consider again the function $h^{r}_{k}(x)\triangleq \theta_{c}^{\eta_k}(x)-\bar{\theta}_c^{\eta_k,r}(x)$ defined in Lemma \ref{two-h-diff}. Then we have
\begin{align}\label{slam-convergence-key-proposition-2-eqn1}
    & \mathbb{E} \left[ \theta_c^{\eta_k}(x^{k+1})-\theta_c^{\eta_k}(x^k) \right] - \mathbb{E} \left[ \bar{\theta}_{c,\epsilon_k}^{\eta_k,\tau_k}(x^{k+1}) - \bar{\theta}_{c,\epsilon_k}^{\eta_k,\tau_k}(x^k) \right] \notag \\
    & = \mathbb{E} \left[ h^{\tau_k}_{k}(x^{k+1})-h^{\tau_k}_{k}(x^k) \right] + \mathbb{E} \left[ \bar{\theta}_c^{\eta_k,\tau_k}(x^{k+1}) -\bar{\theta}_{c,\epsilon_k}^{\eta_k,\tau_k}(x^{k+1}) \right] - \mathbb{E} \left[ \bar{\theta}_c^{\eta_k,\tau_k}(x^k) -\bar{\theta}_{c,\epsilon_k}^{\eta_k,\tau_k}(x^k) \right] \notag \\
    & \leq \mathbb{E} \left[ h^{\tau_k}_{k}(x^{k+1})-h^{\tau_k}_{k}(x^k) \right] + 2(Q+1) L^{y}_{\theta}\sqrt{\epsilon_k},
\end{align}
where the last inequality follows from Lemma \ref{slam-convergence-key-proposition-1}. Since $\bar{x}^{k+1}$ is adapted to $\mathcal{F}_k$, Lemma~\ref{slam-suport-lemma-2} at $\bar{x}^{k+1}$ together with Lemma \eqref{general-lemma} give
\begin{equation}\label{slam-convergence-key-proposition-2-eqn6}
    \mathbb{E} \left[ \left\| \nabla h^{\tau_k}_{k}(\bar{x}^{k+1}) \right\|^2 \,\middle\vert\, \mathcal{F}_{k} \right] \leq\frac{2 L_\theta^2}{N_k} \implies \mathbb{E} \left[ \left\| \nabla h^{\tau_k}_{k}(\bar{x}^{k+1}) \right\|^2 \right] \leq \frac{2 L_\theta^2}{N_k}.
\end{equation}
The $L_{\nabla\theta}$-smoothness of $\theta_c^{\eta_k}$ and $\bar{\theta}_c^{\eta_k,\tau_k}$ from assumption $\mathrm{(B2)}$ gives
\begin{equation*}
    \theta_c^{\eta_k}(x^{k+1}) - \theta_c^{\eta_k}(\bar{x}^{k+1}) \leq \left\langle \nabla\theta_c^{\eta_k}(\bar{x}^{k+1}), x^{k+1}-\bar{x}^{k+1} \right\rangle +\frac{L_{\nabla\theta}}{2} \|x^{k+1}-\bar{x}^{k+1}\|^2
\end{equation*}
and
\begin{equation*}
    \bar{\theta}_c^{\eta_k,\tau_k}(x^{k+1}) - \bar{\theta}_c^{\eta_k,\tau_k}(\bar{x}^{k+1}) \geq \left\langle \nabla\bar{\theta}_c^{\eta_k,\tau_k}(\bar{x}^{k+1}), x^{k+1}-\bar{x}^{k+1} \right\rangle - \frac{L_{\nabla\theta}}{2} \|x^{k+1}-\bar{x}^{k+1}\|^2.
\end{equation*}
Recalling that $h_k^{\tau_k} \triangleq \theta_c^{\eta_k} -\bar{\theta}_c^{\eta_k,\tau_k}$, subtracting the second inequality from the first yields
\begin{align}\label{slam-convergence-key-proposition-2-eqn7}
    h^{\tau_k}_{k}(x^{k+1})-h^{\tau_k}_{k}(\bar{x}^{k+1}) &\leq \left\langle \nabla h^{\tau_k}_{k}(\bar{x}^{k+1}), x^{k+1}-\bar{x}^{k+1} \right\rangle + L_{\nabla \theta} \| x^{k+1}-\bar{x}^{k+1} \|^2 \notag \\
    &\leq\frac1{2 L_{\nabla \theta}}\left\| \nabla h^{\tau_k}_{k}(\bar{x}^{k+1}) \right\|^2 + \frac{3 L_{\nabla \theta}}{2} \| x^{k+1}-\bar{x}^{k+1} \|^2.
\end{align}
From \eqref{proposition-1.2-eqn1}, \eqref{proposition-1.2-eqn2}, \eqref{slam-convergence-key-proposition-2-eqn6}, and \eqref{slam-convergence-key-proposition-2-eqn7}, we may establish
\begin{align}\label{slam-convergence-key-proposition-2-eqn8}
    &\mathbb{E} \left[ h^{\tau_k}_{k}(x^{k+1})-h^{\tau_k}_{k}(x^k) \right] \leq \frac{\alpha}{2 t^{\mathrm{upp}}} \mathbb{E} \left[ \| x^{k+1}-x^k \|^2 \right] + \frac{C_N}{N_k}+2 L_{\nabla \theta} \mathbb{E} \left[ \| x^{k+1}-\bar{x}^{k+1} \|^2 \right] \notag \\
    & \qquad \leq\frac{\alpha}{2 t^{\mathrm{upp}}} \mathbb{E} \left[ \| x^{k+1}-x^k \|^2 \right] + \frac{C_N}{N_k}+8 L_{\nabla \theta}( t^{\mathrm{upp}})^2v_k +4 L_{\nabla \theta} D_X^2\mathbb{P} \{ t_k\neq t_{k-1} \}.
\end{align}
By plugging \eqref{slam-convergence-key-proposition-2-eqn8} into \eqref{slam-convergence-key-proposition-2-eqn1}, we prove the claim.
\end{proof}

\begin{proposition}[Comparison of residuals]\label{slam-convergence-key-proposition-3}\em
    Suppose Assumption $\mathrm{B}$ holds. Then for any $k \geq 0$, we have
    \begin{equation}\label{residual-bound}
        -\mathbb{E} \left[\left\| \bar{G}_{t_k,\epsilon_k}^{\eta_k,\tau_k}(x^k) \right\|^2 \right] \leq -\frac{1}{2} \mathbb{E} \left[ \left\| G_{t_k}^{\eta_k}(x^k) \right\|^2 \right] + 2v_k,
    \end{equation}
    where $v_{k}$ is defined in \eqref{vk-def}.
\end{proposition}
\begin{proof}
    Recall that $\Delta^{k, r} \triangleq \nabla\theta^{\eta_k}_{c}(x^k)-\bar{g}^{\eta_k,r}_{c,\epsilon_k}(x^k)$ given in \eqref{delta-Delta}. By invoking \cite[Lemma 5-(ii)]{wang-shanbhag-xie-2025}, we have
    \begin{align*}
        \mathbb{E} \left[ \left\| G_{t_k}^{\eta_k}(x^k) \right\|^2 \right] &\leq 2 \mathbb{E} \left[ \left\| G_{t_k}^{\eta_k}(x^k) - \bar{G}_{t_k,\epsilon_k}^{\eta_k,\tau_k}(x^k) \right\|^2 \right] + 2 \mathbb{E} \left[ \left\| \bar{G}_{t_k,\epsilon_k}^{\eta_k,\tau_k}(x^k) \right\|^2 \right] \\
        &\leq 2 \mathbb{E} \left[ \left\| \Delta^{k, \tau_{k}} \right\|^2 \right] + 2 \mathbb{E} \left[ \left\| \bar{G}_{t_k,\epsilon_k}^{\eta_k,\tau_k}(x^k) \right\|^2 \right] \overset{\eqref{SLAM-accepted-gradient-errors}}{\leq} 4v_k + 2 \mathbb{E} \left[ \left\| \bar{G}_{t_k,\epsilon_k}^{\eta_k,\tau_k}(x^k) \right\|^2 \right],
    \end{align*}
    where $v_{k}$ is defined in \eqref{vk-def}. By rearranging the term, we may obtain the desired result.
\end{proof}

\begin{proposition}\label{slam-convergence-key-proposition-4}\em
Suppose Assumption $\mathrm{B}$ holds. Then for any $k \geq 0$, we have
\begin{align}\label{slam-convergence-key-proposition-4-main}
    &\frac{\alpha t^{\mathrm{low}}}{8} \mathbb{E} \left[ \left\| G_{ t^{\mathrm{upp}}}(x^k) \right\|^2 \right] \leq \mathbb{E} \left[ \theta_{c}(x^k) \right] - \mathbb{E} \left[ \theta_{c}(x^{k+1}) \right] + 2L_\theta\eta_k+\frac{\alpha t^{\mathrm{upp}} L_{\nabla \theta}^2n^2}{4}\eta_k^2 + 2 B_k \notag \\
    & \qquad + 2(Q+1) L^{y}_{\theta}\sqrt{\epsilon_k} + \frac{C_N}{N_k}+(8 L_{\nabla \theta}( t^{\mathrm{upp}})^2+\alpha t^{\mathrm{upp}}) v_k + 4 L_{\nabla \theta} D_X^2\mathbb{P} \{ t_k\neq t_{k-1} \},
\end{align}
where $B_k$ is defined in \eqref{Bk-def}, $Q$ is defined in \eqref{Q-definition}, and $C_N$ is defined in Proposition \ref{slam-convergence-key-proposition-2}.
\end{proposition}
\begin{proof}
    Since the final batch $\tau_{k}$ satisfies \eqref{sto-biased-Armijo} with $e^{k} = 2B_{k}$ at the $k$-th iteration, we have
    \begin{equation}\label{slam-convergence-key-proposition-4-eqn1}
        \bar{\theta}_{c,\epsilon_k}^{\eta_k,\tau_k}(x^{k+1}) - \bar{\theta}_{c,\epsilon_k}^{\eta_k,\tau_k}(x^k) \leq - \frac{\alpha}{t_k}\|x^{k+1}-x^k\|^2 + 2 B_k = - \alpha t_k \left\|\bar{G}_{t_k,\epsilon_k}^{\eta_k,\tau_k}(x^k)\right\|^2 + 2 B_k,
    \end{equation}
    where the last equality is due to the fact that $x^{k+1} - x^k = -t_k\bar{G}_{t_k,\epsilon_k}^{\eta_k, \tau_k}(x^k)$. By Propositions \ref{slam-convergence-key-proposition-2} and \ref{slam-convergence-key-proposition-3}, we have that
    \begin{align}\label{slam-convergence-key-proposition-4-eqn2}
        \mathbb{E}[\theta_{c}^{\eta_k}(x^{k+1})-\theta_{c}^{\eta_k}(x^k)] &\overset{\eqref{slam-convergence-key-proposition-2-bound}}{\leq} \mathbb{E}[\bar{\theta}_{c,\epsilon_k}^{\eta_k,\tau_k}(x^{k+1})-\bar{\theta}_{c,\epsilon_k}^{\eta_k,\tau_k}(x^k)] +\frac{\alpha}{2 t^{\mathrm{upp}}} \mathbb{E}[\left\| x^{k+1}-x^k \right\|^2] + E_{k} \notag \\
        &\overset{\eqref{slam-convergence-key-proposition-4-eqn1}}{\leq} - \alpha\mathbb{E} \left[ t_k \left\|\bar{G}_{t_k,\epsilon_k}^{\eta_k,\tau_k}(x^k)\right\|^2 \right] + \frac{\alpha}{2t^{\mathrm{upp}}} \mathbb{E} \left[ t_k^2 \left\|\bar{G}_{t_k,\epsilon_k}^{\eta_k,\tau_k}(x^k)\right\|^2 \right] + 2 B_k + E_k \notag \\
        &\leq -\frac{\alpha}{2}\mathbb{E} \left[ t_k \left\|\bar{G}_{t_k,\epsilon_k}^{\eta_k,\tau_k}(x^k)\right\|^2  \right] + 2 B_k + E_k \qquad (t_k^2/t^{\mathrm{upp}} \leq t_k) \notag \\
        &\overset{\eqref{residual-bound}}{\leq} -\frac{\alpha}{4}\mathbb{E} \left[ t_k\left\|G_{t_k}^{\eta_k}(x^k)\right\|^2 \right] + \alpha t^{\mathrm{upp}}v_k + 2B_k + E_k,
    \end{align}
    where $E_{k}$ is defined in Proposition \ref{slam-convergence-key-proposition-2}. We claim that
    \begin{equation}\label{slam-convergence-key-proposition-4-eqn3}
        -\frac{\alpha t_k}{4} \left\|G_{t_k}^{\eta_k}(x^k) \right\|^2 \leq -\frac{\alpha t^{\mathrm{low}}}{8} \left\| G_{t_k}(x^k) \right\|^2 + \frac{\alpha t^{\mathrm{upp}}n^2L_{\nabla\theta}^2}{4} \eta_k^2.
    \end{equation}
    Indeed, by the nonexpansiveness of the projection, we have $\|G_{t_k}(x^k)-G_{t_k}^{\eta_k}(x^k)\| \leq \|\nabla\theta_c(x^k)-\nabla\theta_c^{\eta_k}(x^k)\|$, implying that
    \begin{align*}
        \|G_{t_k}(x^k)\|^2 \leq 2\|G_{t_k}^{\eta_k}(x^k)\|^2 + 2\|\nabla\theta_c(x^k)-\nabla\theta_c^{\eta_k}(x^k)\|^2 \leq 2\|G_{t_k}^{\eta_k}(x^k)\|^2 + 2n^2L_{\nabla\theta}^2\eta_k^2,
    \end{align*}
    where the last inequality is from Lemma \ref{RS-property-grad}-(v). By rearranging the terms and multiplying by $-\alpha t_k/4$, together with $t^{\mathrm{low}} \leq t_{k} \leq t^{\mathrm{upp}}$, we may establish \eqref{slam-convergence-key-proposition-4-eqn3}. By combining \eqref{slam-convergence-key-proposition-4-eqn2} with \eqref{slam-convergence-key-proposition-4-eqn3}, it leads to
    \begin{equation}\label{slam-convergence-key-proposition-4-eqn4}
        \mathbb{E} \left[ \theta_{c}^{\eta_k}(x^{k+1})-\theta_{c}^{\eta_k}(x^k) \right] \leq -\frac{\alpha t^{\mathrm{low}}}{8} \mathbb{E} \left[ \left\| G_{t_k}(x^k) \right\|^2 \right] + \frac{\alpha t^{\mathrm{upp}}n^2L_{\nabla\theta}^2}{4}\eta_k^2 + \alpha t^{\mathrm{upp}}v_k + 2B_k + E_k.
    \end{equation}
    Similar to the proof of \cite[Lemma 5]{wang-shanbhag-xie-2025}, we may show that $\|G_{t^{\mathrm{upp}}}(x^k)\| \leq \|G_{t_k}(x^k)\|$. By Lemma \ref{RS-property-grad}-(iii), we have that
    \begin{equation}\label{slam-convergence-key-proposition-4-eqn5}
        \mathbb{E}[\theta_{c}^{\eta_k}(x^{k+1})-\theta_{c}^{\eta_k}(x^k)] \geq \mathbb{E}[\theta_{c}(x^{k+1})-\theta_{c}(x^k)] - 2L_{\theta}\eta_{k}.
    \end{equation}
    By combining \eqref{slam-convergence-key-proposition-4-eqn4} with \eqref{slam-convergence-key-proposition-4-eqn5}, together with the definition of $E_{k}$ in Proposition \ref{slam-convergence-key-proposition-2}, we complete the proof.
\end{proof}

We define the constants $C_{1}, \cdots, C_{8}$ as
\begin{align*}
    C_1 &\triangleq 2L_\theta, \qquad C_2 \triangleq \frac{\alpha t^{\mathrm{upp}} L_{\nabla \theta}^2n^2}{4}, \qquad C_3 \triangleq 2(Q+3) L^{y}_{\theta}, \\
    C_4 &\triangleq C_N+\frac{ L_\theta^2+16\sqrt{2\pi}L_\theta^2n}{c_0}+32\sqrt{2\pi}L_\theta^2n (8L_{\nabla \theta}( t^{\mathrm{upp}})^2+\alpha t^{\mathrm{upp}}), \\
    C_5 &\triangleq n^2( L^{y}_{\theta})^2(c_0^{-1} + 16 L_{\nabla \theta}( t^{\mathrm{upp}})^2 + 2 \alpha t^{\mathrm{upp}}), \qquad C_{6}\triangleq 4 L_{\nabla \theta} D_X^2.
\end{align*}
Substituting $B_k$ and $v_k$ in \eqref{slam-convergence-key-proposition-4-main} gives
\begin{align}\label{slam-convergence-key-proposition-5-pre}
    \frac{\alpha t^{\mathrm{low}}}{8}\mathbb{E}[\left\| G_{ t^{\mathrm{upp}}}(x^k) \right\|^2] &\leq \mathbb{E}[\theta_{c}(x^k)]-\mathbb{E}[\theta_{c}(x^{k+1})] + C_1 \eta_k + C_2 \eta_k^2 \notag \\
    &+ C_3\sqrt{\epsilon_k} + \frac{C_4}{N_k} + C_5\frac{\epsilon_k}{\eta_k^2} + C_{6}\mathbb{P} \{ t_k\neq t_{k-1} \}.
\end{align}

\begin{theorem}\label{slam-main-result}\em
    Suppose that Assumption $\mathrm{B}$ holds. Consider the sequence generated by \textbf{B-SLAM}. Suppose for a given $K$, we choose $\eta_{k} \triangleq 1/K$, $\epsilon_{k} \triangleq 1/K^{4}$, $N_{k} \triangleq \lceil c_{N} K \rceil$, and $p_{j} = \lceil c_{p} K \rceil$, where $c_{N}\geq 1$ and $c_{p}\in (0, 1)$. Let $\theta^{\inf}_{c} \triangleq \inf_{x\in X} \theta_{c}(x)$. Then the following hold. \\
    (i) We have that
    \begin{align*}
        \mathbb{E} \left[ \left\| G_{t^{\mathrm{upp}}}(x^{R_{K}}) \right\|^{2} \right] &\leq \frac{8}{\alpha t^{\mathrm{low}} K} \left[ \theta_{c}(x^{0}) - \theta^{\inf}_{c} + C_{1} + C_{1} + \frac{C_{4}}{c_{N}} + \frac{C_{2}+C_{3}+C_{5}}{K} + C_{6}Q \left\lceil \frac{1}{c_{p}} \right\rceil \right],
    \end{align*}
    where $Q$ is defined in \eqref{Q-definition}. \\
    (ii) The iteration complexity $K_{\epsilon}$ and the sample complexity $S_{\epsilon} = \sum_{k=0}^{K_{\epsilon}-1} N_{k}$ are $\mathcal{O}(\epsilon^{-2})$ and $\mathcal{O}(\epsilon^{-4})$, respectively.
\end{theorem}
\begin{proof}
    \noindent (i) By summing \eqref{slam-convergence-key-proposition-5-pre} over $k = 0, \ldots, K-1$ and telescoping, we obtain
    \begin{align}\label{slam-main-result-eqn1}
        &\frac{\alpha t^{\mathrm{low}}}{8} \sum_{k=0}^{K-1} \mathbb{E} \left[ \left\| G_{t^{\mathrm{upp}}}(x^{R_{K}}) \right\|^{2} \right] \notag \\
        &\leq \theta_{c}(x^{0}) - \theta^{\inf}_{c} + \sum_{k=0}^{K-1} \left( C_{1} \eta_{k} + C_{2} \eta^{2}_{k} + C_{3} \sqrt{\epsilon_{k}} + \frac{C_{4}}{N_{k}} + C_{5}\frac{\epsilon_{k}}{\eta^{2}_{k}} \right) + C_{6} \sum_{k=0}^{K-1} \mathbb{P} \{ t_k\neq t_{k-1} \}.
    \end{align}
    By Lemma \ref{SLAM-bound-TK-lemma}, we have
    \begin{align}\label{slam-main-result-eqn2}
        \sum_{k=0}^{K-1} \mathbb{P}\{ t_{k}\neq t_{k-1} \} &= \sum_{k=0}^{K-1} \mathbb{E} [ \, \mathbf{1}_{\{ t_{k}\neq t_{k-1} \}} \, ] = \mathbb{E} \left[ \, \sum_{k=0}^{K-1} \mathbf{1}_{\{ t_{k}\neq t_{k-1} \}} \, \right] = \mathbb{E} [ \, T(K) \, ] \notag \\
        &\leq (J(K)+1) \, Q \leq \left\lceil \frac{1}{c_{p}} \right\rceil \, Q,
    \end{align}
    where the derivation of the bound $J(K) + 1 \leq \lceil 1/c_{p} \rceil$ can be found in \cite[Theorem 1]{wang-shanbhag-xie-2025}. By our choices of $\eta_{k}$, $\epsilon_{k}$, and $N_{k}$, combining \eqref{slam-main-result-eqn1} with \eqref{slam-main-result-eqn2}, the final result follows from $\mathbb{E} \left[ \left\| G_{t^{\mathrm{upp}}}(x^{R_{K}}) \right\|^{2} \right] = \frac{1}{K} \sum_{k=0}^{K-1} \mathbb{E} \left[ \left\| G_{t^{\mathrm{upp}}}(x^k) \right\|^{2} \right]$. That is, we may establish
    \begin{align*}
        \mathbb{E} \left[ \left\| G_{t^{\mathrm{upp}}}(x^{R_{K}}) \right\|^{2} \right] &\leq \frac{8}{\alpha t^{\mathrm{low}} K} \left[ \theta_{c}(x^{0}) - \theta^{\inf}_{c} + C_{1} + C_{1} + \frac{C_{4}}{c_{N}} + \frac{C_{2}+C_{3}+C_{5}}{K} + C_{6}Q \left\lceil \frac{1}{c_{p}} \right\rceil \right],
    \end{align*}
    which completes the proof.
    
    (ii) It follows from (i) that $\mathbb{E}[\|G_{t^{\mathrm{upp}}}(x^{R_{K}})\|^{2}] = \mathcal{O}(1/K)$. Therefore, the required iteration complexity is $K_{\epsilon} = \mathcal{O}(\epsilon^{-2})$ and the sample complexity is $S_{\epsilon} = \sum_{k=0}^{K_{\epsilon}-1} N_{k} = \mathcal{O}(\epsilon^{-4})$.
\end{proof}

\begin{remark}\em
    Under the constant-stepsize choice $t^{\mathrm{low}} = t^{\mathrm{upp}} = \gamma/L_{\nabla\theta}$, our iteration and expected outer sample complexities reduce to $\mathcal{O}(L_{\nabla\theta}\epsilon^{-2})$ and $\mathcal{O}(L_{\nabla\theta}^{2}\epsilon^{-4})$, respectively, with the remaining problem constants fixed. Our result matches the dependence on $\epsilon$ in Ghadimi et al. \cite{ghadimi-lan-zhang-2016}, but incurs an additional factor of $L_{\nabla\theta}$ due to our choice $N_k = \mathcal{O}(K)$. It is worth noting that their stochastic gradients and zeroth-order estimators are unbiased, whereas our estimators additionally incur bias from inexact lower-level solutions. $\hfill \Box$
\end{remark}

\section{Numerical Experiments}\label{Sec-5}

In this section, we apply our \textbf{SEG-LS} and \textbf{B-SLAM} schemes on three distinct examples.

\subsection{Stochastic fractional convex quadratic problems}\label{Sec-5.1}

In this subsection, we compare two \textbf{SEG} schemes with the stochastic forward-backward-forward (\textbf{SFBF}) method \cite{bot-mertikopoulos-staudigl-vuong-2021} on a stochastic fractional problem that corresponds to a pseudomonotone $\mathrm{SVI}\:(X, F)$. To this end, we consider
\begin{equation*}
    \min_{x\in X} \;\; f(x) \, \triangleq \, \mathbb{E}\left[\Tilde{g}(x, \bxi)\right]/h(x),
\end{equation*}
where $g(x) \triangleq \mathbb{E}[\Tilde{g}(x, \bxi)]$ and $h(x)$ are strictly positive convex quadratic and linear functions, respectively. In this example, we set $\Tilde{g}(x, \xi) \triangleq \frac{1}{2} x^{\top} Q(\xi) x + c(\xi)^{\top} x + q(\xi)$ and $h(x) = r^{\top}x + t$. The feasible region of this problem is a box of the form $X \triangleq \{ x\in \mathbb{R}^{n} \,\vert\, 0 < x_{\mathrm{min}} \leq x_{i} \leq x_{\mathrm{max}}, ~ \forall i\in [n] \}$.

Before we proceed, let us first examine the Lipschitz continuity constant $L$ of $F(x) = \nabla f(x)$ over $X$ in the following lemma. We may observe that a smaller $h_{\mathrm{low}}$ leads to a larger Lipschitz continuity constant, which can be achieved by choosing a smaller $x_{\min}$ since $h(x)$ is an affine function.

\begin{lemma}
    Let $L_g$, $L_h$, $L_{\nabla g}$, and $L_{\nabla h}$ denote the Lipschitz constants of $g$, $h$, $\nabla g$, and $\nabla h$, respectively. Suppose that there exist constants $h_{\mathrm{low}}$, $h_{\mathrm{upp}}$, $g_{\mathrm{low}}$, $g_{\mathrm{upp}}$, $M_{\nabla h}$, and $M_{\nabla g}$ such that $h_{\mathrm{low}} \leq h(x) \leq h_{\mathrm{upp}}$, $g_{\mathrm{low}} \leq g(x) \leq g_{\mathrm{upp}}$, $\|\nabla h(x)\| \leq M_{\nabla h}$, and $\|\nabla g(x)\| \leq M_{\nabla g}$ hold for all $x \in X$. Then we have that
    \begin{equation*}
        \| \nabla f(x) - \nabla f(y) \| \leq \left( \frac{L_{\nabla g}}{h_{\mathrm{low}}} + \frac{M_{\nabla g}L_{h} + g_{\mathrm{upp}}L_{\nabla h} + M_{\nabla h}L_{g}}{h^{2}_{\mathrm{low}}} + \frac{2g_{\mathrm{upp}}M_{\nabla h}L_{h}}{h^{3}_{\mathrm{low}}} \right) \| x - y \|.
    \end{equation*}
\end{lemma}
\begin{proof}
    Consider the gradient map $\nabla f$ where $h$ is a positive function. Then we have that
    \allowdisplaybreaks
    \begin{align*}
        &\| \nabla f(x) - \nabla f(y) \| \leq \left\| \frac{\nabla g(x)}{h(x)} - \frac{g(x)\nabla h(x)}{h^{2}(x)} - \frac{\nabla g(y)}{h(y)} + \frac{g(y)\nabla h(y)}{h^{2}(y)} \right\| \\
        &\leq \left\| \frac{\nabla g(x)}{h(x)} - \frac{\nabla g(y)}{h(x)} \right\| + \left\| \frac{\nabla g(y)}{h(x)} - \frac{\nabla g(y)}{h(y)} \right\| \\
        &+ \left\| \frac{g(x)\nabla h(x)}{h^{2}(x)} - \frac{g(x)\nabla h(y)}{h^{2}(x)} \right\| + \left\| \frac{g(x)\nabla h(y)}{h^{2}(x)} - \frac{g(y)\nabla h(y)}{h^{2}(x)} \right\| + \left\| \frac{g(y)\nabla h(y)}{h^{2}(x)} - \frac{g(y)\nabla h(y)}{h^{2}(y)} \right\| \\
        &\leq \left( \frac{L_{\nabla g}}{h_{\mathrm{low}}} + \frac{M_{\nabla g}L_{h} + g_{\mathrm{upp}}L_{\nabla h} + M_{\nabla h}L_{g}}{h^{2}_{\mathrm{low}}} + \frac{2g_{\mathrm{upp}}M_{\nabla h}L_{h}}{h^{3}_{\mathrm{low}}} \right) \| x - y \|,
    \end{align*}
    which completes the proof.
\end{proof}

(I) \emph{Problem parameters.} In this example, we consider $n = 4$, $Q(\xi) = M(\xi) + 50 \mathbf{I}_{n}$, where $M(\xi)\in \mathbb{R}^{n\times n}$ is a random diagonal matrix with i.i.d. diagonal entries drawn from $\mathcal{N} (0, 1)$, $c(\xi) \in \mathbb{R}^{n}$ is a random vector with i.i.d. entries drawn from $\mathrm{U}\:[0, 2]$, $q(\xi) \sim \mathrm{U}\:[-1, 1]$, $r = \mathbf{1}_{n}$, and $t = 0$. Suppose that $x_{\mathrm{max}} = 14$. Table \ref{SEG-pseudomonotone-table} shows that different values of $x_{\mathrm{min}}$ result in different values of $L$.

(II) \emph{Implementation.} In this experiment, we set $c_{N} = 0.003$ hence $N_{k} = \lceil c_{N} K \rceil = \lceil 0.003 K \rceil$. With regard to the backtracking linesearch procedure, we set the parameters $\alpha = 0.1$ and $\beta = 0.9$. We compare three SEG schemes under different stepsizes and cycles settings.

(1) \texttt{SEG-LS}: \textbf{SEG-LS} with parameters $t^{\mathrm{upp}} = 1$, $c_{p} = 0.001$ hence $\{ p_{j} \} = \lceil K/1000 \rceil$.

(2) \texttt{SEG-CONST}: SEG scheme with $\mathcal{O}(\tfrac{1}{L})$ constant stepsize. By \cite[Theorem 3.25]{iusem-jofre-oliveira-thompson-2017}, the stepsize should be chosen as $t_k \equiv t \in \left(0, \tfrac{1}{\sqrt{6}L}\right)$. Here, we set the stepsizes as $t_k \equiv \tfrac{1}{6L}$.

(3) \texttt{SFBF}: The \textbf{SFBF} scheme \cite{bot-mertikopoulos-staudigl-vuong-2021} produces a sequence $\{ (x^{k}, y^{k}) \}_{k=0}^{K}$ defined as
\begin{equation}
    \begin{aligned}
        y^{k} &\triangleq \Pi_{X} [ x^{k} - t_{k} \bar{F}_{k}(x^{k}, \xi^{k}) ], \\
        x^{k+1} &\triangleq y^{k} + t_{k} (  \bar{F}_{k}(x^{k}, \xi^{k}) -  \bar{F}_{k}(y^{k}, \zeta^{k}) ),
    \end{aligned} \tag{SFBF}
\end{equation}
where $\{ t_{k} \}_{k=0}^{K}$ is a positive sequence of stepsizes that is bounded by some reciprocal of the Lipschitz constant $L$ and here we set $t_k \equiv \tfrac{1}{6L}$. The main computational advantage of \textbf{SFBF} over \textbf{SEG} is that it only necessitates a single projection step at each iteration. The batch size sequence is required to satisfy $\sum_{k=0}^{\infty} \frac{1}{N_k} < \infty$; see \cite[Assumption 6]{bot-mertikopoulos-staudigl-vuong-2021}.

We choose the initialization as $x^{0} = 7.5\mathbf{e}_{4}$ and set the total iteration budget to $K = 2\texttt{e}4$. When $x_{\min} = 0.1$, the expected square residuals and the first $100$ stepsizes are shown in Figures \(\ref{SEG-psdmon-residuals}\) and \(\ref{SEG-psdmon-stepsizes}\), respectively, averaged over $20$ sample paths.

\begin{figure}
    \centering
    \begin{subfigure}{0.45\textwidth}
        \centering
        \includegraphics[width=\textwidth]{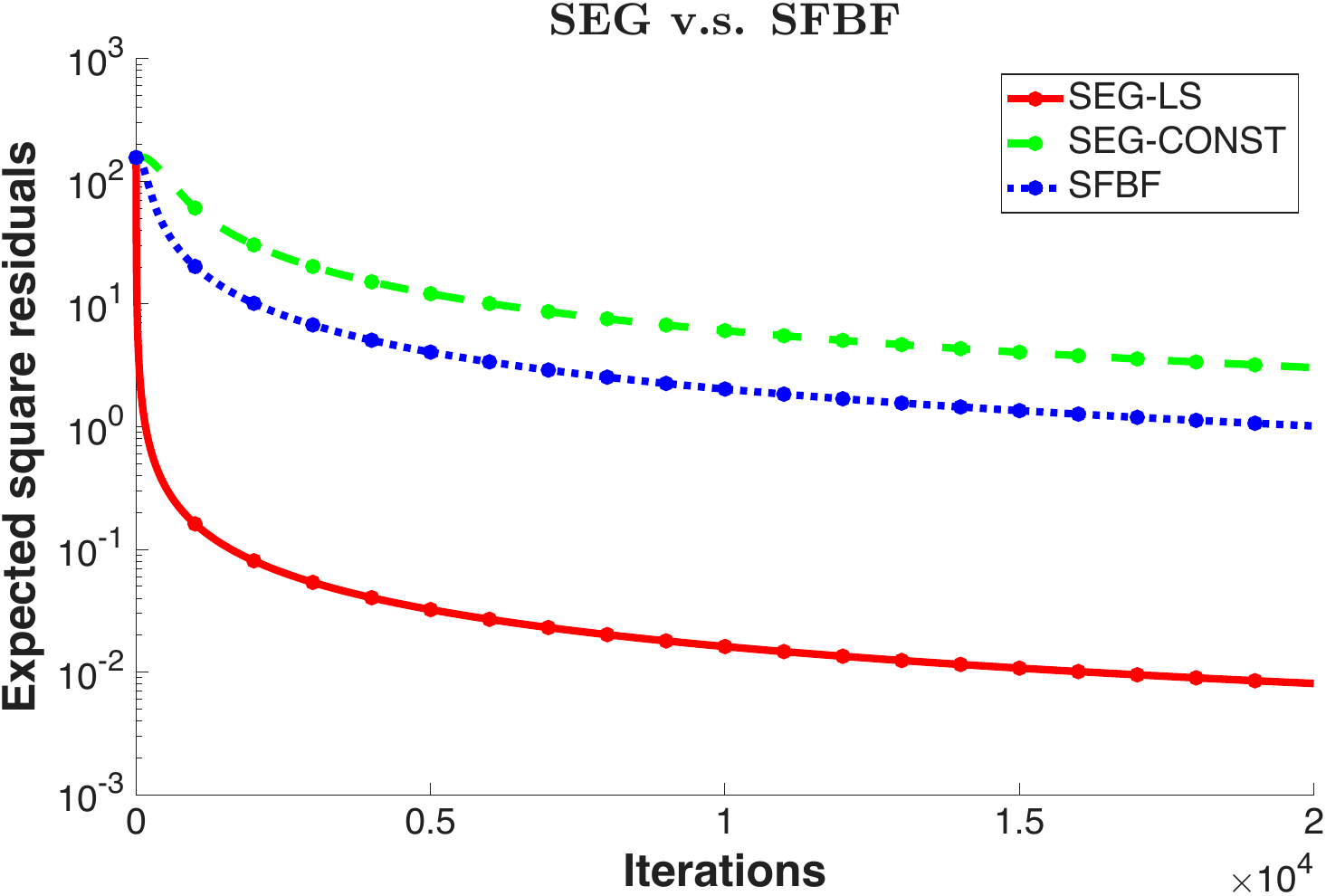}
        \caption{The expected square residuals.}
        \label{SEG-psdmon-residuals}
    \end{subfigure}
    \hfill
    \begin{subfigure}{0.45\textwidth}
        \centering
        \includegraphics[width=\textwidth]{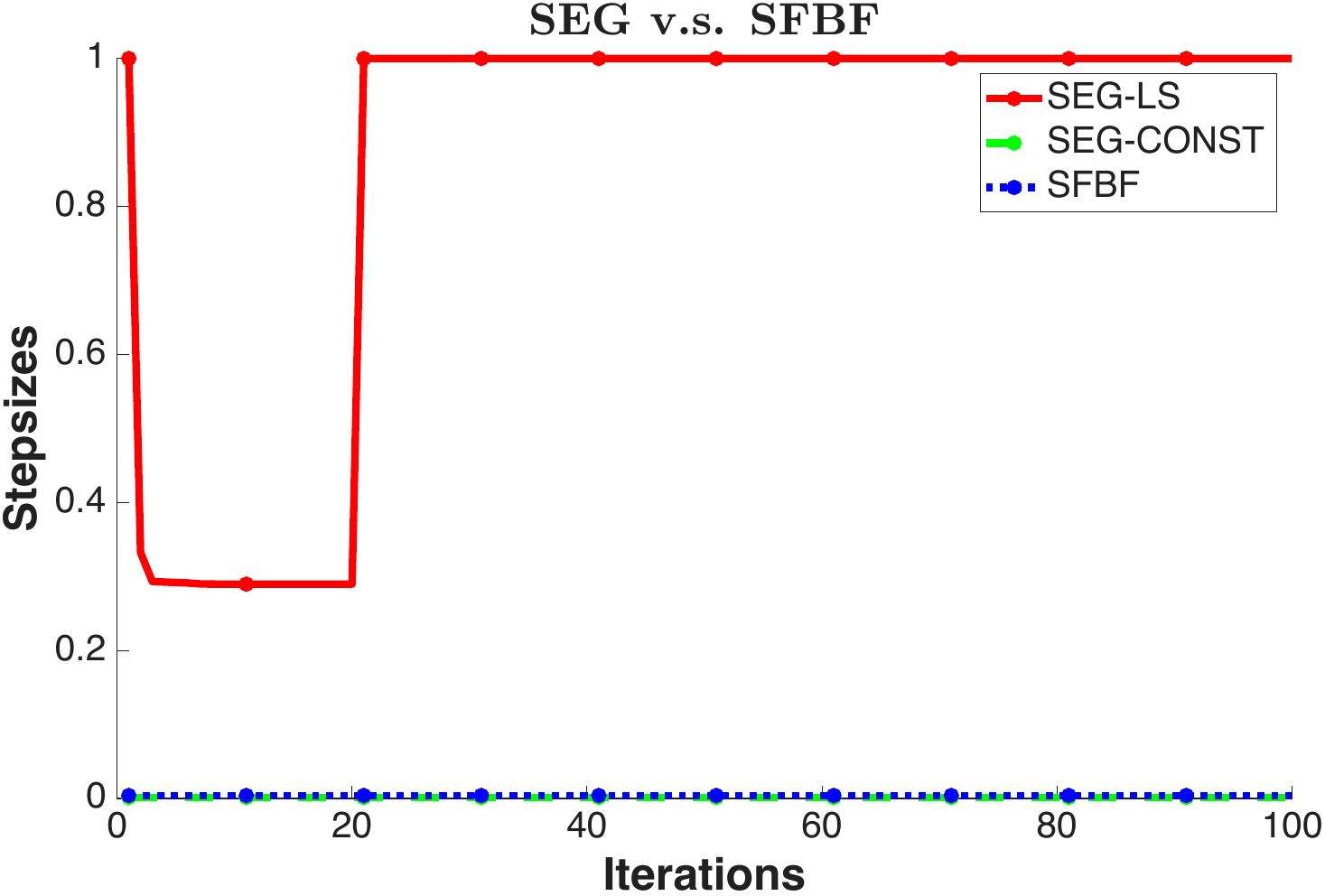}
        \caption{The first $100$ stepsizes.}
        \label{SEG-psdmon-stepsizes}
    \end{subfigure}
    \caption{Two \textbf{SEG} schemes v.s. \textbf{SFBF} for a stochastic fractional problem when $x_{\min} = 0.1$.}
    \label{SEG-psdmon}
\end{figure}

\begin{table}[H]
\centering
{\renewcommand{\arraystretch}{1.0}
\begin{tabular}{@{\hspace{15pt}}c@{\hspace{30pt}}c@{\hspace{30pt}}c@{\hspace{30pt}}c@{\hspace{30pt}}c@{\hspace{20pt}}}
\toprule
\toprule
$x_{\min}$ & $L$ & SEG-LS & SEG-CONST & SFBF \\
\midrule
\midrule
$1$ & $12.5$ & 7.83\text{e}-3 & 2.23\text{e}-1 & 7.29\text{e}-2 \\
$0.75$ & $16.67$ & 7.86\text{e}-3 & 3.27\text{e}-1 & 1.03\text{e}-1 \\
$0.5$ & $25$ & 7.93\text{e}-3 & 5.35\text{e}-1 & 1.51\text{e}-1 \\
$0.25$ & $50$ & 8.01\text{e}-3 & 1.16\text{e}0 & 3.39\text{e}-1 \\
$0.1$ & $125$ & 8.08\text{e}-3 & 3.04\text{e}0 & 1.01\text{e}0 \\
\bottomrule
\bottomrule
\end{tabular}
}
\caption{The expected squared residuals for different values of $x_{\min}$ $(K = 2\texttt{e}4)$.}
\label{SEG-pseudomonotone-table}
\end{table}

Table \ref{SEG-pseudomonotone-table} demonstrate that \texttt{SEG-LS} achieves better performance than \texttt{SEG-CONST} and \texttt{SFBF}, with a more pronounced advantage for sufficiently large values of $L$. Such an advantage stems from its adaptive stepsize selection.

\subsection{Stochastic monotone Nash-Cournot games}\label{Sec-5.2}

In this subsection, we compare two \textbf{SEG} schemes with the relaxed inertial stochastic forward-backward-forward (\textbf{RISFBF}) method \cite{cui-shanbhag-staudigl-vuong-2025} on a stochastic monotone Nash-Cournot game. To this end, we consider the following stochastic $N$-player game
\begin{equation*}
    \min_{x_{i} \in X_{i}} \;\; f_{i}(x_{i}, x_{-i}) \, \triangleq \, \mathbb{E}[\Tilde{c}_{i}(\bxi)]g_{i}(x_{i}) - \mathbb{E}[\Tilde{p}(\bar{x}, \bxi)]x_{i}, \quad \forall i\in [N],
\end{equation*}
where $\mathbb{E}[\tilde{c}_{i}(\boldsymbol{\xi})]g_i(x_i)$ denotes the convex private cost function of player $i$, and $\mathbb{E}[\tilde{p}(\bar{x},\boldsymbol{\xi})]$ denotes the expected linear inverse demand function, with $\tilde{p}(\bar{x}, \xi) \triangleq a(\xi) - b(\xi)\bar{x}$ and $\bar{x} = \sum_{i=1}^{N} x_i$. By computation, we may obtain the concatenation gradient map
\begin{equation*}
    F(x) = ( c_{i} \nabla_{x_i} g_i(x_i))_{i=1}^N - a \mathbf{1} + b(I+ {\bf 1} {\bf 1}^\top) x,
\end{equation*}
where $c_{i} \triangleq \mathbb{E}[\Tilde{c}_{i}(\bxi)]$, $a \triangleq \mathbb{E}[a(\bxi)]$, and $b \triangleq \mathbb{E}[b(\bxi)]$. We specify the parameters below.

(I) \emph{Problem parameters.} In this example, we assume that $N = 6$, $X_{i} = [0, 15]$ for any $i\in [N]$, $\xi \sim \mathrm{U}\:[-4, 4]$, $\Tilde{c}_{i}(\xi) = 1 + \xi/20$ for any $i\in [N]$, $a(\xi) = 10 + \xi/5$ and $b(\xi) = 1 + \xi/10$. Let $g_{i}$ be an increasing convex function given by $g_{i}(x_{i}) \triangleq 3x_{i} + 2\sqrt{(x_{i}-7.5)^{2} + \mu^{-2}}$ where $\mu > 0$ is a parameter. By computation, we may obtain the Lipschitz continuity constant $L(\mu) = 2\mu + 0.7$. Table \ref{SEG-monotone-table} reports the Lipschitz continuity constants $L$ corresponding to different choices of $\mu$.

(II) \emph{Implementation.} In this experiment, we set $c_{N} = 0.003$ hence $N_{k} = \lceil c_{N} K \rceil = \lceil 0.003 K \rceil$. With regard to the backtracking linesearch procedure, we set the parameters $\alpha = 0.1$ and $\beta = 0.9$. We compare two \textbf{SEG} schemes with different stepsizes and \textbf{RISFBF} scheme.

(1) \texttt{SEG-LS}: \textbf{SEG-LS} with parameters $t^{\mathrm{upp}} = 1$, $c_{p} = 0.001$ hence $\{ p_{j} \} = \lceil K/1000 \rceil$.

(2) \texttt{SEG-CONST}: SEG scheme with $\mathcal{O}(\frac{1}{L})$ constant stepsize. Here we set the stepsizes as $t_k \equiv \frac{1}{4L}$.

(3) \texttt{RISFBF}: The \textbf{RISFBF} scheme produces a sequence $\{ (x^{k}, y^{k}, z^{k}) \}_{k=0}^{K}$ defined as
\begin{equation}
    \begin{aligned}
        z^{k} & \, \triangleq \, x^{k} + \alpha_{k} (x^{k} - x^{k-1}), \\
        y^{k} & \, \triangleq \, \Pi_{X} [ z^{k} - t_{k} \bar{F}_{k}(z^{k}, \xi^{k}) ], \\
        x^{k+1} & \, \triangleq \, (1-\rho_{k}) z^{k} + \rho_{k} [ y^{k} + t_{k} ( \bar{F}_{k}(z^{k}, \xi^{k}) - \bar{F}_{k}(y^{k}, \zeta^{k})) ],
    \end{aligned} \tag{RISFBF}
\end{equation}
where $\{ \alpha_{k} \}_{k=0}^{K}$ is a nonnegative inertia sequence, $\{ t_{k} \}_{k=0}^{K}$ is a positive sequence of stepsizes that is bounded by some reciprocal of the Lipschitz constant $L$, and $\{ \rho_{k} \}_{k=0}^{K}$ is a nonnegative relaxation sequence. As with the \textbf{SFBF} method, the \textbf{RISFBF} method requires only one projection step per iteration. We continue to use the mini-batch size $N_{k} = \lceil c_N K\rceil$ and adopt parameters similar to those in \cite{cui-shanbhag-staudigl-vuong-2025}. Specifically, we set $t_{k} \equiv \frac{1}{4L}$, $\alpha_{k} = \alpha_{0}(1-\frac{1}{k+1})$ where $\alpha_{0} = 0.2$, and $\rho_{k} = \frac{3(1-\alpha_{0})^{2}}{2(2\alpha^{2}_{k}-\alpha_{k}+1)(1+Lt_{k})}$.

We choose the initialization as $x^{0} = 4\mathbf{e}_{6}$ and set the total iteration budget to $K = 2\texttt{e}4$. When $\mu = 500$, the expected square residuals and the first $100$ stepsizes are shown in Figure \ref{SEG-RISFBF-residuals} and Figure \ref{SEG-RISFBF-stepsizes}, respectively, averaged over $20$ sample paths.

\begin{figure}
    \centering
    \begin{subfigure}{0.45\textwidth}
        \centering
        \includegraphics[width=\textwidth]{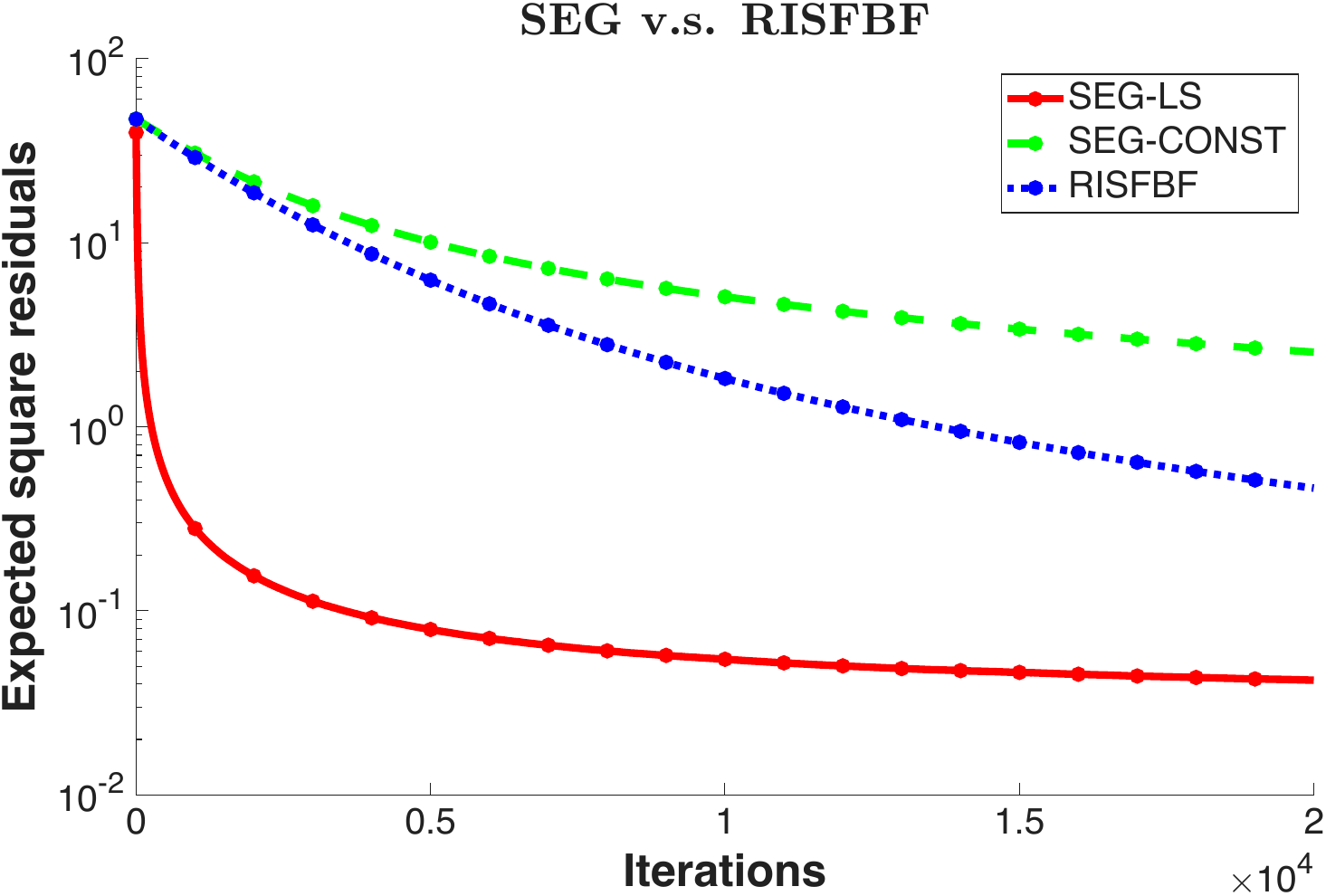}
        \caption{The expected square residuals.}
        \label{SEG-RISFBF-residuals}
    \end{subfigure}
    \hfill
    \begin{subfigure}{0.45\textwidth}
        \centering
        \includegraphics[width=\textwidth]{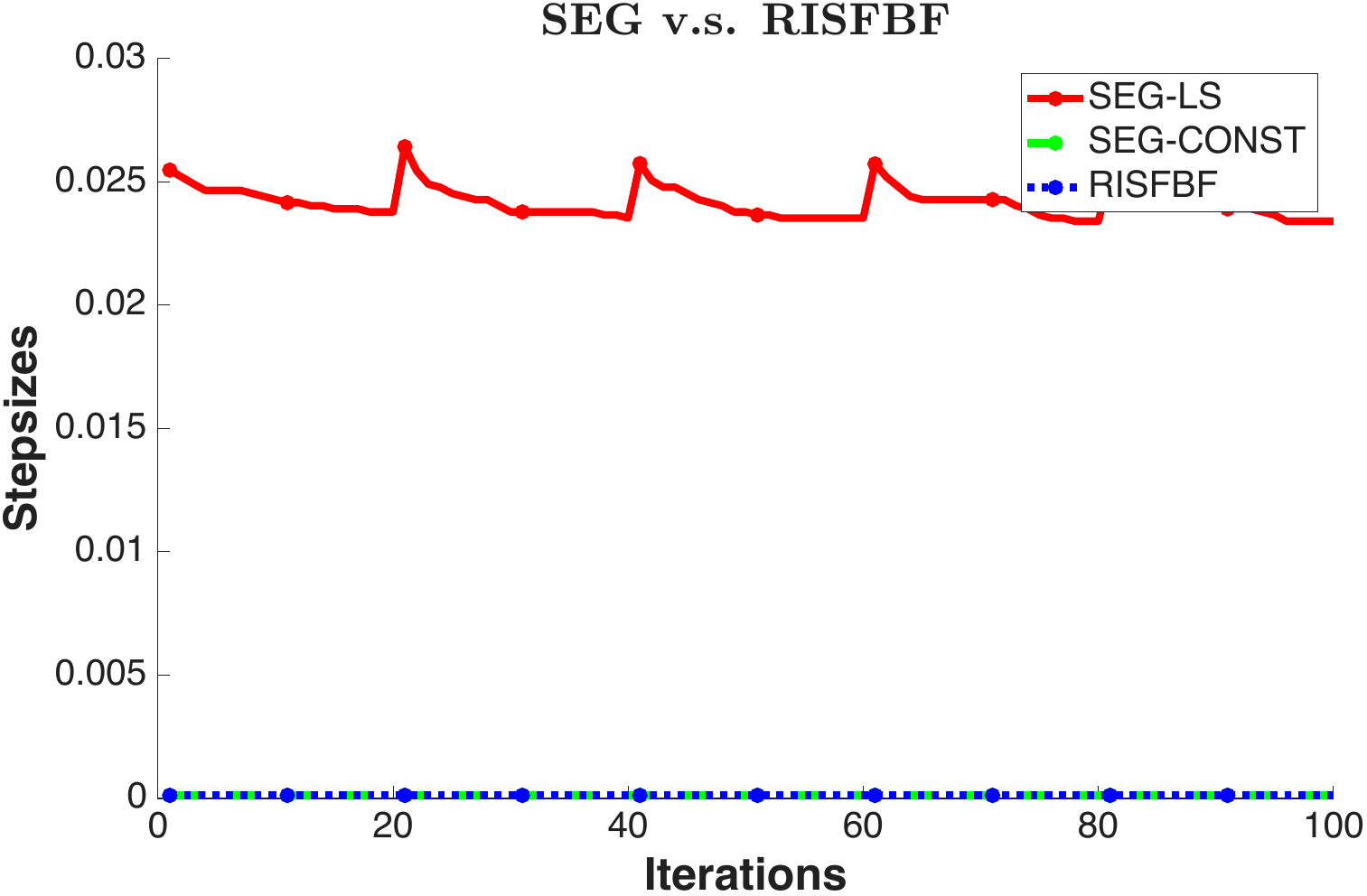}
        \caption{The first $100$ stepsizes.}
        \label{SEG-RISFBF-stepsizes}
    \end{subfigure}
    \caption{Two \textbf{SEG} schemes v.s. \textbf{RISFBF} for a stochastic monotone Cournot game when $\mu = 1000$.}
    \label{SEG-mon}
\end{figure}

\begin{table}[t]
\centering
{\renewcommand{\arraystretch}{1.0}
\begin{tabular}{@{\hspace{15pt}}c@{\hspace{30pt}}c@{\hspace{30pt}}c@{\hspace{30pt}}c@{\hspace{30pt}}c@{\hspace{20pt}}}
\toprule
\toprule
$\mu$ & $L$ & SEG-LS & SEG-CONST & RISFBF \\
\midrule
\midrule
$10$ & $20.7$ & 3.87\text{e}-2 & 4.28\text{e}-2 & 7.89\text{e}-5 \\
$50$ & $100.7$ & 4.19\text{e}-2 & 1.29\text{e}-1 & 1.17\text{e}-3 \\
$100$ & $200.7$ & 4.81\text{e}-2 & 2.56\text{e}-1 & 4.77\text{e}-3 \\ 
$500$ & $1000.7$ & 4.87\text{e}-2 & 1.27\text{e}0 & 1.17\text{e}-1 \\
$1000$ & $2000.7$ & 4.95\text{e}-2 & 2.54\text{e}0 & 4.65\text{e}-1 \\
\bottomrule
\bottomrule
\end{tabular}
}
\caption{The expected squared residuals for different values of $\mu$ $(K = 2\texttt{e}4)$.}
\label{SEG-monotone-table}
\end{table}

We may observe from Figure \ref{SEG-mon} and Table \ref{SEG-monotone-table} that \texttt{SEG-LS} outperforms \texttt{SEG-CONST} and \texttt{RISFBF}, particularly when $L$ is sufficiently large. This improvement can be attributed to the adaptive stepsize update mechanism of \texttt{SEG-LS}. However, when $L$ is small, \texttt{SEG-LS} and \texttt{SEG-CONST} have the similar performances, whereas \texttt{RISFBF} exhibits the best convergence behavior.

\subsection{Strictly copositive congestion games under uncertainty}\label{Sec-5.3}

In this subsection, we consider a stochastic $N$-player strictly copositive congestion game from \cite[Section 6.3]{xiao-shanbhag-2025-GR}, where the $i$th player solves the convex optimization problem
\begin{align*}
    \max_{x_i \in X_i} \;\; f_{i}(x_{i}, x_{-i}) \, \triangleq \, U_{i}(x_{i}) - x_{i} \sum_{j \neq i} g_{j}(x_{j}), \quad \forall i\in [N].
\end{align*}
The corresponding catenation map $F(x) \triangleq (\nabla_{x_{i}} f_{i}(x_{i}, x_{-i}))_{i=1}^{N}$ may be potentially non-monotone. 

(I) \emph{Problem parameters.} We consider the same setting as \cite{xiao-shanbhag-2025-GR}. Suppose there are $N = 8$ players. Assume that $X_{i} = [0, 20]$ for every $i\in [N]$,  $U_{i}(x_{i}) = -\frac{1}{2}x^{2}_{i} + \mathbb{E}[80+\bxi]x_{i}$ where $\bxi \sim \mathrm{U} \: [-4, 4]$, and $g_{j}(x_{j}) \triangleq \tfrac{1}{100}(x_{j}-10)^{2} + 10$ for each $j\neq i$. By \cite{xiao-shanbhag-2025-GR}, the strict copositivity required in Lemma \ref{copositivity_2} holds at the unique NE $x^{\ast} = (10, \cdots, 10)^{\top}$.

(II) \emph{Implementation.} In this experiment, we compare the performances of our \textbf{B-SLAM} scheme with \textbf{SAMGR} scheme in \cite{xiao-shanbhag-2025-GR} with three different constant stepsizes. With regard to the backtracking linesearch procedure in both schemes, we set the parameters $\alpha = 0.1$ and $\beta = 0.9$. The algorithmic parameters are specified as follows.

(1) \texttt{SAMGR}: \textbf{SAMGR} with three different constant stepsizes 
$\gamma^{0} \in \{0.10,\,0.15,\,0.20\}$. All other parameter settings are the same as those in \cite{xiao-shanbhag-2025-GR}.

(2) \texttt{B-SLAM}: \textbf{B-SLAM} with parameters $t^{\mathrm{upp}} = 1$, $c_{p} = 0.02$ hence $\{ m_{k} \} = \lceil K/50 \rceil$.

\begin{figure}[h]
    \centering
    \begin{subfigure}{0.45\textwidth}
        \centering
        \includegraphics[width=\textwidth]{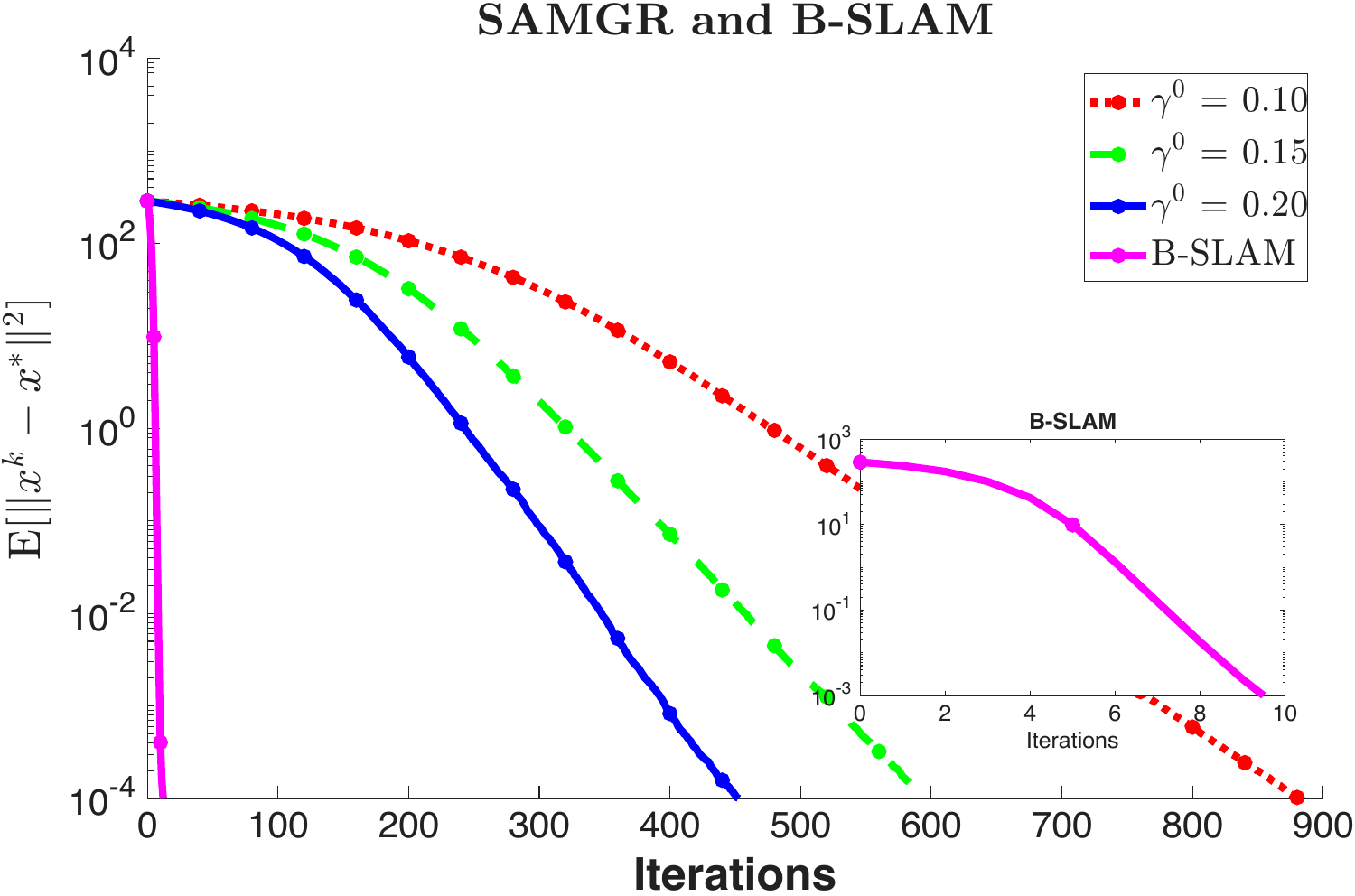}
        \caption{\textbf{SAMGR} and \textbf{B-SLAM}}
        \label{SAMGR-SLAM}
    \end{subfigure}
    \hfill
    \begin{subfigure}{0.45\textwidth}
        \centering
        \includegraphics[width=\textwidth]{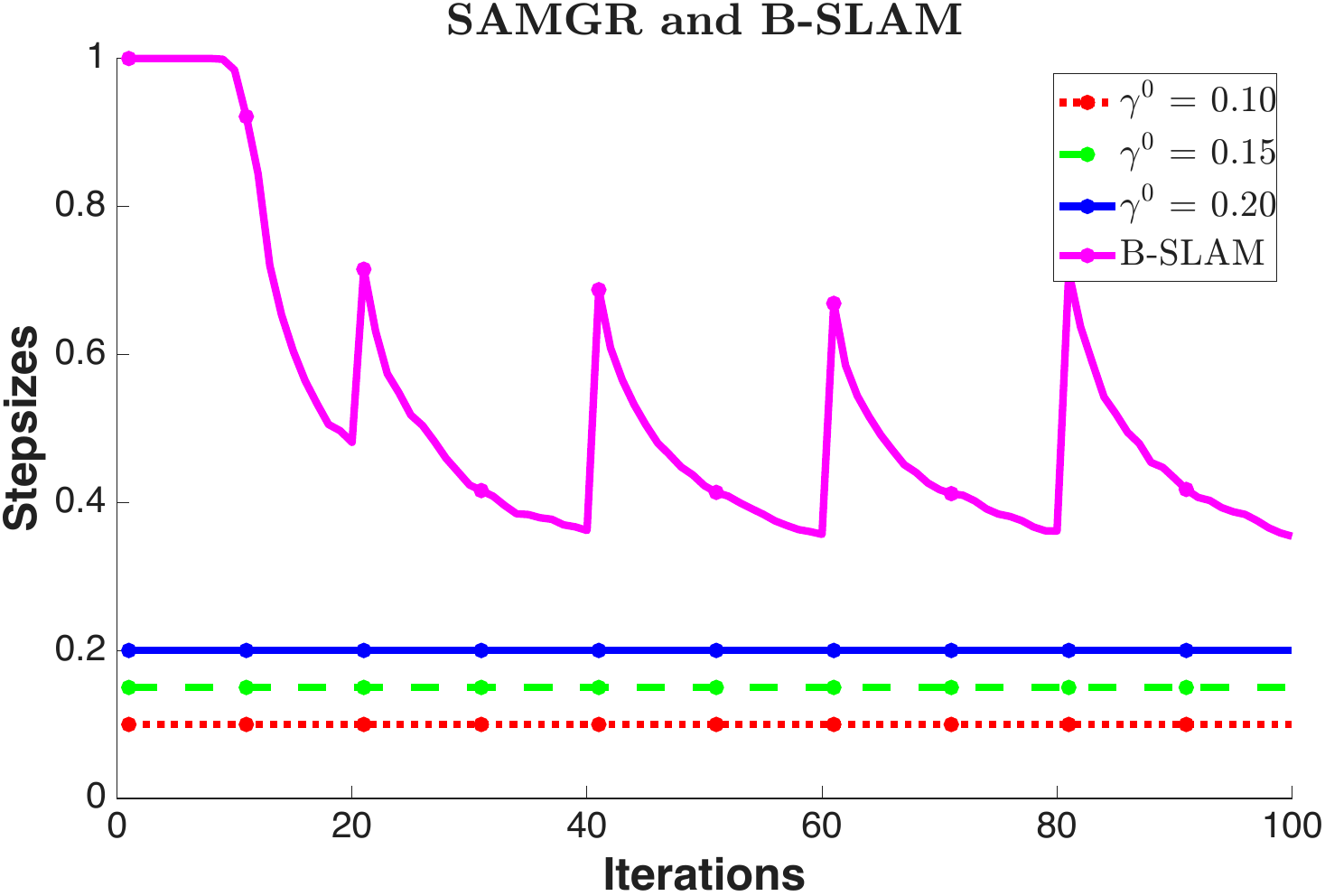}
        \caption{The first $100$ stepsizes.}
        \label{SAMGR-SLAM-stepsizes}
    \end{subfigure}
    \caption{\textbf{SAMGR} and \textbf{B-SLAM} for the stochastic network congestion problem.}
    \label{SAMGR-BSLAM-comparison-figure}
\end{figure}

For both schemes, we initialize at $x^{0} = 4\mathbf{e}_{8}$ and set the total iteration budget to $K = 1\texttt{e}3$. Due to the unavailability of $\nabla \theta_{c}(x)$, we use $\mathbb{E}[\| x^{k} - x^{\ast} \|^{2}]$ as the convergence measure. The expected square residuals and stepsize sequences at the first $100$ iterations are shown in Figure \ref{SAMGR-SLAM} and Figure \ref{SAMGR-SLAM-stepsizes}, respectively, averaged over $20$ sample paths.

As illustrated in Figure \ref{SAMGR-SLAM}, our \textbf{B-SLAM} significantly outperforms \textbf{SAMGR}. It can be observed that our \textbf{B-SLAM} achieves an accuracy of $10^{-3}$ within $10$ iterations, which is mainly attributed to the adaptive stepsize nature, as shown in Figure \ref{SAMGR-SLAM-stepsizes}. It can be observed from Figure \ref{SAMGR-SLAM-stepsizes} that the final accepted stepsize is larger than the lower bound $t^{\mathrm{low}}$ at the 20th iteration, but equals $t^{\mathrm{low}}$ at the 40th, 60th, 80th, and 100th iterations.

\section{Conclusion}\label{Sec-6}

Extant stochastic approximation schemes for solving SVI problems are subject to several limitations. In particular, such schemes often rely on constant or diminishing steplength sequences that depend on prior knowledge of the Lipschitz constant of the map, potentially resulting in overly conservative stepsizes and degraded empirical performance. In this work, we introduce the periodically restarted linesearch and address these challenges by developing stochastic extragradient  and gradient schemes equipped with linesearch for solving SVI problems. Specifically, we propose a parameter-free \textbf{SEG-LS} scheme, wherein the resulting stepsize sequence is adaptive, non-monotone, and does not require knowledge of the Lipschitz constant. We further consider a stochastic biased linesearch procedure and address non-monotone SVIs under strict copositivity by proposing a biased \textbf{SLAM} (\textbf{B-SLAM}) scheme, which also ensures provable convergence. Preliminary numerical results support the theoretical findings and demonstrate the practical promise of the proposed schemes. Future work will focus on extending these linesearch-enabled approaches to broader classes of stochastic equilibrium problems. 

\bibliographystyle{plain}
\bibliography{references}
 
\end{document}